%% file: Effective_IBC_DB_261029.tex
\documentclass[10pt,reqno]{amsart}
\usepackage{amssymb,mathtools,amsthm,mathrsfs}
\usepackage{color}
\usepackage[dvipsnames]{xcolor}

\usepackage{stmaryrd}
\usepackage{bm}
\usepackage[normalem]{ulem}
\usepackage{amscd}
\usepackage{tikz-cd}
\usepackage{graphicx}
\usepackage{epstopdf} %This is to transfer .eps figure to .pdf figure; please compile your article using PDFLaTex or PDFTeXify.
\usepackage[shortlabels]{enumitem}
\usepackage{booktabs,dcolumn}
\usepackage[tableposition=top]{caption}
\usepackage[misc]{ifsym}
\usepackage{comment}
\usepackage[superscript]{cite} % M3AS: superscript numbered citations
\usepackage[hidelinks]{hyperref}
\theoremstyle{plain}

\numberwithin{equation}{section}
\newtheorem{theorem}{Theorem}[section]
\newtheorem{proposition}[theorem]{Proposition}
\newtheorem{lemma}{Lemma}[section]

\theoremstyle{definition}

\newtheorem{definition}[theorem]{Definition}
\newtheorem{remark}{Remark}[section]

\def\R{\mathbb{R}}

\def\<{\langle} \def\>{\rangle}

\def\L{\mathcal{L}}

\newcommand{\Bv}{{\boldsymbol{v}}}

\newcommand{\Bn}{{\boldsymbol{n}}}

\newcommand{\Bu}{{\boldsymbol{u}}}

\newcommand{\Be}{{\boldsymbol{e}}}

\newcommand{\BU}{\boldsymbol{U}}

\newcommand{\Bw}{{\boldsymbol{w}}}

\newcommand{\Bx}{\boldsymbol{x}}

\def\dd{\mathop{}\!\mathrm{d}}
\def\ii{\mathop{}\!\mathrm{i}}

\title[Effective interface laws for Darcy--Boussinesq convection]{Effective interface laws for Darcy--Boussinesq convection across a vanishing low-permeability layer}

\author{Kaijian Sha}
\address{School of Mathematical Sciences, Eastern Institute of Technology, Ningbo, China}
\email{kjsha@eitech.edu.cn}

\author{Xiaoming Wang}
\address{School of Mathematical Sciences, Eastern Institute of Technology, Ningbo, China}
\email{wxm@eitech.edu.cn}


\begin{document}

%\linenumbers

\begin{abstract}
We study Darcy--Boussinesq convection in a layered porous medium containing a thin layer whose thickness $h_0$ and permeability $K_0$ vanish simultaneously. The competition between these two small parameters is measured by the effective hydraulic conductance $\alpha = \lim \frac{K_0}{h_0} \in [0,\infty]$.  We show that the collapsing layer is replaced by an effective interface whose transmission law depends qualitatively on this ratio.  When $\alpha = 0$, the limiting interface is impermeable; when $ 0 <\alpha<\infty$, the normal Darcy flux is proportional to the pressure jump, yielding a finite-resistance transmission law; and when $\alpha=\infty$, the pressure is continuous and the interface becomes hydraulically transparent.  We establish estimates uniform in the layer thickness and permeability, quantitative convergence of solutions on every finite time interval, and upper semicontinuity of the associated global attractors.  The analysis has to accommodate both a degenerating geometry and a loss of coercivity in the pressure equation.  In particular, we use a generalized Poincar\'e inequality for piecewise $H^1$ pressures, a pressure extension across the collapsing layer, and uniform thin-strip estimates.  The results provide a rigorous finite- and long-time validation of effective interface models for convection through thin low-permeability interphases.
\end{abstract}
\maketitle
%\tableofcontents

\section{Introduction}

Thin low-permeability layers occur naturally in heterogeneous porous formations and in engineered barriers, where they may strongly alter buoyancy-driven transport even when their physical thickness is small.  Examples include stratified geological formations relevant to groundwater flow and subsurface carbon storage, as well as layered porous structures used to regulate heat and mass transfer; see, for instance, \cite{nield2017convection,bickle2007modelling,huppert2014ar,hewitt2014jfm,hewitt2020jfm,hewitt2022jfm}.  From the viewpoint of continuum modeling, a basic question is whether such a thin layer must be resolved as a genuine subdomain or can be replaced by an effective transmission law.  This question becomes particularly delicate when the layer degenerates simultaneously in geometry and in a constitutive parameter: the layer may disappear as a set while retaining a nontrivial hydraulic effect in the macroscopic model.

Rigorous derivation of reduced models from thin or highly heterogeneous structures is a recurring theme in applied analysis.  Classical homogenization results for viscous flow in porous media, such as Allaire's analysis of Stokes flow in a connected periodic porous medium \cite{Allaire1989Stokes}, provide a rigorous route from microscale fluid equations to effective porous-flow laws.  For transport in degenerating geometries, Allaire, Pankratova, and Piatnitski \cite{AllairePankratovaPiatnitski2012} analyzed a nonstationary convection--diffusion equation in a thin rod and in a layer, illustrating how geometry, convection, and boundary layers interact in a singular limit.  Depending on the scaling, thin conducting coatings can generate Dirichlet, Neumann, Robin, or more general effective conditions; see, for example, \cite{ChenPondWang2012,AuvrayVial2018}.  Related asymptotic and homogenization theories for fluid and porous-media interfaces lead to jump, membrane, and transmission laws; see \cite{AngotGoyeauOchoa2017,GahnNeussRadu2025,MarusicPazin2026}.  
Within M3AS, closely related thin-domain and effective-interface models include the reinforcement problem for a Poisson equation with a thin layer \cite{LiZhang2011}, single-phase flow through a thin porous layer \cite{AmazianePankratovPiatnitski2007}, interface laws for coupled porous/elastic media \cite{MikelicWheeler2012}, and the recent derivation of a nonlocal Hele--Shaw--Cahn--Hilliard model from flow in thin heterogeneous layers \cite{CardoneJagerWoukeng2024}. These works illustrate a principle central to the present paper: an effective interface model should be viewed not merely as a formal simplification, but as a continuum model whose validity must be justified through estimates and convergence results that remain stable as the underlying small scales degenerate.

We study this issue for the nonlinear Darcy--Boussinesq system in a layered porous medium.  The Darcy--Boussinesq equations are a standard continuum description of buoyancy-driven convection in fluid-saturated porous media \cite{nield2017convection,doering1998jfm}.  More generally, nonlinear continuum models coupling fluid and heat transport in porous geophysical media can be organized around thermodynamic energy balances and corresponding a priori estimates; a representative example is Roub\'i\v{c}ek's model for heat and fluid flow in damageable poro-elastic continua \cite{Roubicek2017Geophysical}.  For the Darcy--Boussinesq system itself, the mathematical theory includes global well-posedness, regularity, dissipativity, and long-time dynamics in homogeneous and layered settings; see \cite{fabrie1986aam,fabrie1996aa,ly1999jns,oliver2000gevrey,CNW2025}.  For discontinuous material coefficients, Dong and Wang \cite{DW2025} rigorously justified a sharp material-interface model as the limit of a diffuse transition zone.  More recently, the authors considered two complementary one-parameter singular limits: vanishing thickness of a layer with nondegenerate material coefficients \cite{ShaWangThickness2026}, and vanishing permeability of a layer of fixed positive thickness \cite{ShaWangPermeability2026}.  Long-time convergence for narrow transition zones, including convergence of global attractors and statistical quantities, was studied in \cite{ShaWangWu2026}.  The simultaneous limit considered here is qualitatively different from either one-parameter problem: the geometry degenerates while the permeability tends to zero, and the competition between these two effects selects the limiting hydraulic transmission law.

More precisely, let $h_0=h_0(\varepsilon)$ and $K_0=K_0(\varepsilon)$ denote the thickness and permeability of the intermediate layer, respectively, and assume
\[
 h_0(\varepsilon)\to0,\qquad K_0(\varepsilon)\to0,
 \qquad \frac{K_0(\varepsilon)}{h_0(\varepsilon)}\to\alpha\in[0,\infty].
\]
The ratio $\frac{K_0}{h_0}$ is the effective hydraulic conductance of the disappearing layer.  In all three limiting regimes, conservation and the nondegeneracy of the diffusivity yield the common transmission conditions
\[
\llbracket u_z^0\rrbracket_\Gamma=0,\qquad
\llbracket \psi^0\rrbracket_\Gamma=0,\qquad
\llbracket D^0\partial_z\psi^0\rrbracket_\Gamma=0.
\]
The hydraulic law is selected by the limiting conductance:
\begin{equation}\label{intro:interface-phase-diagram}
\left\{\begin{aligned}
 &u_z^0= 0, && \text{ if }\alpha=0,\\
 &u_z^0= -\alpha\llbracket p^0\rrbracket_\Gamma, && \text{ if } 0<\alpha<\infty,\\
 &\llbracket p^0\rrbracket_\Gamma =0, && \text{ if } \alpha=\infty.
\end{aligned}\right.
\end{equation}
Thus $\alpha=0$ produces an impermeable interface, $0<\alpha<\infty$ produces an interface with finite hydraulic resistance, and $\alpha=\infty$ gives perfect hydraulic contact.  The two small parameters therefore do not simply perturb a prescribed interface condition: their relative scaling determines which macroscopic nonlinear continuum model is selected.

A principal objective of this work is to validate these effective continuum models at two distinct time scales.  On finite time intervals, one must show that solutions of the full three-layer problem converge quantitatively to solutions of the corresponding effective interface problem as the layer collapses.  For dissipative nonlinear systems, however, finite-time convergence alone does not ensure that the reduced model reproduces the asymptotic dynamics of the original system.  It is therefore equally important to obtain estimates that are uniform with respect to the small parameters in the dissipative regime and to pass to the limit at the level of long-time dynamics.  This finite-time/infinite-time distinction is central to our analysis.

Our main results reflect this viewpoint.  First, for every $\alpha\in[0,\infty]$, we formulate the corresponding nonlinear Darcy--Boussinesq interface problem and establish its well-posedness in a framework that permits pressure jumps when the hydraulic resistance is finite.  Second, we derive energy and higher-order estimates for the full thin-layer problem with constants independent of the vanishing thickness and permeability.  In particular, these estimates yield dissipative bounds and a common absorbing set that are uniform in $\varepsilon$.  Third, for sufficiently regular common initial data, we prove quantitative convergence of the scalar field, velocity, and outer pressure gradients on arbitrary finite time intervals.  Finally, the uniform dissipativity and compactness allow us to pass to the infinite-time limit and prove upper semicontinuity of the global attractors.  Consequently, the effective interface problems are justified not only as transient approximations but also as models for the long-time dynamics of the original nonlinear continuum system.

The simultaneous limit presents several analytical difficulties that are directly tied to this two-scale validity.  When $0<\alpha<\infty$, the limiting pressure is only piecewise $H^1$ and may jump across the interface, whereas the prelimit pressure is globally $H^1$.  We address this mismatch by constructing a pressure extension across the collapsing layer and by proving a generalized Poincar\'e inequality for piecewise Sobolev functions.  In the regime $\alpha=0$, the jump penalty disappears from the pressure energy and one must instead control the normal flux directly at the shrinking boundary.  At the opposite extreme $\alpha=\infty$, standard thin-domain interpolation estimates are not uniform; we overcome this by subtracting a reflected comparison field and using a thin-strip inequality with a thickness-independent constant.  These ingredients are used both in the quantitative singular-limit analysis and in establishing the uniform dissipative structure needed for the attractor convergence.

The paper is organized as follows.  Section~2 introduces the layered Darcy--Boussinesq continuum model and the homogenized variables.  Section~3 derives the three effective interface laws and establishes well-posedness of the limiting problems.  Section~4 develops estimates that are uniform in the layer thickness and permeability, with particular emphasis on the dissipative bounds required for long-time analysis.  Section~5 proves quantitative finite-time convergence and upper semicontinuity of the global attractors.  Section~6 summarizes the main conclusions and discusses geometric extensions and further singular limits.  The Appendix collects the thin-strip, weighted Sobolev, and generalized Poincar\'e inequalities used in the analysis.

\section{The sharp interface model}
For convection in two-dimensional layered domain $\Omega =  \mathbb{T}_L\times (-H,H)$, the governing equations are the following system \cite{nield2017convection},
\begin{equation}\label{Sharp}
  \left\{
    \begin{aligned}
&\Bu=-\frac{K}{\mu}\left(\nabla P + \rho_{0}(1+\alpha_T\phi)g\Be_{z}\right),\\
&\nabla \cdot \Bu =0,\\
&b\partial_t \phi  + \Bu\cdot\nabla\phi - \nabla \cdot(bD\nabla\phi)=0.
    \end{aligned}\right.
\end{equation}
Here $\Bu$, $\phi$, and $P$ are the unknown fluid velocity, concentration, and pressure, respectively; $\rho_{0}$, $\alpha_T$, $\mu$, and $g$ are the constant reference fluid density, constant thermal/solutal expansion coefficient, constant dynamic viscosity, and gravitational acceleration, respectively; and $\Be_{z}$ stands for the unit vector in the vertical ($z$) direction. In addition, $K$, $b$, $D$ represent the permeability, porosity, and diffusivity respectively.

For the sake of simplicity, we assume that the domain $\Omega$ is divided into three layers by two interfaces located at $z= \pm \tfrac{h_0}{2}$, where $0<h_0<\frac{H}{2}$ is thickness of the intermediate layer. For the idealized sharp interface model, the permeability, porosity, and diffusivity coefficients are assumed to be constant in each layer. Namely,
% \[
% K=K(\boldsymbol{x})=K_i, \quad b=b(\boldsymbol{x})=b_i, \quad D=D(\boldsymbol{x})=D_i, \quad \text{ for any }\Bx\in\Omega_i,
% \]
\begin{equation}\nonumber \label{def:KbD}
    (K,b,D)= \left\{
\begin{aligned}
&(K_+,b_+,D_+),\quad &&z\in \left(\tfrac{h_0}{2},H\right), \\
&(K_0,b_0,D_0),\quad &&z\in \left(-\tfrac{h_0}{2},\tfrac{h_0}{2}\right), \\
&(K_-,b_-,D_-),\quad &&z\in \left(-H, -\tfrac{h_0}{2}\right),
\end{aligned}\right.
    \end{equation}
    for a set of positive constants $\{K_{i},b_{i},D_{i}\}_{i=\pm,0}$. On the interfaces $z=\pm \tfrac{h_0}{2}$, we assume that 
    \begin{equation}\label{interface-1}
\Bu\cdot\Be_{z},\ \phi,\ P,\ \text{and }bD\partial_z\phi \text{ are continuous.} 
    \end{equation}
Furthermore, the system \eqref{Sharp} is supplemented with the initial condition
    \begin{equation}\label{initialdata}
\phi|_{t=0}=\phi_{\rm in}
    \end{equation}
    and the boundary conditions
    \begin{equation}\label{bc0}
\phi|_{z=\pm H}=c_{\pm H},~~\Bu\cdot \Be_z|_{z=\pm H}= 0,
    \end{equation}
where $c_{H},c_{-H}$ are two constants.

    % \begin{lemma}
    %  Assume that $\phi_0\in L^2(\Omega)$ and the porosity $b$ is a constant. The system \eqref{DB} subject to the boundary and interfacial conditions \eqref{interface-1}, \eqref{interface-2}, \eqref{bc} together with the periodic conditions
    %     in the horizontal direction(s), and the initial condition \eqref{initialdata} admits a global weak solution $(\Bu, \phi, P)$ satisfying
    %     \begin{equation}\label{eq2-1}
    %\phi\in L^2(0,T;H_0^1(\Omega))\cap L^\infty(0,T;L^2(\Omega)),~~~~\partial_t \phi \in L^2(0,T;H_0^1(\Omega)^*),
    %     \end{equation}
    %     and
    %     \begin{equation}\label{eq2-2}
    %     (b\phi(t_2),\psi) -  (b\phi(t_1),\psi) = \int_{t_1}^{t_2} (bD \nabla\phi,\nabla\psi) + (\Bu\cdot \nabla \phi ,\psi) \dd t
    %     \end{equation}
    %     for any $T>0$, $t_1,t_2 \in [0,T]$ and $\psi\in H_0^1(\Omega)$. Moreover, one has $\phi(0) =\phi_0$ in $L^2(\Omega)$ and $\Bu \in L^2(0,T;L^q(\Omega))$, $P\in L^2(0,T;W^{1,q}(\Omega))$ with with $q\in [2,\infty)$ if $d= 2$, and $q=6$ if $d=3$.
    % \end{lemma}
To overcome the nonhomogeneous boundary condition of $\phi$ in \eqref{bc0}, we introduce a smooth function $\phi_b(z;\delta)$ on $[-H,H]$ satisfying
\begin{equation}\label{phi_b1}
  \phi_b(z;\delta)= \left\{
    \begin{aligned}
& c_H, && z = H,\\
& \frac{c_H+c_{-H}}{2}&&z \in (-H+\delta,H-\delta),\\
& c_{-H}, && z=-H
    \end{aligned}\right.
  \end{equation}
  and
  \begin{equation}\label{phi_b2}
    |\phi_b'| \leq \frac{c_\Delta}{\delta}, ~~~~|\phi_b''| \leq \frac{2c_\Delta}{\delta^2},
  \end{equation}
  where $\delta>0$ is a parameter to be determined later and $c_\Delta := |c_H-c_{-H}|$.  We always choose $\delta\le \frac{H}{2}$; for sufficiently small $h_0$, the support of $\phi_b'$ is then contained in the two outer layers and is separated from the material interfaces.

Setting
\[ \psi =\phi -  \phi_b,\quad
p = P +\rho_0 g z + \alpha_T \rho_0 g \int_{-H}^z\phi_b(s;\delta)\dd s,\]
and assuming $\mu = 1$, $\rho_{0} \alpha_T g =1$, and $b_+=b_0=b_-=1$ for the remainder of this paper, we obtain the homogenized  Darcy-Boussinesq system
\begin{equation}\label{Sharps}
  \left\{
    \begin{aligned}
&\Bu=- K \left(\nabla p  + \psi \Be_{z}\right),\\
&\nabla \cdot \Bu =0,\\
& \partial_t \psi + \Bu\cdot\nabla \psi  +  \phi_b'\Bu \cdot \Be_z   - \nabla\cdot(D\nabla\psi)=D \phi_b'',
    \end{aligned}\right.
\end{equation}
which is supplemented with the interfacial boundary conditions
\begin{equation}\label{interface1}
  \Bu\cdot\Be_{z},\ \psi,\ p,\ \text{and }D\partial_z\psi\text{ are continuous at } z=\pm \tfrac{h_0}{2} ,
\end{equation}
% \begin{equation}\label{interface2}
% D\partial_z \psi  \text{ are continuous at } z=z_{j},\ 1\leqslant j\leqslant l-1,
% \end{equation}
the initial condition
\begin{equation}\label{initial}
  \psi|_{t=0}=\psi_{\rm in}:=\phi_{\rm in} -\phi_b 
\end{equation}
and homogeneous boundary condition
\begin{equation}\label{bc}
  \psi|_{z=\pm H}= \Bu\cdot \Be_z|_{z=\pm H}= 0.
\end{equation} 

To define the weak solution of the equation, we introduce the following notation. For $1\leq r \leq \infty$ and $k\in \R$, let $L^r (\Omega)$ and $H^k(\Omega)$ denote the usual Lebesgue space of integrable functions and Sobolev spaces on the domain $\Omega$. The inner product in $L^2(\Omega)$ will be denoted by $(\cdot,\cdot)$.
Let $H_0^1(\Omega)$ be the closure of $C^\infty_c( \Omega )$ in the $H^1$ norm. %We denote the dual space of $H_0^1(\Omega)$ by $H^{-1}(\Omega)$. 
We also denote %by $\boldsymbol{H}$ the closure of $\{\Bu \in C^\infty_c( \Omega ):~\nabla \cdot \Bu=0\}$ in the $L^2$ norm.
$$\boldsymbol{H}:=\overline{\{\Bu \in C^\infty_c( \Omega ):~\nabla \cdot \Bu=0\}}^{L^2}$$
Then the weak solution of the problem \eqref{Sharps}-\eqref{bc} is defined as follows.
\begin{definition}\label{def:weaksolution}
Let $\psi_{\rm in}\in L^2(\Omega)$ be given, and let $T>0$. A weak solution of the problem \eqref{Sharps}-\eqref{bc} on the interval $[0,T]$ is a triple $(\Bu, p,\psi)$ satisfying the following conditions:
\begin{enumerate}[(1)]
  \item $\psi \in L^\infty(0,T;L^2(\Omega)) \cap L^2(0,T;H_0^1(\Omega))$,
  \begin{equation}\label{weak-psi}
    (\psi(t_2), \varphi) - ( \psi(t_1), \varphi) + \int_{t_1}^{t_2}(\Bu\cdot\nabla \psi+ \phi_b'u_z,\varphi)  + (D\nabla\psi,\nabla\varphi)-(D \phi_b'',\varphi)\dd t=0,
  \end{equation}
  for any $t_1,t_2\in [0,T]$ and $\varphi\in H_0^1(\Omega)$, and $\psi(0)= \psi_{\rm in}$,
  \item $p \in L^\infty(0,T;H^1(\Omega))$ is a zero-mean weak solution to the problem
  \begin{equation}\label{pressure}
    \left\{
\begin{aligned}
&-\nabla \cdot(K\nabla p) = \nabla\cdot (K\psi \Be_{z}) , ~~~~&&\text{ in } \Omega,\\
&\partial_z p  = 0,&&\text{ on } z= \pm H,\\
\end{aligned}\right.
  \end{equation} 
  \item $\Bu \in L^\infty(0,T;\boldsymbol{H})$ satisfies \eqref{Sharps}$_1$.
\end{enumerate}
\end{definition}

The solution to the Neumann boundary value problem \eqref{pressure} is understood in the weak sense:
\begin{equation}\label{weakform-p}
\int_{\Omega}K(\nabla p+\psi \Be_z) \cdot\nabla q\dd \Bx=0,\quad \text{ for any }q\in H^1(\Omega).
\end{equation}
\begin{remark}
Once the piecewise constant permeability is positive in each layer, the existence of the problem \eqref{pressure} for any $\psi\in L^2(\Omega)$ is guaranteed by the Lax--Milgram theorem. Uniqueness can be obtained if we restrict to the mean-zero subspace of $H^1$. More specifically, if $\psi\in H_0^1(\Omega)$, then $p$ is piecewise $H^2$. 
\end{remark}

For higher regularity of the solutions to the sharp interface model \eqref{Sharps}, we recall the following space, which is introduced in \cite{CNW2025}.
\begin{definition}
Define
\begin{equation}\nonumber
  \mathcal{W} = \{\psi \in H_0^1(\Omega):~\partial_x \psi \in H^1(\Omega),~D\partial_z \psi \in H^1(\Omega)\}
\end{equation}
endowed with norm
\begin{equation}\nonumber
  \|\psi\|_\mathcal{W} = \|\psi\|_{H^1(\Omega)}+\|\partial_x \psi\|_{H^1(\Omega)}+\|D\partial_z \psi\|_{H^1(\Omega)}.
\end{equation}
\end{definition}

  Following \cite{CNW2025}, one can obtain the existence, uniqueness, and regularity of solutions to the sharp interface model \eqref{Sharps}-\eqref{bc} as follows.
  \begin{proposition}\label{prop1}
    % Assume that the  permeability $K$ and diffusivity $D$ are piecewise constant and positive in each layer.  
    For each initial data $\psi_{\rm in}\in L^2(\Omega)$, there exists a unique global weak solution $(\Bu, p,\psi)$ to the problem \eqref{Sharps}--\eqref{bc}. In particular, the map $(\psi_{\rm in},t)\mapsto \psi(t)$ is continuous with respect to $\psi_{\rm in}$ and $t$. Furthermore, if $\psi_{\rm in}\in H_0^1(\Omega)$, then we have
\begin{equation}
\psi \in L^\infty(0,T;H_0^1(\Omega)) \cap L^2(0,T;\mathcal{W}),~~~~~ \text{for any }T>0.
\end{equation}
  \end{proposition}

 %%%%%%%%%%%%%%%%%%%%%%%%%%%%%%%%%%%%%%%%%%%%%%%%%%%%%%%%%%%%
\section{Effective interface models and well-posedness}

In this section, we derive effective interface models
by considering the simultaneous vanishing of the thickness
and permeability of the intermediate layer, and investigate
the well-posedness of the resulting limit problems.

\subsection{The simultaneous limit of thickness and permeability}
\label{sec.3.1}

We consider problem \eqref{Sharps}-\eqref{bc} with the
thickness $h_0$ and permeability $K_0$ of the intermediate
layer depending on a small parameter $\varepsilon\in(0,1]$.
More precisely, we assume that
\[
h_0(\varepsilon)\to 0,\qquad K_0(\varepsilon)\to 0
\quad\text{as }\varepsilon\to 0.
\]
To emphasize the dependence on $\varepsilon$, we denote
the permeability and diffusivity coefficients by
$K^\varepsilon$ and $D^\varepsilon$, respectively, and
write $(\Bu^\varepsilon,\psi^\varepsilon,p^\varepsilon)$
for the corresponding solution to \eqref{Sharps}-\eqref{bc}. 
  % Specifically, we assume that the thickness $h_0 = h_0(\varepsilon)$ and the permeability $K_0 = K_0(\varepsilon)$ of the intermediate layer depend on a parameter $\varepsilon \in (0,1]$ and satisfy
% \[
% h_0(\varepsilon), K_0(\varepsilon) \to 0 \quad \text{as } \varepsilon \to 0 .
% \]
%   In the remainder of this paper, the overall permeability and diffusivity are denoted by $K^\varepsilon$ and $D^\varepsilon$, respectively, to indicate their dependence on $\varepsilon$. Then the solution to problem \eqref{Sharps}--\eqref{bc} with the $\varepsilon$-dependent permeability $K^\varepsilon$ and diffusivity $D^\varepsilon$ is denoted by $(\Bu^\varepsilon, \psi^\varepsilon, p^\varepsilon)$.
  
We also introduce the following notation:
\[
\Omega_{+}
:= \Omega \cap \left\{   0 < z < H \right\},
\qquad
\Omega_{-}
:= \Omega \cap \left\{  -H < z < 0 \right\},
\]
  \[
    \Omega_{+,\varepsilon}
    := \Omega \cap\left\{  \tfrac{h_0}{2} < z < H \right\},
    \qquad
    \Omega_{-,\varepsilon}
    := \Omega \cap \left\{  -H < z < -\tfrac{h_0}{2} \right\},
  \]
  and
  \[
    \Omega_{0,\varepsilon}
    := \Omega \cap \left\{  -\tfrac{h_0}{2} < z < \tfrac{h_0}{2} \right\}.
  \]
  When $\varepsilon \to 0$, the intermediate layer $\Omega_{0,\varepsilon}$  collapses, in a geometric sense, to an interface
  \[ \Gamma:=\Omega\cap \{ z=0\}.\]
Formally, $(\Bu^\varepsilon, \psi^\varepsilon, p^\varepsilon)$ converges to a limit $(\Bu^0, \psi^0, p^0)$ satisfying
\begin{equation}\label{Sharps-limit}
  \left\{
    \begin{aligned}
&\Bu^0=- K^0 \left(\nabla p^0 + \psi^0\Be_{z}\right),\\
&\nabla \cdot \Bu^0=0,\\
& \partial_t \psi^0+ \Bu^0\cdot\nabla \psi^0 +  \phi_b'\Bu^0\cdot \Be_z   - \nabla\cdot(D^0\nabla\psi^0)=D^0 \phi_b'',
    \end{aligned}\right.
\end{equation}
where the limit permeability $K^0$ and diffusivity $D^0$ are given by
\begin{equation}\nonumber
    (K^0,D^0)= \left\{
\begin{aligned}
&(K_+,D_+),\quad &&z\in \left(0,H\right), \\
&(K_-,D_-),\quad &&z\in \left(-H, 0\right).
\end{aligned}\right.
\end{equation} 
Due to flux conservation, the normal flux $\Bu^0\cdot \Bn$ must be continuous across the interfaces. Moreover, because the diffusivity in the intermediate layer remains uniformly positive, the concentration $\psi^0$ is expected to remain continuous across the limiting interface, together with continuity of the diffusive flux. However, the simultaneous vanishing of the thickness and permeability is a singular limit for the Darcy law and creates an effective hydraulic resistance at $\Gamma$. The precise form of the resulting transmission condition depends on the scaling of $K_0$ relative to $h_0$.

The continuity of the normal flux across the interfaces $\Gamma$ gives that 
  \[\llbracket u_z^0\rrbracket_{\Gamma}:= u_z^0
|_{z=0^+}- u_z^0|_{z=0^-} =0.\] 
This, together with Darcy's law \eqref{Sharps-limit}$_1$, yields that 
  \begin{equation}\label{2-1}
    -K_+(\partial_z p^0+\psi^0)|_{z=0^+} = -K_-(\partial_z p^0+\psi^0)|_{z=0^-} =:u^0_z|_{\Gamma}.
  \end{equation}
  In $\Omega_{0,\varepsilon}$, by Darcy's law \eqref{Sharps}$_1$, we have
  \begin{equation}\label{2-1-1}
    p^\varepsilon|_{z=\tfrac{h_0}{2}} - p^\varepsilon|_{z=-\tfrac{h_0}{2}} =\int^{\tfrac{h_0}{2}}_{-\tfrac{h_0}{2}}\partial_z p^\varepsilon \dd z = -\int^{\tfrac{h_0}{2}}_{-\tfrac{h_0}{2}}\frac{1}{K_0}u_z^\varepsilon+ \psi^\varepsilon \dd z.
  \end{equation}
Taking to the limit $\varepsilon\to 0$, we formally deduce
\begin{equation}\label{2-2}
  -\alpha \llbracket p^0\rrbracket_{\Gamma} = u^0_z|_{\Gamma},
\end{equation}
where
\begin{equation}\nonumber
\llbracket p^0\rrbracket_{\Gamma}:= p^0|_{z=0^+} - p^0|_{z=0^-} \quad \text{ and }\quad\alpha :=\lim\limits_{\varepsilon\to 0} \frac{K_0}{h_0}.
\end{equation}

Combining \eqref{2-1} and \eqref{2-2}, the finite-conductance law can equivalently be written as
\begin{equation}\label{2-3}
(\partial_z p^{0})^{\pm}=\frac{\alpha}{K_\pm}\llbracket p^0\rrbracket_{\Gamma}-\psi^0,
\qquad\text{on }\Gamma,
\end{equation}
for $0<\alpha<\infty$.  Formally, the same relation reduces to
\begin{equation}\label{2-4}
(\partial_z p^{0})^{\pm}=-\psi^0,\qquad\text{on }\Gamma,
\end{equation}
when $\alpha=0$, consistently with $u_z^0|_\Gamma=0$.  At the opposite extreme, $\alpha=\infty$ forces
\[
\llbracket p^0\rrbracket_\Gamma=0,
\]
so that the limiting interface has perfect hydraulic contact.

For later use, we also record the normalization inherited from the prelimit pressure.  Integrating \eqref{2-1-1} with respect to $x\in\mathbb{T}_L$ gives
\begin{equation}\label{2-5}
\int_{\mathbb{T}_L}\!\left(p^\varepsilon\big|_{z=\tfrac{h_0}{2}}-p^\varepsilon\big|_{z=-\tfrac{h_0}{2}}\right)\dd x
=-\int_{\mathbb{T}_L}\int_{-\tfrac{h_0}{2}}^{\tfrac{h_0}{2}}\left(\frac{1}{K_0}u_z^\varepsilon+\psi^\varepsilon\right)\dd z\dd x.
\end{equation}
By incompressibility and horizontal periodicity,
\[
\frac{\dd}{\dd z}\int_{\mathbb{T}_L}u_z^\varepsilon\dd x
=-\int_{\mathbb{T}_L}\partial_xu_x^\varepsilon\dd x=0.
\]
Together with $u_z^\varepsilon|_{z=\pm H}=0$, this yields
\begin{equation}\label{2-6}
\int_{\mathbb{T}_L}u_z^\varepsilon\dd x=0,
\qquad \text{ for any }z\in[-H,H].
\end{equation}
Substituting \eqref{2-6} into \eqref{2-5} and passing formally to the limit gives
\begin{equation}\label{2-7}
\int_\Gamma\llbracket p^0\rrbracket_\Gamma\dd\sigma=0.
\end{equation}
For $0<\alpha<\infty$, \eqref{2-7} follows automatically from
$u_z^0=-\alpha\llbracket p^0\rrbracket_\Gamma$ and conservation of the total normal flux.  For $\alpha=0$, by contrast, the two subdomains are hydraulically disconnected and their pressures are determined only up to independent additive constants; in that case \eqref{2-7} is used to select a unique representative and should be viewed as a normalization rather than an additional physical interface law.

We summarize the effective interface conditions.  For every $\alpha\in[0,\infty]$, the scalar field and normal flux satisfy
\begin{equation}\label{interface-common} 
\llbracket u_z^0\rrbracket_\Gamma=0,\qquad
\llbracket\psi^0\rrbracket_\Gamma=0,\qquad
\llbracket D^0\partial_z\psi^0\rrbracket_\Gamma=0. 
\end{equation}
The hydraulic transmission law is
\begin{equation}\label{interface-hydraulic-summary} 
\left\{\begin{aligned} 
 &u_z^0= 0, &&\text{ if } \alpha=0,\\ 
 &u_z^0= -\alpha\llbracket p^0\rrbracket_\Gamma, &&\text{ if } 0<\alpha<\infty,\\ 
 &\llbracket p^0\rrbracket_\Gamma = 0, && \text{ if } \alpha=\infty.
\end{aligned} \right.
\end{equation}
Thus the three distinguished limits correspond, respectively, to an impermeable interface, a finite-resistance interface, and perfect hydraulic contact.  For $0<\alpha<\infty$, the middle relation in \eqref{interface-hydraulic-summary} is equivalent to the pressure transmission condition \eqref{2-3}.

Accordingly, the limit problem \eqref{Sharps-limit} is supplemented with the initial condition
\begin{equation}\label{initial-limit}
\psi^0|_{t=0}=\psi_{\rm in}^0,
\end{equation}
the outer boundary conditions
\begin{equation}\label{bc-limit}
\psi^0|_{z=\pm H}=\Bu^0\cdot\Be_z|_{z=\pm H}=0,
\end{equation}
and the interface laws \eqref{interface-common}--\eqref{interface-hydraulic-summary}.  To retain the notation used below, for finite $\alpha \in [0,\infty)$ we equivalently collect these conditions as
\begin{equation}\label{interface1-limit-1}
\left\{\begin{aligned}
&\llbracket u_z^0\rrbracket_\Gamma=0,\qquad
\llbracket\psi^0\rrbracket_\Gamma=0,\qquad
\llbracket D^0\partial_z\psi^0\rrbracket_\Gamma=0,\\
&u_z^0=-\alpha\llbracket p^0\rrbracket_\Gamma,\\
&\int_\Gamma\llbracket p^0\rrbracket_\Gamma\dd\sigma=0,
\end{aligned}\right. 
\end{equation}
where the last condition is a normalization when $\alpha=0$ and follows from the transmission law when $0<\alpha<\infty$.  For $\alpha=\infty$, the interface conditions become
\begin{equation}\label{interface1-limit}
\llbracket u_z^0\rrbracket_\Gamma=0,\qquad
\llbracket p^0\rrbracket_\Gamma=0,\qquad
\llbracket\psi^0\rrbracket_\Gamma=0,\qquad
\llbracket D^0\partial_z\psi^0\rrbracket_\Gamma=0.
\end{equation}

\subsection{Well-posedness of the limit problem}

The main goal of this paper is to rigorously justify the above formal limit process in both finite-time and long-time regimes. To this end, we first define the weak solution to the limit problem. For the case $\alpha= \infty$, the definition is given in Definition \ref{def:weaksolution}.  For finite $\alpha$, the pressure is defined only up to additive constants on the two connected components when $\alpha=0$.  We fix this freedom by introducing
\begin{equation}\label{def:Q-pressure}
\mathcal Q:=\left\{q\in L^2(\Omega): ~q|_{\Omega^\pm } \in H^1(\Omega_\pm),~\int_{\Omega}q \dd\Bx=0,~ \int_\Gamma\llbracket q\rrbracket_\Gamma\dd\sigma=0
 \right\}.
\end{equation}
For $\alpha>0$, the last normalization follows from the transmission law and conservation of the total normal flux; for $\alpha=0$, it is needed to select a unique representative.
  \begin{definition}\label{def:limit}
    Let $\psi_{\rm in}^0\in L^2(\Omega)$ be given,  and let $\alpha\in [0,\infty)$.
    A weak solution of the limit problem \eqref{Sharps-limit}, \eqref{initial-limit}-\eqref{interface1-limit-1} on the interval $[0,T]$ with $T>0$ is a triple $(\Bu^0, p^0, \psi^0)$ satisfying the following conditions:
    \begin{enumerate}[(1)]
\item $\psi^0 \in L^\infty(0,T;L^2(\Omega)) \cap L^2(0,T;H_0^1(\Omega))$ satisfies the weak formula \eqref{weak-psi} for any $t_1,t_2\in [0,T]$ and $\varphi\in H_0^1(\Omega)$, and $\psi^0(0)= \psi_{\rm in}^0$,
\item $p^0 \in L^\infty(0,T;\mathcal Q)$ is a weak solution to the problem
\begin{equation}\label{pressure-limit}
  \left\{
    \begin{aligned}
&-\nabla \cdot(K^0\nabla p^0) = \nabla\cdot (K^0\psi^0 \Be_{z}) , ~~~~&& \text{ in } \Omega_\pm,\\
&\partial_z p^0 = 0,&& \text{ on } \partial\Omega,\\
&\partial_z p^0 = \frac{\alpha}{K^0}\llbracket p^0\rrbracket_{\Gamma}-\psi^0,&& \text{ on } \Gamma,\\
&\int_\Gamma\llbracket p^0\rrbracket_\Gamma\dd\sigma=0,
    \end{aligned}\right.
  \end{equation}
  where $\llbracket p^0\rrbracket_{\Gamma}$ is the jump of $p^0$ at the interface $\Gamma$. 
\item $\Bu^0 \in L^\infty(0,T;\boldsymbol{H})$ satisfies \eqref{Sharps-limit}$_1$.
\end{enumerate}
    \end{definition}

The solution to the finite-conductance pressure transmission problem \eqref{pressure-limit} is understood in the weak sense:
\begin{equation}\nonumber
  \int_{\Omega_\pm} K^0(\nabla  p^0+\psi^0 \Be_z)\cdot \nabla q  \dd \Bx \pm \alpha \int_{\Gamma}   \llbracket p^0\rrbracket_{\Gamma} q   \dd \sigma=0,\quad \text{ for any }q\in H^1(\Omega_\pm).
\end{equation}
Taking the testing function $q= p^0|_{\Omega_\pm}$ yields
    \begin{equation}\nonumber
0=\int_{\Omega_\pm} K^0(|\nabla p^0 |^2 +  \psi^0\partial_z p^0)  \dd \Bx\pm \alpha\int_{\Gamma} \llbracket p^0\rrbracket_{\Gamma} p^0 \dd \sigma.
    \end{equation}
Summing these equations gives
\begin{equation}\nonumber
0= \int_{\Omega_+ \cup \Omega_-} K^0 (|\nabla p^0|^2 +\psi^0\partial_z p^0)\dd \Bx + \alpha\int_{\Gamma}  \llbracket p^0\rrbracket_{\Gamma}^2\dd \sigma.
\end{equation}
This, together with Lemma \ref{lem:generalized-Poincare}, implies that
\begin{equation}\label{pressure-limit-est}
\|p^0\|_{H^1(\Omega_+\cup \Omega_-)} ^2
+\alpha\|\llbracket p^0\rrbracket_\Gamma\|_{L^2(\Gamma)}^2
\leq C\|\psi^0\|_{L^2(\Omega)}^2.
\end{equation}
Indeed, Lemma~\ref{lem:generalized-Poincare} and the two normalizations in \eqref{def:Q-pressure} control the piecewise $H^1$ norm, while testing the pressure equation with $p^0$ and applying Young's inequality controls the gradient and jump terms by $\|\psi^0\|_{L^2(\Omega)}^2$. Hence the corresponding bilinear form is coercive on $\mathcal Q$ (with the jump term providing additional coercivity when $\alpha>0$). The Lax--Milgram theorem therefore gives, for every $\psi^0\in L^2(\Omega)$, a unique normalized weak solution $p^0\in\mathcal Q$.

\begin{remark}
The parameter $\alpha=\lim\limits_{\varepsilon\to0}\frac{K_0}{h_0}$ determines the effective permeability of the limiting interface $\Gamma$ and, consequently, the behavior of the normal velocity across $\Gamma$. Indeed, for $\alpha\in[0,\infty)$, the transmission condition for the pressure yields
\[
u_z^0=-\alpha \llbracket p^0\rrbracket_\Gamma.
\]
Hence, when $\alpha=0$, one has $u_z^0=0$, 
and the limiting interface is impermeable to the normal flow, although the pressure may possess a nontrivial jump satisfying
\[
\int_\Gamma \llbracket p^0\rrbracket_\Gamma\dd \sigma=0.
\]
\end{remark}

Using the elliptic estimate \eqref{pressure-limit-est} and following \cite{CNW2025}, one can also establish the well-posedness theory of the limit problem \eqref{Sharps-limit}, \eqref{initial-limit}-\eqref{interface1-limit} as follows.
  \begin{proposition}\label{prop1-limit}
For each initial datum $\psi_{\rm in}^0\in L^2(\Omega)$, there exists a unique global weak solution $(\Bu^0,\psi^0, p^0)$ to the limit problem \eqref{Sharps-limit}, \eqref{initial-limit}-\eqref{interface1-limit}. The solution map is continuous from $L^2(\Omega)\times[0,\infty)$ into $L^2(\Omega)$. Furthermore, if $\psi_{\rm in}^0\in H_0^1(\Omega)$, then we have
\begin{equation}\nonumber
\psi^0 \in L^\infty(0,T;H_0^1(\Omega)) \cap L^2(0,T;\mathcal{W}^0),\quad \text{for any }T>0,
\end{equation}
Here the weighted space $\mathcal{W}^0$ is defined in \eqref{def:W-epsilon}.
\end{proposition}

\section{Uniform estimates in thickness and permeability}  
In this section, we establish uniform estimates for solutions to problem \eqref{Sharps}--\eqref{bc} and its limiting counterpart, which play a crucial role in proving finite-time convergence and attractor convergence in Section~5. We begin with dissipative estimates for the scalar field, followed by higher-order estimates in the horizontal direction, and then derive estimates for the pressure and velocity in each layer. The main challenge is to obtain estimates with constants independent of the vanishing thickness and permeability of the intermediate layer, despite the resulting loss of uniform ellipticity in the pressure equation.

In the rest of this paper, we assume that the thickness $h_0 = h_0(\varepsilon)$ and the permeability $K_0 = K_0(\varepsilon)$ of the intermediate layer depend on a parameter $\varepsilon \in (0,1]$ and satisfy
\begin{equation}\label{alpha}
  h_0(\varepsilon),~~K_0(\varepsilon) \to 0 \quad \text{ and } \quad \frac{K_0(\varepsilon)}{h_0(\varepsilon)} \to \alpha\in [0,\infty], \qquad \text{as } \varepsilon \to 0 .
\end{equation}
The overall permeability and diffusivity are denoted by $K^\varepsilon$ and $D^\varepsilon$, respectively. The limit permeability and diffusivity corresponding to the two-layer domain are denoted by $K^0$ and $D^0$, respectively. 

For convenience, we adopt a unified solution notation $(\Bu^\varepsilon,  p^\varepsilon, \psi^\varepsilon)$ of the solution to the problem \eqref{Sharps}--\eqref{bc} for $\varepsilon\in (0,1]$ or the limit problem \eqref{Sharps-limit}, \eqref{initial-limit}-\eqref{interface1-limit} for $\varepsilon=0$.

\begin{proposition}\label{L^2}
Let $\delta$ be chosen so that \eqref{delta_2} is satisfied and $\alpha\in[0,\infty]$.
Then for any $\varepsilon\in[0,1]$ and any initial data $\psi_{\rm in}^\varepsilon\in L^2(\Omega)$, the corresponding solution $(\Bu^\varepsilon, p^\varepsilon, \psi^\varepsilon)$ to problem \eqref{Sharps}--\eqref{bc} (with $\varepsilon>0$) or its limit problem  \eqref{Sharps-limit}, \eqref{initial-limit}-\eqref{interface1-limit} (with $\varepsilon=0$) satisfies 
\begin{equation}\label{L^2-1}
\|\psi^\varepsilon(t)\|_{L^2(\Omega)}^2 \leq \|\psi_{\rm in}^\varepsilon\|_{L^2(\Omega)}^2 e^{-\frac{\min_i D_i}{H^2}t} +\frac{M_1H^2}{\min_i D_i}\left(1- e^{-\frac{\min_i D_i}{H^2}t}\right)
\end{equation}
and
\begin{equation}\label{L^2-1-1}
\int_0^t \|\sqrt{D^\varepsilon} \nabla \psi^\varepsilon(s)\|_{L^2(\Omega)}^2\dd s  \leq M_1 t +\| \psi_{\rm in}^\varepsilon\|_{L^2(\Omega)}^2,
\end{equation}
for any $t>0$, where the constant $M_1$ given in \eqref{L^2-8} is independent of $\varepsilon$ and the initial data.

In particular, for any  $t\ge T_1:=T_1(\|\psi_{\rm in}^\varepsilon\|_{L^2(\Omega)}) =\frac{2H^2}{\min_i D_i}\ln^+ \|\psi_{\rm in}^\varepsilon\|_{L^2(\Omega)}$, one has
\begin{equation}\label{L^2-1-2}
\|\psi^\varepsilon(t) \|_{L^2(\Omega)}^2\leq  \frac{M_1H^2}{\min_i D_i}+1,~~~~~~\int_t^{t+1} \|\sqrt{D^\varepsilon}\nabla\psi^\varepsilon \|_{L^2(\Omega)}^2 \dd s \leq  \frac{M_1H^2 }{\min_i D_i } +M_1+1.
\end{equation}
    \end{proposition}
    \begin{proof} 
We first prove the uniform estimate for $\varepsilon>0$. Testing the equation \eqref{Sharps}$_3$ by $\psi^\varepsilon$ and using integration by parts, one has
\begin{equation}\label{L^2-2}
\frac12  \frac{\dd}{\dd t} \|\psi^\varepsilon\|_{L^2(\Omega)}^2 + \|\sqrt{D^\varepsilon}\nabla\psi^\varepsilon\|_{L^2(\Omega)}^2 =  -   ( \phi_b' u_z^\varepsilon, \psi^\varepsilon ) + (D^\varepsilon \phi_b'', \psi^\varepsilon ) .
\end{equation}
By virtue of \eqref{phi_b1} and \eqref{phi_b2},  we have
\begin{equation}\label{L^2-3}
\left|( \phi_b' u_z^\varepsilon, \psi^\varepsilon ) \right| \leq  c_\Delta \delta^{-1}\|u_z^\varepsilon\|_{L^2(\Omega_\delta)} \|\psi^\varepsilon\|_{L^2(\Omega_\delta)},
\end{equation}
where $\Omega_\delta := \Omega \cap (\{-H<z<-H+\delta\} \cup \{H-\delta <z<H\})$. Noting that both $u_z^\varepsilon$ and $\psi^\varepsilon$ vanish on  the boundary $\{z=\pm H\}$, one uses Poincar\'e inequality to obtain
\begin{equation}\label{L^2-3-1}
\|u_z^\varepsilon\|_{L^2(\Omega_\delta)} \leq   \delta\|\partial_z u_z^\varepsilon\|_{L^2(\Omega_\delta)},~~\|\psi^\varepsilon\|_{L^2(\Omega_\delta)} \leq   \delta\|\partial_z \psi^\varepsilon\|_{L^2(\Omega_\delta)},
\end{equation}
and thus,
\begin{equation}\label{L^2-4}
\left|( \phi_b' u_z^\varepsilon, \psi^\varepsilon ) \right|\leq c_\Delta\delta\|\partial_z u_z^\varepsilon\|_{L^2(\Omega_\delta)}\|\partial_z \psi^\varepsilon\|_{L^2(\Omega_\delta)} \leq  c_\Delta\delta \|\partial_x u_x^\varepsilon\|_{L^2(\Omega)} \|\partial_z \psi^\varepsilon\|_{L^2(\Omega)}.
\end{equation}
Here we have used the fact that $\partial_z u_z^\varepsilon = - \partial_x u_x^\varepsilon$ due to the incompressibility condition $\nabla \cdot \Bu^\varepsilon =0$. Since the permeability coefficient $K^\varepsilon$ is independent of $x$, we have from \eqref{Sharps}$_1$ that
\begin{equation}\label{L^2-5}
\|\partial_x \Bu^\varepsilon\|_{L^{2}(\Omega)} \leq   \|K^\varepsilon\partial_x \nabla p^\varepsilon\|_{L^{2}(\Omega)}+ \|K^\varepsilon\partial_x  \psi^\varepsilon\|_{L^2(\Omega)},
\end{equation}
while the pressure $p^\varepsilon$ is determined by the elliptic problem
\begin{equation} \label{pressure_x}
\left\{
  \begin{aligned}
    &-\nabla\cdot(K^\varepsilon\nabla  \partial_x p^\varepsilon) = \nabla\cdot (K^\varepsilon \partial_x \psi^\varepsilon \Be_z) ,  ~~~~\text{ in } \Omega,\\
    &\partial_z \partial_x p^\varepsilon(x,-H)  = \partial_z \partial_x p^\varepsilon(x,H)  = 0.
  \end{aligned}\right.
\end{equation}
Then it follows from the standard elliptic estimate  that we have
\begin{equation}\label{L^2-5-1}
  \|\sqrt{K^\varepsilon}\partial_x\nabla p^\varepsilon\|_{L^2(\Omega)} \leq    \|\sqrt{K^\varepsilon}\partial_x\psi^\varepsilon\|_{L^2(\Omega)}.
\end{equation}
This, together with \eqref{L^2-5}, yields
\begin{equation}\label{L^2-5-2}
  \begin{aligned}
    \|\partial_x \Bu^\varepsilon\|_{L^{2}(\Omega)} \leq &~\|\sqrt{K^\varepsilon}\|_{L^\infty(\Omega)}(\|\sqrt{K^\varepsilon}\partial_x\nabla p^\varepsilon\|_{L^2(\Omega)}+  \|\sqrt{K^\varepsilon}\partial_x\psi^\varepsilon\|_{L^2(\Omega)})\\
    \leq&~  2 \max_i  K_i  \|\partial_x\psi^\varepsilon\|_{L^2(\Omega)}.
  \end{aligned}
\end{equation}
Hence,
\begin{equation}\label{L^2-6}
  \left|( \phi_b' u_z^\varepsilon, \psi^\varepsilon ) \right|\leq 2c_\Delta \delta \max_i  K_i    \|\nabla \psi^\varepsilon\|_{L^2(\Omega)}^2.
\end{equation}
On the other hand, using H\"older's inequality, Young's inequality and \eqref{L^2-3-1}, we have
\begin{equation}\label{L^2-7}
  \begin{aligned}
    \left|(D^\varepsilon \phi_b'',\psi ^\varepsilon)\right| \leq  &~ 2c_\Delta \delta^{-2} \max_i D_i  |\Omega_\delta|^\frac12 \| \psi^\varepsilon\|_{L^2(\Omega_\delta)}  \\
    \leq  &~ 2c_\Delta L^\frac12  \delta^{-\frac12} \max_i D_i    \| \partial_z\psi^\varepsilon\|_{L^2(\Omega_\delta)} \\
    \leq &~ \frac{1}{4}\min_i  D_i \|\nabla \psi^\varepsilon\|_{L^2(\Omega)}^2+  \frac{8 c_\Delta^2 L }{\delta}\frac{ (\max_i D_i )^2}{\min_i  D_i }.
  \end{aligned}
\end{equation}

Substituting \eqref{L^2-6}-\eqref{L^2-7} into \eqref{L^2-2} and choosing
\begin{equation}\label{delta_2}
  \delta \leq \min\left\{\frac{H}{2},\frac{ \min_i D_i}{8  c_\Delta \max_i  K_i }\right\},
\end{equation}
one has
\begin{equation}\label{L^2-8}
  \frac{\dd}{\dd t} \|\psi^\varepsilon\|_{L^2(\Omega)}^2  +   \|\sqrt{D^\varepsilon}\nabla\psi^\varepsilon\|_{L^2(\Omega)}^2 \leq  \frac{16 c_\Delta^2 L }{\delta}\frac{ (\max_i D_i )^2}{\min_i  D_i }=:M_1.
\end{equation}
Integrating \eqref{L^2-8} over $[0,T]$ gives \eqref{L^2-1-1}. Furthermore, due to Poincar\'e inequality, we also infer from \eqref{L^2-8} that
\begin{equation}\nonumber
  \frac{\dd}{\dd t} \|\psi^\varepsilon\|_{L^2(\Omega)}^2  + \frac{\min_i D_i}{H^2} \|\psi^\varepsilon\|_{L^2(\Omega)}^2 \leq M_1.
\end{equation}
Using the classical Gr\"onwall inequality, we get \eqref{L^2-1}. In particular, it follows from \eqref{L^2-1} that
\begin{equation}\label{L^2-10}
  \|\psi^\varepsilon(t)\|_{L^2(\Omega)}^2 \leq \frac{M_1H^2}{\min_i D_i}+1,~~~~\text{ for any }t \ge T_1=\frac{2H^2}{\min_i D_i}\ln^+  \|\psi_{\rm in}^\varepsilon\|_{L^2(\Omega)}.
\end{equation}
Then, for $t\ge T_1$, we integrate the energy inequality \eqref{L^2-8} over $[t,t+1]$ to obtain
\begin{equation}\label{L^2-11}
  \|\psi^\varepsilon(t+1) \|_{L^2(\Omega)}^2+  \int_t^{t+1} \|\sqrt{D^\varepsilon}\nabla\psi^\varepsilon \|_{L^2(\Omega)}^2 \dd s  \leq \frac{M_1H^2}{\min_i D_i} +M_1 +1.
\end{equation}
Combining \eqref{L^2-10} and \eqref{L^2-11}, one proves \eqref{L^2-1-2}. 

The case $\varepsilon=0$ and $\alpha=\infty$ can be proved in the same way.
For the case $\varepsilon=0$ and $\alpha\in [0,\infty)$, the only difference is that the limit pressure $p^0$ is determined by the pressure transmission problem \eqref{pressure-limit} instead of the Neumann problem \eqref{pressure}. To this end, we take the $x$-derivative of \eqref{pressure-limit} to obtain
\begin{equation}\label{pressure-limit-x}
\left\{\begin{aligned}
&-\nabla \cdot(K^0\nabla \partial_xp^0) = \nabla\cdot (K^0\partial_x\psi \Be_{z}) , ~~~~&& \text{ in } \Omega_\pm,\\
&\partial_z \partial_x p^0 = 0,&& \text{ on } z=\pm H,\\
&\partial_z \partial_x p^0 = \frac{\alpha}{K^0}\llbracket \partial_xp^0\rrbracket-\partial_x\psi^0,&& \text{ on } z=0,
    \end{aligned}\right.
  \end{equation}
Then, similar to \eqref{pressure-limit-est}, one also has
\begin{equation}\label{L^2-12}
\int_{\Omega_+ \cup \Omega_-} K^0  |\nabla\partial_x p^0|^2  \dd \Bx + \alpha\int_{\Gamma}  \llbracket \partial_xp^0\rrbracket ^2\dd \sigma\leq  \int_{\Omega_+ \cup \Omega_-} K^0  |\partial_x\psi^0|^2  \dd \Bx,
\end{equation}
which coincides with \eqref{L^2-5-1} for the case $\varepsilon>0$.  
With \eqref{L^2-12} in hand, one can prove the estimate \eqref{L^2-1}-\eqref{L^2-1-1} and \eqref{L^2-1-2} similarly. Then we finish the proof of this proposition.
\end{proof}

\begin{proposition}\label{H^1}
Let $\delta$ be chosen so that \eqref{delta_2} is satisfied and $\alpha\in[0,\infty]$. 
Then for any $\varepsilon\in[0,1]$ and any initial data $\psi_{\rm in}^\varepsilon\in L^2(\Omega)$, the corresponding solution $(\Bu^\varepsilon, p^\varepsilon, \psi^\varepsilon)$ to problem \eqref{Sharps}--\eqref{bc} (with $\varepsilon>0$) or its limit problem  \eqref{Sharps-limit}, \eqref{initial-limit}-\eqref{interface1-limit} (with $\varepsilon=0$)  satisfies 
\begin{equation} \label{H^1-0}
  \|\sqrt{D^\varepsilon}\nabla \psi^\varepsilon(t)\|_{L^2(\Omega)}^2   \leq  M_5, \qquad\text{for any } t\ge T_1+1 
\end{equation}
and 
\begin{equation} \label{H^1-0-1}
  \int_t^{t+1}\|\L^\varepsilon \psi^\varepsilon\|_{L^2(\Omega)}^2  \dd s \leq  M_6, \qquad\text{for any } t\ge T_1+1,
\end{equation}
where the constants $M_5,M_6$ given in \eqref{M5} and \eqref{M6}, respectively, are independent of $\varepsilon$ and the initial data. 

If, in addition, $\psi_{\rm in}^\varepsilon\in H_0^1(\Omega)$, we have
\begin{equation}\label{H^1-1}
  \|\sqrt{D^\varepsilon}\nabla \psi^\varepsilon(t)\|_{L^2(\Omega)} ^2 \leq M_7(\|\psi_{\rm in}^\varepsilon\|_{H^1(\Omega)}),\qquad\text{for any } t\ge 0 
\end{equation} 
and 
\begin{equation}\label{H^1-1-1}
  \int_0^t\|\L^\varepsilon \psi^\varepsilon\|_{L^2(\Omega)}^2  \dd s  \leq M_8(t,\|\psi_{\rm in}^\varepsilon\|_{H^1(\Omega)}),\qquad\text{for any } t\ge 0,
\end{equation} 
where the constants $M_7,M_8$ independent of $\varepsilon$  are given in \eqref{M7} and \eqref{M8}, respectively.
\end{proposition}
\begin{proof}The proof is divided into three steps.

 {\bf Step 1. The case $\varepsilon>0$ and $\alpha\in [0,\infty)$}. Testing \eqref{Sharps}$_3$ by $\mathcal{L}^\varepsilon \psi^\varepsilon = -\nabla\cdot (D^\varepsilon\nabla\psi^\varepsilon)$ and using integration by parts, we obtain 
\begin{equation}\label{H^1-3}
  \begin{aligned}
&~\frac12  \frac{\dd}{\dd t} \|\sqrt{ D^\varepsilon}\nabla \psi^\varepsilon\|_{L^2(\Omega)}^2 + \|\mathcal{L}^\varepsilon \psi^\varepsilon\|_{L^2(\Omega)}^2 \\
= &~-(\Bu^\varepsilon\cdot\nabla  \psi^\varepsilon,\mathcal{L}^\varepsilon \psi^\varepsilon) -(\phi_b' u_z, \mathcal{L}^\varepsilon \psi^\varepsilon) +  (D^\varepsilon\phi_b'',\mathcal{L}^\varepsilon \psi^\varepsilon)   .
  \end{aligned}
\end{equation}
Using H\"older's inequality, Young's inequality and \eqref{phi_b2}, one has
\begin{equation}\label{H^1-4}
  \begin{aligned}
    |(D^\varepsilon \phi_b'',\mathcal{L}^\varepsilon \psi^\varepsilon)|\leq  &~ \frac{2c_\Delta}{\delta^2}\max_i D_i  \int_{\Omega_{\delta}} | \mathcal{L}^\varepsilon \psi^\varepsilon|  \dd \Bx  \\
    \leq  &~ \frac{2c_\Delta}{\delta^2}\max_i D_i   |\Omega_\delta|^\frac12 \|\mathcal{L}^\varepsilon \psi^\varepsilon\|_{L^2(\Omega_\delta)}  \\
    \leq  &\frac{1}{4}\|\mathcal{L}^\varepsilon \psi^\varepsilon\|_{L^2(\Omega)}^2 + 8c_\Delta^2 L \delta^{-3} (\max_i D_i )^2.
  \end{aligned}
\end{equation}
In a way similar to \eqref{L^2-3}-\eqref{L^2-6}, one may use Young's inequality to deduce
\begin{equation}\label{H^1-5}
  \begin{aligned}
    \left|(\phi_b' u_z^\varepsilon, \mathcal{L}^\varepsilon \psi^\varepsilon)\right|
    \leq  &~  c_\Delta \|\partial_x u_x^\varepsilon\|_{L^2(\Omega)}\|\mathcal{L}^\varepsilon \psi^\varepsilon\|_{L^2(\Omega)}\\
    \leq  &~  2  c_\Delta  \max_i  K_i    \|\partial_x \psi^\varepsilon\|_{L^2(\Omega)}\|\mathcal{L}^\varepsilon \psi^\varepsilon\|_{L^2(\Omega)}\\
    \leq  &~ \frac{1}{4}\|\mathcal{L}^\varepsilon \psi^\varepsilon\|_{L^2(\Omega)}^2+ 4c_\Delta^2 (\max_i  K_i)^2  \|\nabla\psi^\varepsilon\|_{L^2(\Omega)}^2.
  \end{aligned}
\end{equation}
Furthermore, using H\"older's inequality, Young's inequality and Lemmas \ref{lemmaA1}-\eqref{lemma-equivalence}, we have
\begin{equation}\label{H^1-6}
  \begin{aligned}
    \left|(\Bu^\varepsilon\cdot\nabla \psi^\varepsilon, \mathcal{L}^\varepsilon \psi^\varepsilon)\right|
    \leq  &~ \|\Bu^\varepsilon\|_{L^4(\Omega)} \|\nabla \psi^\varepsilon\|_{L^4(\Omega)}  \|\mathcal{L}^\varepsilon  \psi^\varepsilon\|_{L^2(\Omega)}\\
    \leq  &~ C\|\Bu^\varepsilon\|_{L^4(\Omega)} \|\nabla \psi^\varepsilon\|_{L^2(\Omega)}^\frac12 \|\mathcal{L}^\varepsilon  \psi^\varepsilon\|_{L^2(\Omega)}^\frac32\\
    \leq  &~\frac{1}{4}\|\mathcal{L}^\varepsilon  \psi^\varepsilon\|_{L^2(\Omega)}^2+  C\|\nabla \psi^\varepsilon\|_{L^2(\Omega)}^2\sum_{i}\|\Bu^\varepsilon\|_{L^4(\Omega_{i,\varepsilon})}^4 .
  \end{aligned}
\end{equation}
Similar to \eqref{L^2-5-1}, one has also
\begin{equation}\label{H^1-7-0}
  \|\sqrt{K^\varepsilon} \nabla p^\varepsilon\|_{L^2(\Omega )}
  \leq   \| \sqrt{K^\varepsilon}\psi^\varepsilon\|_{L^2(\Omega)}.
\end{equation}
This, together with \eqref{Sharps}$_1$, gives
\begin{equation} \label{H^1-7-1}
  \| \Bu^\varepsilon\|_{L^2(\Omega_{i,\varepsilon})} \leq \| K^\varepsilon\nabla p^\varepsilon\|_{L^2(\Omega_{i,\varepsilon})}+ \| K^\varepsilon\psi^\varepsilon\|_{L^2(\Omega_{i,\varepsilon})} \leq 2\sqrt{K_i}  \| \sqrt{K^\varepsilon}\psi^\varepsilon\|_{L^2(\Omega)},
\end{equation}
for $i=\pm,0$.

On the other hand, since $\psi^\varepsilon \in H_0^1(\Omega)$, the pressure $p^\varepsilon$ is piecewise $H^2$ and satisfies
\[ -\Delta p^\varepsilon = \partial_z\psi^\varepsilon,\quad \text{ in each }\Omega_{i,\varepsilon}.\]
Using \eqref{L^2-5-1}, one has
\begin{equation}\nonumber
  \|\sqrt{K^\varepsilon}\partial_{z}^2 p^\varepsilon\|_{L^2(\Omega_{i,\varepsilon}) }\leq   \|\sqrt{K^\varepsilon}\partial_{x}^2 p^\varepsilon\|_{L^2(\Omega_{i,\varepsilon}) }+ \|\sqrt{K^\varepsilon}\partial_{z}\psi^\varepsilon\|_{L^2(\Omega_{i,\varepsilon}) }
  \leq 2\|\sqrt{K^\varepsilon}\nabla\psi^\varepsilon\|_{L^2(\Omega) }
\end{equation}
and
\begin{equation} \label{H^1-7-2}
  \begin{aligned}
    \|\nabla\Bu^\varepsilon\|_{L^2(\Omega_{i,\varepsilon})} \leq&~
    \|K^\varepsilon\nabla^2 p^\varepsilon\|_{L^2(\Omega_{i,\varepsilon})} +\| K^\varepsilon \nabla \psi^\varepsilon\|_{L^2(\Omega_{i,\varepsilon})}\\
    \leq&~\sqrt{K_i}(\|\sqrt{K^\varepsilon} \nabla^2 p^\varepsilon\|_{L^2(\Omega_{i,\varepsilon})} +\|\sqrt{K^\varepsilon}\nabla \psi^\varepsilon\|_{L^2(\Omega_{i,\varepsilon})})\\
    \leq&~C \sqrt{K_i}\|\sqrt{K^\varepsilon}\nabla \psi^\varepsilon\|_{L^2(\Omega)}.
  \end{aligned}
\end{equation}
By virtue of Lemma \ref{lemmaA1}, Poincar\'e inequality and \eqref{H^1-7-1}-\eqref{H^1-7-2}, one has
\begin{equation}\label{H^1-7}
  \begin{aligned}
    \|\Bu^\varepsilon\|_{L^4(\Omega_{i,\varepsilon})}\leq&~  Ch_i^{-\frac14}( \|\Bu^\varepsilon\|_{L^2(\Omega_{i,\varepsilon})}^\frac12\|\nabla\Bu^\varepsilon\|_{L^2(\Omega_{i,\varepsilon})}^\frac12 + \|\Bu^\varepsilon\|_{L^2(\Omega_{i,\varepsilon})}) \\
    \leq&~C h_i^{-\frac14}K_i^{\frac12}\|\nabla\psi^\varepsilon\|_{L^2(\Omega)}^\frac12\|\psi^\varepsilon\|_{L^2(\Omega)}^\frac12\\
    \leq&~C K_i^{\frac14} \|\nabla\psi^\varepsilon\|_{L^2(\Omega)}^\frac12\|\psi^\varepsilon\|_{L^2(\Omega)}^\frac12.
  \end{aligned}
\end{equation}
Here we used the assumption \eqref{alpha} with $\alpha\in [0,\infty)$. Substituting the estimates \eqref{H^1-4}-\eqref{H^1-6} into \eqref{H^1-3} and using \eqref{H^1-7}, one obtains
\begin{equation}\label{H^1-8}
  \begin{aligned}
    &~ \frac{\dd}{\dd t} \|\sqrt{D^\varepsilon}\nabla \psi^\varepsilon\|_{L^2(\Omega)}^2  + \frac{1}{2} \|\mathcal{L}^\varepsilon \psi^\varepsilon\|_{L^2(\Omega)}^2 \\
    \leq& ~ M_2 + (M_3  + M_4\|\psi^\varepsilon\|_{L^2(\Omega)}^2 \|\sqrt{D^\varepsilon}\nabla \psi^\varepsilon\|_{L^2(\Omega)}^2)\|\sqrt{D^\varepsilon} \nabla \psi^\varepsilon\|_{L^2(\Omega)}^2  ,
  \end{aligned}
\end{equation}
with constants
\begin{equation}\label{M2M3}
  M_2= 16 c_\Delta^2 L \delta^{-3} (\max_i D_i )^2,\quad ~~M_3 = \frac{8 c_\Delta^2(\max_i  K_i)^2   }{ \min_i  D_i},
\end{equation}
and
\begin{equation}\label{M4}
  M_4 = \frac{C\max_i K_i}{(\min_i D_i)^2}.
\end{equation}

Set
\begin{equation}\nonumber
  y(t) := \|\sqrt{D^\varepsilon}\nabla \psi^\varepsilon\|_{L^2(\Omega)}^2, ~~~g(t) := M_3+M_4\|\psi^\varepsilon\|_{L^2(\Omega)}^2 \|\sqrt{D^\varepsilon}\nabla \psi^\varepsilon\|_{L^2(\Omega)}^2,~~~h(t) := M_2.
\end{equation}
Then, for any $t>0$, we have
\begin{equation}\nonumber
  y'(t) \leq g(t) y(t) + h(t) \quad\text{ and} \quad
  \int_{t}^{t+1} h(s)\dd s = M_2.
\end{equation}
Furthermore, if $t\ge T_1$, it follows from Proposition \ref{L^2} that
\begin{equation}\nonumber
  \int_{t}^{t+1} y(s)\dd s = \int_{t}^{t+1} \|\sqrt{D^\varepsilon}\nabla \psi^\varepsilon\|_{L^2(\Omega)}^2 \dd s \leq \frac{M_1H^2 }{\min_i D_i} +M_1+ 1
\end{equation}
and
\begin{equation}\nonumber
  \begin{aligned}
    \int_{t}^{t+1} g(s)\dd s \leq&~ M_3 +M_4\sup_{s\in[t,t+1]} \|\psi^\varepsilon(s)\|_{L^2(\Omega)}^2\int_{t}^{t+1}   \|\sqrt{D^\varepsilon}\nabla \psi^\varepsilon\|_{L^2(\Omega)}^2 \dd s \\
    \leq&~ M_3+ M_4 \left(\frac{M_1H^2}{\min_i D_i}+M_1+ 1\right)^2.
  \end{aligned}
\end{equation}
Then an application of the uniform Gr\"onwall lemma \cite[Lemma 1.1]{Temam1997} gives
\begin{equation}\label{M5}
\begin{aligned}
  y(t+1)  \leq &~ \left(\frac{M_1H^2}{\min_i D_i}+M_2+M_1+1\right) \\
  &\times\operatorname{exp}\left[M_3+M_4\left(\frac{M_1H^2}{\min_i D_i}+M_1 +1\right)^2\right]=:M_5,
\end{aligned}
\end{equation}
for any $t\ge T_1$, which is exactly \eqref{H^1-0}. Integrating \eqref{H^1-8} over $[t,t+1]$ for any $t\ge T_1+1$ and using \eqref{M5}, one has
\begin{equation}\label{M6}
  \begin{aligned}
 &~ \frac{1}{2} \int_t^{t+1}\|\mathcal{L}^\varepsilon \psi^\varepsilon(s)\|_{L^2(\Omega)}^2 \dd s \\ 
    % \leq& ~  M_2 + y(t)+ \sup_{s\in (t,t+1)} y(s)\cdot \int_{t}^{t+1} g(s)\dd s  \\
    \leq&~ M_2 +  M_5\left[1+ M_3+M_4\left(\frac{M_1H^2}{\min_i D_i}+M_1+ 1\right)^2\right]\\
    =&:\frac{1}{2}M_6.
  \end{aligned}
\end{equation}
Next we assume that $\psi_{\rm in}^\varepsilon \in H_0^1(\Omega)$. With the aid of \eqref{L^2-1} and \eqref{L^2-1-1}, one could apply Gr\"onwall's inequality to \eqref{H^1-8} and deduce
\begin{equation}\nonumber
  \begin{aligned}
   y(t) \leq&~ y(0)\exp\left(\int_0^t g(s) \dd s\right)  + \int_0^t h(s)  \exp\left(\int_\tau^t g(s) \dd s\right) \dd \tau\\
    \leq&~ (\|\sqrt{D^\varepsilon}\nabla \psi_{\rm in}^\varepsilon\|_{L^2(\Omega)}^2+  M_2 t)\exp\left(M_3 t +M_4\int_0^t \|\psi^\varepsilon\|_{L^2(\Omega)}^2 \|\sqrt{D^\varepsilon}\nabla \psi^\varepsilon\|_{L^2(\Omega)}^2  \dd s\right)\\
    \leq&~ (\|\sqrt{D^\varepsilon}\nabla \psi_{\rm in}^\varepsilon\|_{L^2(\Omega)}^2+ M_2 t)\\
    &\times\exp\left[M_3 t +M_4\left(\frac{M_1H^2}{\min_i D_i}+\|\psi_{\rm in}^\varepsilon\|_{L^2(\Omega)}^2\right)\left(M_1t + \|\psi_{\rm in}^\varepsilon\|_{L^2(\Omega)}^2\right)\right]\\
    =&:M_5'(t,\|\psi_{\rm in}^\varepsilon\|_{H^1(\Omega)}).
  \end{aligned}
\end{equation}
This, together with the uniform estimate \eqref{M5} in long time, gives
\begin{equation}\label{M7}
  \begin{aligned}
    \|\sqrt{D^\varepsilon}\nabla \psi^\varepsilon(t)\|_{L^2(\Omega)}^2\leq&~   \max\{M_5,M_5'(T_1+1,  \|\psi_{\rm in}^\varepsilon\|_{H^1(\Omega)})\}  \\
    =&:M_7(\|\psi_{\rm in}^\varepsilon\|_{H^1(\Omega)}),
  \end{aligned}
\end{equation}
for any $t\ge 0$. Finally, we integrate \eqref{H^1-8} over $[0,t]$ to obtain 
\begin{equation}\label{M8}
\begin{aligned}
&~\frac12\int_0^t \|\mathcal{L}^\varepsilon\psi^\varepsilon\|_{L^2(\Omega)}^2 \dd s\\
 \leq&~ \|\sqrt{D^\varepsilon}\nabla \psi_{\rm in}^\varepsilon\|_{L^2(\Omega)}^2 +\left[M_2+M_3M_7+M_4 M_7^2\left(\frac{M_1H^2}{\min_iD_i}+\|\psi_{\rm in}^\varepsilon\|_{L^2(\Omega)}^2\right)\right]t\\
=&: \frac12 M_8(t,\|\psi_{\rm in}^\varepsilon\|_{H^1(\Omega)}).
\end{aligned}
\end{equation}
Then we finish the proof for the case $\varepsilon>0$ and $\alpha\in [0,\infty)$.

{\bf Step 2. The case $\varepsilon>0$ and $\alpha=\infty$.} In this case, the estimate \eqref{H^1-7} is not applicable for the intermediate layer $\Omega_{0,\varepsilon}$, since $h_0^{-\frac14} K_0^{\frac12}$ can be unbounded as $\varepsilon\to 0$. However, in the two layers $\Omega_{\pm,\varepsilon}$, the estimate \eqref{H^1-7} is still valid and is uniform in $\varepsilon$. Then we have
\begin{equation}\label{H^1-11}
\|\Bu^\varepsilon\|_{L^4(\Omega_{\pm,\varepsilon})}
\leq
C
\|\psi^\varepsilon\|_{L^2(\Omega)}^{\frac12}
\|\nabla\psi^\varepsilon\|_{L^2(\Omega)}^{\frac12}.
\end{equation}
For the intermediate layer, since $h_0(\varepsilon)\to0$, we may assume that
\[
0<h_0(\varepsilon)\leq\min\left\{1,\frac{H}{2}\right\}, \quad \text{ for any }\varepsilon\in (0,1].
\]
For notational convenience, we write
\[
\Bu^\varepsilon_0
:=
\Bu^\varepsilon\big|_{\Omega_{0,\varepsilon}}\quad \text{ and }\quad
\qquad
\Bu^\varepsilon_+
:=
\Bu^\varepsilon\big|_{\Omega_{+,\varepsilon}}.
\]

Set
\[
Q_{+,\varepsilon}
:=
\mathbb{T}_L
\times
\left(
\frac{h_0}{2},
\frac{3h_0}{2}
\right)
\subset\Omega_{+,\varepsilon}.
\]
Due to the continuity of the normal velocity across the interface, one has
\[
u^\varepsilon_{0,z}
=
u^\varepsilon_{+,z}
\qquad
\text{on }
z=\frac{h_0}{2}.
\]
Moreover, since the pressure $p^\varepsilon$ is continuous across the interface
and is piecewise $H^2$, its tangential derivative is also
continuous. Then Darcy's law  gives
\[
u^\varepsilon_{0,x}
=
-K_0\partial_xp^\varepsilon
=
\frac{K_0}{K_+}u^\varepsilon_{+,x},
\qquad
\text{on }
z=\frac{h_0}{2}.
\]
We define a reflected comparison field
$\Bv^\varepsilon$ in $\Omega_{0,\varepsilon}$ by
\begin{equation}\label{def:v}
\Bv^\varepsilon(x,z)
:=
\left(
\frac{K_0}{K_+}
u^\varepsilon_{+,x}(x,h_0-z),
\,
u^\varepsilon_{+,z}(x,h_0-z)
\right),
\quad\text{ for }(x,z)\in \Omega_{0,\varepsilon}.
\end{equation}
It is straightforward to see that 
\begin{equation}
\label{eq:G-matches-u}
\Bv^\varepsilon
= \Bu^\varepsilon_0
\qquad
\text{on }
z=\frac{h_0}{2}. 
\end{equation}
Moreover, by \eqref{H^1-7-1}-\eqref{H^1-7-2} and \eqref{H^1-11}, one has
\begin{equation}\label{H^1-12}
\|\Bv^\varepsilon\|_{L^2(\Omega_{0,\varepsilon})} \leq
C  
\|\Bu^\varepsilon \|_{L^2(Q_{+,\varepsilon})}
\leq
C
\|\psi^\varepsilon\|_{L^2(\Omega)},
\end{equation}
\begin{equation}\label{H^1-13}
\|\nabla \Bv^\varepsilon\|_{L^2(\Omega_{0,\varepsilon})}\leq
C
\|\nabla \Bu^\varepsilon \|_{L^2(Q_{+,\varepsilon})}
\leq
C
\|\nabla\psi^\varepsilon\|_{L^2(\Omega)}
\end{equation}
and 
\begin{equation}\label{H^1-14}
\|\Bv^\varepsilon\|_{L^4(\Omega_{0,\varepsilon})}
\leq
C
\|\Bu^\varepsilon\|_{L^4(Q_{+,\varepsilon})}
\leq
C
\|\psi^\varepsilon\|_{L^2(\Omega)}^{\frac12}
\|\nabla\psi^\varepsilon\|_{L^2(\Omega)}^{\frac12}.
\end{equation}

Define
\[
\Bw^\varepsilon
:= \Bu^\varepsilon_0-\Bv^\varepsilon
\qquad
\text{in }\Omega_{0,\varepsilon}.
\]
By \eqref{eq:G-matches-u}, one has $\Bw^\varepsilon
\big|_{z=\frac{h_0}{2}}
=0$. On the other hand, by virtue of \eqref{H^1-7-1} and \eqref{H^1-7-2}, we have
\begin{equation}
\|\Bu^\varepsilon_0\|_{L^2(\Omega_{0,\varepsilon})}
\leq
C K_0^{\frac12}
\|\psi^\varepsilon\|_{L^2(\Omega)} 
\label{eq:u0-L2}
\end{equation}
and
\begin{equation}
\|\nabla \Bu^\varepsilon_0\|_{L^2(\Omega_{0,\varepsilon})}
\leq
C K_0^{\frac12}
\|\nabla\psi^\varepsilon\|_{L^2(\Omega)} .
\label{eq:u0-H1}
\end{equation}
Combining \eqref{H^1-12}--\eqref{eq:u0-H1}, we obtain
\begin{equation}\nonumber
\|\Bw^\varepsilon\|_{L^2(\Omega_{0,\varepsilon})}
\leq
C\|\psi^\varepsilon\|_{L^2(\Omega)},
\qquad
\|\nabla \Bw^\varepsilon\|_{L^2(\Omega_{0,\varepsilon})}
\leq
C\|\nabla\psi^\varepsilon\|_{L^2(\Omega)}.\end{equation}
Applying Lemma~\ref{lem:A2} to
$\Bw^\varepsilon$ gives
\begin{equation}\label{eq:w-L4}
\|\Bw^\varepsilon\|_{L^4(\Omega_{0,\varepsilon})}\leq
C
\|\Bw^\varepsilon\|_{L^2(\Omega_{0,\varepsilon})}^{\frac12}
\|\nabla \Bw^\varepsilon\|_{L^2(\Omega_{0,\varepsilon})}^{\frac12}
\leq
C
\|\psi^\varepsilon\|_{L^2(\Omega)}^{\frac12}
\|\nabla\psi^\varepsilon\|_{L^2(\Omega)}^{\frac12}.
\end{equation}
Consequently, by \eqref{H^1-14} and \eqref{eq:w-L4}, one obtains
\begin{equation}\label{H^1-7-3}
\|\Bu^\varepsilon\|_{L^4(\Omega_{0,\varepsilon})}
\leq
C
\|\psi^\varepsilon\|_{L^2(\Omega)}^{\frac12}
\|\nabla\psi^\varepsilon\|_{L^2(\Omega)}^{\frac12}.
\end{equation}
Combining \eqref{H^1-7-3} and \eqref{H^1-11}, one obtains also the estimate given in \eqref{H^1-7}. Then one could repeat the argument in Step 1. to obtain \eqref{H^1-0}--\eqref{H^1-0-1} for the case $\varepsilon>0$ and $\alpha=\infty$.

{\bf Step 3. The case $\varepsilon=0$.} The case $\varepsilon=0$ and $\alpha=\infty$ can be proved in the same way. The proof for $\varepsilon=0$ and $\alpha\in[0,\infty)$ is quite similar to that in Step 1. if one uses the estimate \eqref{pressure-limit-est} for the pressure transmission problem \eqref{pressure-limit} instead of the estimate \eqref{H^1-7-0} for the Neumann boundary problem \eqref{pressure}, so we omit the details for brevity. Then we finish the proof of this proposition.
\end{proof}

\begin{proposition}\label{H^1-x}
Let $\delta$ be chosen so that \eqref{delta_2} is satisfied and $\alpha\in[0,\infty]$. 
Then for any $\varepsilon\in[0,1]$, and any initial data $\psi_{\rm in}^\varepsilon\in L^2(\Omega)$, the corresponding solution $(\Bu^\varepsilon, p^\varepsilon, \psi^\varepsilon)$ to problem \eqref{Sharps}--\eqref{bc} (with $\varepsilon>0$) or its limit problem  \eqref{Sharps-limit}, \eqref{initial-limit}-\eqref{interface1-limit} (with $\varepsilon=0$)  satisfies 
\begin{equation} \label{H^1-x-0}
\|\sqrt{D^\varepsilon}\nabla\partial_x\psi^\varepsilon(t)\|_{L^2(\Omega)}^2\leq  M_9, \qquad\text{for any } t\ge T_1+2,
\end{equation}
where $T_1= T_1(\|\psi^\varepsilon_{\rm in}\|_{L^2(\Omega)})$ and the constant $M_9$ given in \eqref{M9} are independent of $\varepsilon$ and the initial data.  

If, in addition, $\psi^\varepsilon_{\rm in}$ and $\partial_x\psi^\varepsilon_{\rm in}$ belong to $H_0^1(\Omega)$, then it holds that 
\begin{equation} \label{H^1-x-0-1}
\|\sqrt{D^\varepsilon}\nabla\partial_x\psi^\varepsilon(t)\|_{L^2(\Omega)}^2\leq  M_{10}(\|\psi_{\rm in}^\varepsilon\|_{H^1(\Omega)},\|\partial_x\psi_{\rm in}^\varepsilon\|_{H^1(\Omega)}), \qquad\text{for any } t\ge 0,
\end{equation}
where $M_{10}$ given in \eqref{M10} is a constant independent of $\varepsilon$.
\end{proposition}
\begin{proof} The proof is divided into three steps.  We first prove the uniform estimate for $\varepsilon>0$ and then sketch the proof for $\varepsilon=0$ at the end.

{\bf Step 1. $L^2$-estimate for $\partial_x\psi^\varepsilon$.} For $L^2$ initial data, the differentiated estimates below are understood via horizontal difference quotients or Galerkin approximations, starting at $t_0=T_1+1$, where Proposition \eqref{H^1} provides the required uniform $H^1$ regularization. Passing to the limit justifies the equations for $\partial_x\psi^\varepsilon$.
Set 
\[(\Psi^\varepsilon,\BU^\varepsilon, P^\varepsilon):= \partial_x(\psi^\varepsilon, u^\varepsilon, p^\varepsilon).\]
Then $(\Psi^\varepsilon,\BU^\varepsilon, P^\varepsilon)$ satisfies the following problem 
\begin{equation}\label{Sharps-x}
  \left\{
    \begin{aligned}
&\BU^\varepsilon=- K^\varepsilon\left(\nabla P^\varepsilon + \Psi^\varepsilon \Be_{z}\right),\\
&\nabla \cdot \BU^\varepsilon =0,\\
& \partial_t \Psi^\varepsilon + \Bu^\varepsilon\cdot\nabla \Psi^\varepsilon + \BU^\varepsilon\cdot\nabla \psi^\varepsilon  +  \phi_b'\BU^\varepsilon \cdot \Be_z   - \nabla\cdot(D^\varepsilon\nabla\Psi^\varepsilon)=0
\end{aligned}\right.
\end{equation}
and the boundary  conditions
\begin{equation}
\BU^\varepsilon\cdot\Bn|_{z=\pm H}= \Psi^\varepsilon|_{z=\pm H}= 0.
\end{equation}
The interface conditions for $\Psi^\varepsilon,\BU^\varepsilon, P^\varepsilon$ are the same as those for $\psi^\varepsilon, u^\varepsilon, p^\varepsilon$, respectively. 
% the initial conditions
% \begin{equation}
%   \Psi^\varepsilon|_{t=T_1+1}= \partial_x\psi^\varepsilon|_{t=T_1+1}, 
% \end{equation}
% and 
% Differentiating \eqref{eq:Psieq}$_3$ with respect to $x$ (recall that $\phi_b$ and $D^\varepsilon$ do not depend on $x$) yields
% \begin{equation}\label{eq:Psieq}
% \partial_t\Psi^\varepsilon + u^\varepsilon\cdot\nabla\Psi^\varepsilon + U^\varepsilon\cdot\nabla\psi^\varepsilon + \phi_b'U^\varepsilon_z = \nabla\cdot(D^\varepsilon\nabla\Psi^\varepsilon),
% \end{equation}
% supplemented with homogeneous Dirichlet boundary conditions at $z=\pm H$ and continuity of $\Psi^\varepsilon$ and the diffusive flux at the interfaces.
% Take $t_0:=T_1+1$; by \eqref{eq:P42} we have $\Psi^\varepsilon(t_0)\in L^2(\Omega)$ with $\|\Psi^\varepsilon(t_0)\|_{L^2}\le C_1$. 
Testing \eqref{Sharps-x} with $\Psi^\varepsilon$ and using integration by parts, we obtain
\begin{equation}\label{H^1-x-5}
\frac12\frac{\dd }{\dd t}\|\Psi^\varepsilon\|_{L^2(\Omega)}^2 + \|\sqrt{D^\varepsilon}\nabla\Psi^\varepsilon\|_{L^2(\Omega)}^2
=    (\BU^\varepsilon\cdot\nabla \Psi^\varepsilon, \psi^\varepsilon) - (\phi_b'U^\varepsilon_z,\Psi^\varepsilon). 
\end{equation}
Similar to \eqref{L^2-6}, one also has 
\begin{equation}\label{H^1-x-6}
  |(\phi_b'U^\varepsilon_z,\Psi^\varepsilon)| \le 2c_\Delta \delta \max_i  K_i    \|\nabla \Psi^\varepsilon\|_{L^2(\Omega)}^2 \leq \frac{1}{4}\min_i  D_i \|\nabla \Psi^\varepsilon\|_{L^2(\Omega)}^2 .
\end{equation}
Moreover, similar to \eqref{H^1-7-1}-\eqref{H^1-7} and \eqref{H^1-7-3} , it holds that 
\begin{equation}\label{H^1-x-7-1}
\|\BU^\varepsilon\|_{L^2(\Omega_{i,\varepsilon})} \leq CK_i^\frac12\| \Psi^\varepsilon\|_{L^2(\Omega)},\quad \|\nabla \BU^\varepsilon\|_{L^2(\Omega_{i,\varepsilon})} \leq CK_i^\frac12\|\nabla\Psi^\varepsilon\|_{L^2(\Omega)}
\end{equation}
and
\begin{equation}\label{H^1-x-7}
\|\BU^\varepsilon\|_{L^4(\Omega_{i,\varepsilon})} \leq C\|\nabla\Psi^\varepsilon\|_{L^2(\Omega)}^\frac12\|\Psi^\varepsilon\|_{L^2(\Omega)}^\frac12 ,
\end{equation}
By Lemma \ref{lemmaA1}, Young's inequality and \eqref{H^1-x-7}, one has 
\begin{equation}\label{H^1-x-8}
\begin{aligned}
| (\BU^\varepsilon\cdot\nabla \Psi^\varepsilon, \psi^\varepsilon) | \leq&~ \|\nabla\Psi^\varepsilon\|_{L^2(\Omega)}\|\BU^\varepsilon\|_{L^4(\Omega)} \|\psi^\varepsilon\|_{L^4(\Omega)}  \\
\leq&~ C \|\nabla\Psi^\varepsilon\|_{L^2(\Omega)} ^\frac32 \| \Psi^\varepsilon\|_{L^2(\Omega)}^\frac12 \|\psi^\varepsilon\|_{L^2(\Omega)}^{\frac12}\|\nabla \psi^\varepsilon\|_{L^2(\Omega)} ^\frac12\\ 
\leq&~ \frac14\min_i D_i\|\nabla\Psi^\varepsilon\|_{L^2(\Omega)}^2 + C\|\Psi^\varepsilon\|_{L^2(\Omega)}^2 \|\psi^\varepsilon\|_{L^2(\Omega)}^2\|\nabla \psi^\varepsilon\|_{L^2(\Omega)} ^2\\
\leq&~ \frac14\min_i D_i\|\nabla\Psi^\varepsilon\|_{L^2(\Omega)}^2 + C \|\psi^\varepsilon\|_{L^2(\Omega)}^2\|\nabla \psi^\varepsilon\|_{L^2(\Omega)} ^4.
\end{aligned}
\end{equation}
Inserting \eqref{H^1-x-6} and \eqref{H^1-x-8} into \eqref{H^1-x-5}, we arrive at
\begin{equation}\label{H^1-x-9}
\frac{\dd}{\dd t}\|\Psi^\varepsilon\|_{L^2}^2 + \|\sqrt{D^\varepsilon}\nabla\Psi^\varepsilon\|_{L^2}^2 \le C \|\psi^\varepsilon\|_{L^2(\Omega)}^2\|\nabla \psi^\varepsilon\|_{L^2(\Omega)} ^4 .
\end{equation}
% This, together with the uniform estimates give in Proposition \ref{L^2}-\ref{H^1}, yields 
% \begin{equation}\label{eq:PsiL2final}
% \frac{\dd}{\dd t}\|\Psi^\varepsilon\|_{L^2}^2 + \|\sqrt{D^\varepsilon}\nabla\Psi^\varepsilon\|_{L^2}^2  \leq C M_5^2\left(\frac{M_1H^2}{\min_i D_i}+1\right),\quad \text{ for any }t\ge T_1+1.
% \end{equation}
Integrating \eqref{H^1-x-9} from $t$ to $t+1$ for any $t\ge T_1+1$ and using the uniform estimates given in Propositions \ref{L^2}-\ref{H^1} gives
\begin{equation}\label{eq:Psiint}
\begin{aligned}
  &~\int_t^{t+1} \|\sqrt{D^\varepsilon}\nabla\Psi^\varepsilon(s)\|_{L^2}^2\dd s \\
  \le &~ C\int_t^{t+1}\|\psi^\varepsilon(s)\|_{L^2(\Omega)}^2\|\nabla \psi^\varepsilon(s)\|_{L^2(\Omega)} ^4 + \|\partial_x \psi^\varepsilon(t)\|_{L^2}^2  \\
  \le&~  C M_5^2\left(\frac{M_1H^2}{\min_i D_i}+1\right) .
\end{aligned}
\end{equation}
 
{\bf Step 2. $H^1$ estimate for $\partial_x\psi^\varepsilon$.}
Testing \eqref{Sharps-x}$_3$ with $\L^\varepsilon\Psi^\varepsilon$ and applying integration by parts, we obtain
\begin{equation}\label{eq:PsiH2}
\begin{aligned} 
&~\frac12\frac{\dd}{\dd t}\|\sqrt{D^\varepsilon}\nabla\Psi^\varepsilon\|_{L^2(\Omega)}^2 + \|\L^\varepsilon\Psi^\varepsilon\|_{L^2(\Omega)}^2\\
= & -(\Bu^\varepsilon\cdot\nabla\Psi^\varepsilon, \L^\varepsilon\Psi^\varepsilon) 
   - (\BU^\varepsilon\cdot\nabla\psi^\varepsilon, \L^\varepsilon\Psi^\varepsilon)- (\phi_b'U^\varepsilon_z, \L^\varepsilon\Psi^\varepsilon).
\end{aligned}
\end{equation}
Now we estimate each term in \eqref{eq:PsiH2}. By \eqref{H^1-7} and Lemmas \ref{lemmaA1}-\ref{lemma-equivalence}, one has 
\begin{equation} 
\begin{aligned}
|(\Bu^\varepsilon\cdot\nabla\Psi^\varepsilon, \L^\varepsilon\Psi^\varepsilon) | \leq&~ \|\Bu^\varepsilon\|_{L^4}\|\nabla\Psi^\varepsilon\|_{L^4}\|\L^\varepsilon\Psi^\varepsilon\|_{L^2} \\
      \leq&~ C\|\nabla\psi^\varepsilon\|_{L^2(\Omega)} ^\frac12\|\psi^\varepsilon\|_{L^2(\Omega)} ^\frac12 \|\nabla\Psi^\varepsilon\|_{L^2}^{\frac12} \|\L^\varepsilon\Psi^\varepsilon\|_{L^2}^\frac32 \\
      \leq&~ \frac18\|\L^\varepsilon\Psi^\varepsilon\|_{L^2}^2 + C\|\nabla\Psi^\varepsilon\|_{L^2}^2 \|\nabla\psi^\varepsilon\|_{L^2(\Omega)} ^2\|\psi^\varepsilon\|_{L^2(\Omega)} ^2.
\end{aligned}
\end{equation}
Similarly, by \eqref{H^1-x-7} and Lemmas \ref{lemmaA1}-\ref{lemma-equivalence}, we have
\begin{equation}\nonumber
\begin{aligned}
| (\BU^\varepsilon\cdot\nabla\psi^\varepsilon, \L^\varepsilon\Psi^\varepsilon) | \leq &~ \| \BU^\varepsilon\|_{L^4} \|\nabla\psi^\varepsilon\|_{L^4} \|\L^\varepsilon\Psi^\varepsilon\|_{L^2}\\
\leq&~\| \Psi^\varepsilon\|_{L^2} ^\frac12\|\nabla \Psi^\varepsilon\|_{L^2} ^\frac12 \| \nabla \psi^\varepsilon\|_{L^2} ^\frac12 \|\L^\varepsilon\psi^\varepsilon\|_{L^2}^\frac12 \|\L^\varepsilon\Psi^\varepsilon\|_{L^2} \\
\leq&~\|\nabla \Psi^\varepsilon\|_{L^2} ^\frac12 \| \nabla \psi^\varepsilon\|_{L^2} \|\L^\varepsilon\psi^\varepsilon\|_{L^2}^\frac12 \|\L^\varepsilon\Psi^\varepsilon\|_{L^2} \\
% \leq&~ \frac18\|\L^\varepsilon\Psi^\varepsilon\|_{L^2}^2 + C \|\nabla\Psi^\varepsilon\|_{L^2} \|\nabla\psi^\varepsilon\|_{L^2} ^2\|\L^\varepsilon\psi^\varepsilon\|_{L^2}\\
\leq&~ \frac18\|\L^\varepsilon\Psi^\varepsilon\|_{L^2}^2 + C \|\nabla\Psi^\varepsilon\|_{L^2}^2+  \|\nabla\psi^\varepsilon\|_{L^2} ^4\|\L^\varepsilon\psi^\varepsilon\|_{L^2}^2.
\end{aligned}
\end{equation}
The estimate for the last term is the same as \eqref{H^1-5} in the proof of Proposition \ref{H^1},
\begin{equation}  
    \left|(\phi_b' U_z^\varepsilon, \mathcal{L}^\varepsilon \Psi^\varepsilon)\right| 
    \leq   \frac{1}{4}\|\mathcal{L}^\varepsilon \Psi^\varepsilon\|_{L^2}^2+ C \|\nabla\Psi^\varepsilon\|_{L^2}^2. 
\end{equation}
Substituting these estimates into \eqref{eq:PsiH2} yields the differential inequality
\begin{equation}\label{eq:PsiH2diff}
\frac{\dd}{\dd t} Y(t) +\frac12 \|\L^\varepsilon\Psi^\varepsilon\|_{L^2}^2 \le G(t) Y(t) + H(t), 
\end{equation}
where we have set
\begin{align*}
Y(t) &:= \|\sqrt{D^\varepsilon}\nabla\Psi^\varepsilon(t)\|_{L^2}^2,\\
G(t) &:= C(1 +  \|\nabla\psi^\varepsilon\|_{L^2} ^2\|\psi^\varepsilon\|_{L^2(\Omega)} ^2),\\
H(t) &:= \|\nabla\psi^\varepsilon\|_{L^2} ^4\|\L^\varepsilon\psi^\varepsilon\|_{L^2}^2.
\end{align*}
In view of Propositions \ref{L^2}-\ref{H^1} , we have the uniform integral bounds
\begin{equation}\label{eq:ghbds}
\int_t^{t+1} G(s)\dd s \le C +CM_5\left(\frac{M_1H^2}{\min_i D_i}+1\right) ,\qquad \int_t^{t+1} H(s)\dd s \le C M_5^2M_6, 
\end{equation}
for all $t\ge T_1+1$. Recall also from \eqref{eq:Psiint} and \eqref{eq:ghbds} that 
\[ \int_t^{t+1} Y(s)\dd s\le C M_5^2\left(\frac{M_1H^2}{\min_i D_i}+1\right) \text{  for any }t\ge T_1+1.\] Applying the uniform Gr\"onwall lemma (\cite{Temam1997}) to \eqref{eq:PsiH2diff} on the interval, we conclude that for any $t\ge T_1+1$,
\begin{equation}\label{M9}
\begin{aligned}
Y(t+1)
 \le&~ \left[C M_5^2\left(\frac{M_1H^2}{\min_i D_i}+1\right) + M_5^2M_6\right] \\
& \times \exp\left[C +CM_5\left(\frac{M_1H^2}{\min_i D_i}+1\right)\right]
=: M_9.
\end{aligned}
\end{equation}
Then we obtain \eqref{H^1-x-0}.   

{\bf Step 3. Global-in-time estimates of $\partial_x\psi^\varepsilon$.}
Finally, if $\psi^\varepsilon_{\rm in},\partial_x\psi^\varepsilon_{\rm in}\in H_0^1(\Omega)$,  one can apply Gr\"onwall inequality to the differential inequality \eqref{eq:PsiH2diff} on $[0,t]$ to obtain 
\begin{equation}\nonumber
\begin{aligned}
 Y(t) \leq&~ \left(Y(0) + \int_0^t H(s)\dd s\right)\exp\left(\int_0^t G(s) \dd s\right)   \\
       \leq&~ \left(\|\sqrt{D^\varepsilon}\nabla \partial_x \psi_{\rm in}^\varepsilon\|_{L^2(\Omega)}^2+ \int_0^t  \|\nabla\psi^\varepsilon\|_{L^2} ^4\|\L^\varepsilon\psi^\varepsilon\|_{L^2}^2 \dd s\right)\\
       &\times \exp\left(\int_0^t C(1 +  \|\nabla\psi^\varepsilon\|_{L^2(\Omega)} ^2\|\psi^\varepsilon\|_{L^2(\Omega)} ^2) \dd s\right)   \\ 
    \leq&~ (\|\sqrt{D^\varepsilon}\nabla \partial_x \psi_{\rm in}^\varepsilon\|_{L^2(\Omega)}^2+ M_7^2M_8)\\
       &\times \exp\left[Ct +C\left(\frac{M_1H^2}{\min_i D_i}+\|\psi_{\rm in}^\varepsilon\|_{L^2(\Omega)}^2\right)\left(M_1t + \|\psi_{\rm in}^\varepsilon\|_{L^2(\Omega)}^2\right)\right] \\
    =&:M_{10}'(t,\|\psi_{\rm in}^\varepsilon\|_{H^1(\Omega)},\|\partial_x\psi_{\rm in}^\varepsilon\|_{H^1(\Omega)}).
\end{aligned}
\end{equation}
This, together with the long time estimate \eqref{M9}, gives \eqref{H^1-x-0-1} with 
\begin{equation}\label{M10}
  M_{10}:= \max\{M_9, M_{10}'(T_1+2,\|\psi_{\rm in}^\varepsilon\|_{H^1(\Omega)},\|\partial_x\psi_{\rm in}^\varepsilon\|_{H^1(\Omega)})\}.
\end{equation}

The case $\varepsilon=0$ and $\alpha=\infty$ can be proved in the same way. The proof for the case $\varepsilon=0$ and $\alpha\in [0,\infty)$ is similar since the estimates for $\BU^0$ and $P^0$ follows from the estimate \eqref{pressure-limit-est} for the pressure transmission problem \eqref{pressure-limit}. So we omit the details and finish the proof of this proposition.
\end{proof}

The following proposition gives the uniform-in-$\varepsilon$ estimates for $p^\varepsilon$ and $\Bu^\varepsilon$.
\begin{proposition}\label{uniform-p}
  Let $\delta$ be chosen so that \eqref{delta_2} is satisfied and $\alpha\in[0,\infty]$. 
For any $\varepsilon\in (0,1]$, let $(\Bu^\varepsilon,\psi^\varepsilon,p^\varepsilon)$ be a solution to the problem \eqref{Sharps}-\eqref{bc}. Then there exists a constant $C$ independent of $\varepsilon$, such that
\begin{equation}\label{uniform-u-1}
  \|\Bu^\varepsilon\|_{L^2(\Omega_{i,\varepsilon})}   \leq   C K_i^\frac12\|\psi^\varepsilon\|_{L^2(\Omega)},\quad
  \|\nabla\Bu^\varepsilon\|_{L^2(\Omega_{i,\varepsilon})}   \leq   CK_i^\frac12\|\nabla \psi^\varepsilon\|_{L^2(\Omega)}  
\end{equation}
and 
\begin{equation}\label{uniform-u-2}
  \|\Bu^\varepsilon\|_{L^4(\Omega_{i,\varepsilon})}   \leq   C \|\psi^\varepsilon\|_{L^2(\Omega)}^\frac12\|\nabla \psi^\varepsilon\|_{L^2(\Omega)}^\frac12, 
\end{equation}
for $i=+,-,0$. Furthermore, we have
\begin{equation}\label{uniform-p-1}
  \|p^\varepsilon\|_{H^1(\Omega_{\pm,\varepsilon})}   \leq   C\|\psi^\varepsilon\|_{L^2(\Omega)},\quad   \| p^\varepsilon\|_{H^2(\Omega_{\pm,\varepsilon})}   \leq   C \|\nabla\psi^\varepsilon\|_{L^2(\Omega)}.
\end{equation}   
% \begin{equation}\label{uniform-p-1-1}
%   \|p^\varepsilon\|_{L^2(\Omega_{0,\varepsilon})}   \leq   Ch_0^\frac12 \|\psi^\varepsilon\|_{L^2(\Omega)},
% \end{equation} 
% \begin{equation}\label{uniform-p-1-2}
 
% \end{equation}
% and
% \begin{equation}\label{uniform-p-1-3}
%   \|\partial_x p^\varepsilon\|_{L^2(\Omega_{0,\varepsilon})}   \leq    Ch_0^\frac12 \|\partial_x\psi^\varepsilon\|_{L^2(\Omega)}.
% \end{equation}
Finally, for the solution  $(\Bu^0,\psi^0,p^0)$ to the limit problem \eqref{Sharps-limit}, \eqref{initial-limit}-\eqref{interface1-limit}, we have 
\begin{equation}\label{uniform-limit-p-1}
  \|\Bu^0\|_{L^2(\Omega_{\pm})}   \leq   C\|\psi^0\|_{L^2(\Omega)},\quad
  \|\nabla\Bu^0\|_{L^2(\Omega_{\pm})}   \leq   C\|\nabla \psi^0\|_{L^2(\Omega)},
\end{equation} 
\begin{equation}\label{uniform-limit-p-2}
  \|\Bu^0\|_{L^4(\Omega_{\pm})}   \leq   C \|\nabla \psi^0\|_{L^2(\Omega)}^\frac12\|\psi^0\|_{L^2(\Omega)}^\frac12
\end{equation}  
and 
\begin{equation}\label{uniform-limit-p-4}
   \|p^0\|_{H^1(\Omega_{\pm})}   \leq   C\|\psi^0\|_{L^2(\Omega)},\quad \| p^0\|_{H^2(\Omega_{\pm})}   \leq   C \|\nabla\psi^0\|_{L^2(\Omega)}.
\end{equation} 
\end{proposition}
\begin{proof} 
The estimates \eqref{uniform-u-1}-\eqref{uniform-u-2} for $\Bu^\varepsilon$ are exactly given by \eqref{H^1-7-1}-\eqref{H^1-7} and \eqref{H^1-7-3}. Then \eqref{uniform-limit-p-1}-\eqref{uniform-limit-p-2} can be proved similarly.  Now we prove the estimates for $p^\varepsilon$ with $\varepsilon>0$. To this end, we write the Fourier expansion in the horizontal variable $x$ as
\begin{equation}\nonumber
  p^\varepsilon(x,z) = p_0^\varepsilon(z) +  \sum_{n\in \mathbb{Z},n\neq 0} p^\varepsilon_n(z) e^{\ii n \frac{2\pi}{L} x}=: p_0^\varepsilon(z) + p_{\neq}^\varepsilon(x,z),
\end{equation}
where
\begin{equation}\nonumber
  p^\varepsilon_n (z) = \frac{1}{L}\int_{0}^{L} p^\varepsilon(x,z) e^{- \ii n \frac{2\pi}{L}x} \dd x.
\end{equation}
Similarly, we write also 
\begin{equation}\nonumber
  \psi^\varepsilon(x,z) = \psi^\varepsilon_0(z) + \psi^\varepsilon_{\neq}(x,z)
\end{equation}
and
\begin{equation}\nonumber
  \Bu^\varepsilon = \Bu^\varepsilon_0(z) + \Bu^\varepsilon_{\neq}(x,z)= u^\varepsilon_{x,0}(z) \Be_x+ u^\varepsilon_{z,0}(z)\Be_z+ u^\varepsilon_{x,\neq}(x,z) \Be_x + u^\varepsilon_{z,\neq}(x,z) \Be_z .
\end{equation}

For the zero mode, it follows from \eqref{pressure} that $p^\varepsilon_0$ satisfies
\begin{equation}\nonumber
  \left\{
    \begin{aligned}
&-(K^\varepsilon(p^\varepsilon_{0})')' = (K^\varepsilon\psi^\varepsilon_0)',\quad \text{ in }(-H,H)\\
&(p^\varepsilon_0)'(\pm H)= 0.
    \end{aligned}\right.
  \end{equation}
  Noting that $K^\varepsilon$ is constant in each layer, then $K^\varepsilon((p^\varepsilon_0)'+ \psi^\varepsilon_0)$ is piecewise constant. Since $K^\varepsilon((p^\varepsilon_0)'+ \psi^\varepsilon_0)$ is continuous at any $z= \pm \tfrac{h_0}{2}$, and vanishes at the boundary $z=\pm H$, we conclude that
  \begin{equation}\label{uniform-p-2}
    (p^\varepsilon_0)'+  \psi^\varepsilon_0=0.
  \end{equation}
  On the other hand, since $p^\varepsilon$ has zero mean value, one has
  \begin{equation}\label{uniform-p-2-1}
    0 = \int_\Omega p^\varepsilon \dd x\dd z  = L\int_{-H}^H p^\varepsilon_0\dd z.
  \end{equation}
  Then an application of Poincar\'e inequality gives
  \begin{equation}\label{uniform-p-3}
    \| p^\varepsilon_0 \|_{H^1(\Omega)} \leq  C\| (p^\varepsilon_0 )'\|_{L^2(-H,H)}  \leq C \|\psi^\varepsilon_0\|_{L^2(-H,H)} \leq C \|\psi^\varepsilon_0\|_{L^2(\Omega)}.
  \end{equation}

  Furthermore, according to \eqref{pressure} and \eqref{Sharps}$_1$, the nonzero mode $p^\varepsilon_{\neq}$ satisfies
  \begin{equation}  \nonumber
    \left\{
\begin{aligned}
&-\nabla\cdot(K^\varepsilon \nabla p^\varepsilon_{\neq} ) = \nabla \cdot(K^\varepsilon \psi^\varepsilon_{\neq}\Be_z)   , ~~~~&& \text{ in } \Omega_{\pm,\varepsilon}\\
&  \partial_z p^\varepsilon_{\neq}= -\psi^\varepsilon_{\neq}  -\frac{1}{K_\pm } u^\varepsilon_{z,\neq},  &&\text{ on } \partial\Omega_{\pm,\varepsilon}.
\end{aligned}\right.
    \end{equation}
    Noting that
    \begin{equation}\nonumber
\int_{0}^{L} p^\varepsilon_{\neq}(x,z) \dd x = 0,\quad \text{  for any }z\in (-H,H),
    \end{equation}
    one has
    \begin{equation}\nonumber
\int_{\Omega_{\pm,\varepsilon}} p^\varepsilon_{\neq}  \dd x\dd z = 0.
    \end{equation}
    Since \(0<h_0\le H/2\), the rectangles \(\Omega_{\pm,\varepsilon}\) have heights uniformly bounded above and below. Hence the Poincaré, trace, and elliptic constants below can be chosen independently of \(\varepsilon\). Then it follows from the standard elliptic estimates and \eqref{uniform-u-1} that
    \begin{equation}\label{uniform-p-4}
\begin{aligned}
\| p^\varepsilon_{\neq}\|_{H^1(\Omega_{\pm,\varepsilon})} \leq&~C\| \nabla p^\varepsilon_{\neq}\|_{L^2(\Omega_{\pm,\varepsilon})} \\
\leq &~C(\|\psi^\varepsilon_{\neq}\|_{L^2(\Omega_{\pm,\varepsilon})} +\|u^\varepsilon_{z,\neq}\|_{H^{-\frac12} (\partial\Omega_{\pm,\varepsilon})})\\
\leq&~ C(\|\psi^\varepsilon_{\neq}\|_{L^2(\Omega_{\pm,\varepsilon})}+\|\Bu^\varepsilon_{\neq}\|_{L^2(\Omega_{\pm ,\varepsilon})}) \\
\leq&~ C \|\psi^\varepsilon \|_{L^2(\Omega)}.
\end{aligned}
\end{equation} Combining \eqref{uniform-p-3} and \eqref{uniform-p-4}, one obtains the $H^1$ estimate given in \eqref{uniform-p-1}. 

% By using  Poinc\'are inequality, trace theorem, \eqref{H^1-7-0} and \eqref{uniform-p-1}, one has
%   \begin{equation}\nonumber
%   \begin{aligned}
%    \|p^\varepsilon\|_{L^2(\Omega_{0,\varepsilon})} \leq &~ \left\|p^\varepsilon - p^\varepsilon\left(x,\tfrac{h_0}{2}\right)\right\|_{L^2(\Omega_{0,\varepsilon})} + \left\| p^\varepsilon\left(x,\tfrac{h_0}{2}\right)\right\|_{L^2(\Omega_{0,\varepsilon})} \\
%    \leq&~h_0 \left\|\partial_zp^\varepsilon \right\|_{L^2(\Omega_{0,\varepsilon})} + h_0^\frac12 \left\|p^\varepsilon\left(x,\tfrac{h_0}{2}\right)\right\|_{L^2(0,L)} \\
%    \leq&~ h_0 K_0^{-\frac12} \left\|\sqrt{K^\varepsilon}\partial_zp^\varepsilon \right\|_{L^2(\Omega_{0,\varepsilon})} + h_0^\frac12 \left\|p^\varepsilon \right\|_{H^1(\Omega_{+,\varepsilon})} \\
%   \leq&~C h_0^\frac12\left\|\psi^\varepsilon \right\|_{L^2(\Omega)}.
%   \end{aligned}
%   \end{equation}
Furthermore, since $\partial_x p^\varepsilon$ satisfies the elliptic equation \eqref{pressure_x} and $\int_{\Omega} \partial_x p^\varepsilon  \dd \Bx = 0$, we can apply the same argument as above to deduce that 
\begin{equation}\label{uniform-p-5}
    \|\partial_x p^\varepsilon\|_{H^1(\Omega_{\pm,\varepsilon})}   \leq   C\|\partial_x\psi^\varepsilon\|_{L^2(\Omega)}.
\end{equation}
% and 
% \begin{equation}\label{uniform-p-6}
%   \|\partial_x p^\varepsilon\|_{L^2(\Omega_{\pm,\varepsilon})}   \leq   C h_0^\frac12\|\partial_x\psi^\varepsilon\|_{L^2(\Omega)}.
% \end{equation}
Then, by virtue of the elliptic equation \eqref{pressure} and \eqref{uniform-p-5}, we have
\begin{equation}\nonumber 
 \|\sqrt{K^\varepsilon}\partial_z^2 p^\varepsilon\|_{L^2(\Omega_{\pm ,\varepsilon})} \leq  \|\sqrt{K^\varepsilon}\partial_{x}^2 p^\varepsilon\|_{L^2(\Omega_{\pm,\varepsilon}) }+ \|\sqrt{K^\varepsilon}\partial_{z}\psi^\varepsilon\|_{L^2(\Omega_{\pm,\varepsilon}) }
  \leq C\| \nabla\psi^\varepsilon\|_{L^2(\Omega) },
\end{equation}
and hence,
\begin{equation}\nonumber
\|p^\varepsilon\|_{H^2(\Omega_{\pm,\varepsilon})} \leq C\left(\|p^\varepsilon\|_{H^1(\Omega_{\pm,\varepsilon})}+\|\partial_x p^\varepsilon\|_{H^1(\Omega_{\pm,\varepsilon})} + \|\partial_{z}^2p^\varepsilon\|_{L^2(\Omega_{\pm,\varepsilon})}\right) \leq C\|\nabla \psi^\varepsilon\|_{L^2(\Omega)}.
\end{equation}
The case $\varepsilon=0$ and $\alpha=\infty$ can be proved in the same way. 

For \(\varepsilon=0\) and \(0\le\alpha<\infty\), incompressibility implies that the horizontal mean normal velocity is constant in each outer layer. The no-flux boundary condition gives \(u_{z,0}^0\equiv0\). Consequently, \((p_0^0)'+\psi_0^0=0\) in \((-H,0)\) and \((0,H)\). For \(\alpha>0\), the transmission law then gives \(\llbracket p_0^0\rrbracket_\Gamma=0\); for \(\alpha=0\), the same conclusion follows from the chosen mean-jump normalization. Together with the global mean normalization, this yields the required \(H^1\) estimate for \(p_0^0\)
%Finally, for the case $\varepsilon=0$ and $\alpha\in [0,\infty)$, the zero mode $p^0_0$ satisfies the corresponding pressure transmission problem 
%\begin{equation}\nonumber
%  \left\{
%    \begin{aligned}
%&-(K^0(p^0_{0})')' = (K^0\psi^0_0)',&& \text{ in }(-H,H)\\
%&(p^0_0)' = 0,&& \text{ on } z=\pm H,\\
%&(p^0_0)' = \frac{\alpha}{K}\llbracket p^0_0 \rrbracket_\Gamma-\psi^0,&& \text{ on } z=0,\\
%&\int_{0}^L \llbracket p^0_0 \rrbracket_\Gamma\dd x =0,
%    \end{aligned}\right.
%  \end{equation} 
and the nonzero mode $p^0_{\neq}$ satisfies
\begin{equation}  \nonumber
    \left\{
\begin{aligned}
&-\nabla\cdot(K^0 \nabla p^0_{\neq} ) = \nabla \cdot(K^0 \psi^0_{\neq}\Be_z)   , ~~~~&& \text{ in } \Omega_{\pm}\\
&  \partial_z p^0_{\neq}= -\psi^0_{\neq}  -\frac{1}{K_\pm } u^0_{z,\neq},  &&\text{ on } \partial\Omega_{\pm}.
\end{aligned}\right.
    \end{equation}
Then the estimates \eqref{uniform-limit-p-4} for $p^0$ can be proved in a way similar to the case $\varepsilon>0$. We omit the details and finish the proof of this proposition.
\end{proof} 

Note that $p^\varepsilon$ and $\Bu^\varepsilon$ are determined by $\psi^\varepsilon$ through the elliptic equations \eqref{pressure} (or \eqref{pressure-limit} when $\varepsilon=0$) and  \eqref{Sharps}$_1$ (or \eqref{Sharps-limit} when $\varepsilon=0$) , which are linear equations  with coefficients independent of $x$. Then one can take the horizontal derivative $\partial_x$ on these equations and obtain the corresponding estimates for $\partial_x p^\varepsilon$ and $\partial_x \Bu^\varepsilon$ in a similar way. 
\begin{proposition}\label{uniform-p-x}
    Let $\delta$ be chosen so that \eqref{delta_2} is satisfied and $\alpha\in[0,\infty]$. 
For any $\varepsilon\in (0,1]$, let $(\Bu^\varepsilon,p^\varepsilon,\psi^\varepsilon)$ be a solution to the problem \eqref{Sharps}-\eqref{bc}. Then there exists a constant $C$ independent of $\varepsilon$, such that
\begin{equation}\label{uniform-u-1-x}
  \|\partial_x\Bu^\varepsilon\|_{L^2(\Omega_{i,\varepsilon})}   \leq   C K_i^\frac12\|\partial_x\psi^\varepsilon\|_{L^2(\Omega)},~~
  \|\nabla\partial_x\Bu^\varepsilon\|_{L^2(\Omega_{i,\varepsilon})}   \leq   CK_i^\frac12 \|\nabla \partial_x\psi^\varepsilon\|_{L^2(\Omega)}  
\end{equation}
and 
\begin{equation}\label{uniform-u-2-x}
  \|\partial_x\Bu^\varepsilon\|_{L^4(\Omega_{i,\varepsilon})}   \leq   C  \|\nabla \partial_x\psi^\varepsilon\|_{L^2(\Omega)}^\frac12\|\partial_x\psi^\varepsilon\|_{L^2(\Omega)}^\frac12, 
\end{equation}
for $i=+,-,0$. Furthermore, we have
\begin{equation}\label{uniform-p-1-x}
  \|\partial_x p^\varepsilon\|_{H^1(\Omega_{\pm,\varepsilon})}   \leq   C\|\partial_x\psi^\varepsilon\|_{L^2(\Omega)},\quad   \|\partial_x p^\varepsilon\|_{H^2(\Omega_{\pm,\varepsilon})}   \leq   C \|\nabla\partial_x\psi^\varepsilon\|_{L^2(\Omega)}.
\end{equation}    
Finally, for the solution  $(\Bu^0,p^0,\psi^0)$ to the limit problem \eqref{Sharps-limit}, \eqref{initial-limit}-\eqref{interface1-limit}, we have 
\begin{equation}\label{uniform-limit-p-1-x}
  \|\partial_x\Bu^0\|_{L^2(\Omega_{\pm})}   \leq   C\|\partial_x\psi^0\|_{L^2(\Omega)},\quad
  \|\nabla\partial_x\Bu^0\|_{L^2(\Omega_{\pm})}   \leq   C\|\nabla \partial_x\psi^0\|_{L^2(\Omega)},
\end{equation} 
\begin{equation}\label{uniform-limit-p-2-x}
  \|\partial_x\Bu^0\|_{L^4(\Omega_{\pm})}   \leq   C \|\nabla \partial_x\psi^0\|_{L^2(\Omega)}^\frac12\|\partial_x\psi^0\|_{L^2(\Omega)}^\frac12
\end{equation}  
and 
\begin{equation}\label{uniform-limit-p-4-x}
   \|\partial_x p^0\|_{H^1(\Omega_{\pm})}   \leq   C\|\partial_x\psi^0\|_{L^2(\Omega)},\quad \|\partial_x p^0\|_{H^2(\Omega_{\pm})}   \leq   C \|\nabla\partial_x\psi^0\|_{L^2(\Omega)}.
\end{equation} 
\end{proposition}

\clearpage

\section{Convergence}
This section justifies the effective interface models at two levels. We first establish quantitative convergence of solutions at each fixed positive time, distinguishing the three regimes determined by the limiting hydraulic conductance. We then combine these estimates with the uniform bounds from Section 4 to prove upper semicontinuity of the global attractors, which connects the finite-time approximation with the long-time dynamics of the limiting interface problem.
\subsection{Finite-time convergence} 

% We normalize both $p^\varepsilon$ and $p^0$ by
% \[
% \int_\Omega p^\varepsilon\dd x
% =
% \int_\Omega p^0\dd x
% =0.
% \]

% We shall use the following intermediate-layer estimates established
% above:
% \begin{align}
% \|u^\varepsilon\|_{L^2(\Omega_{0,\varepsilon})}
% &\leq
% C K_0^{\frac12}
% \|\psi^\varepsilon\|_{L^2(\Omega)},
% \label{eq:middle-u-L2-alpha-infinity}
% \\
% \|u^\varepsilon\|_{L^4(\Omega_{0,\varepsilon})}
% &\leq
% C
% \left(
% K_0^{\frac14}+h_0^{\frac14}
% \right)
% \|\psi^\varepsilon\|_{H^1(\Omega)}.
% \label{eq:middle-u-L4-alpha-infinity}
% \end{align}
% The second estimate follows by subtracting an interfacial lifting,
% applying the uniform Ladyzhenskaya inequality to the resulting
% zero-trace remainder, and estimating the lifting by Lemma~A.3.

\begin{theorem} \label{thm:convergence} 
Suppose that $\psi_{\rm in}^\varepsilon = \psi_{\rm in}^0$ for any $\varepsilon\in (0,1]$ and $\psi_{\rm in}^0,\partial_x \psi_{\rm in}^0 \in H_0^1(\Omega)$. Then the solution $(\Bu^\varepsilon, p^\varepsilon, \psi^\varepsilon)$ of problem \eqref{Sharps}--\eqref{bc} converges to the solution $(\Bu^0, p^0, \psi^0)$ of the limit problem \eqref{Sharps-limit}, \eqref{initial-limit}-\eqref{interface1-limit} in the following sense:
  \begin{equation}\nonumber
\|\Bu^\varepsilon-\Bu^0\|_{L^2(\Omega)} + \|\psi^\varepsilon-\psi^0\|_{L^2(\Omega)} + \| \nabla p^\varepsilon -\nabla p^0\|_{L^2(\Omega_{\pm,\varepsilon})} \leq M_{11} \mathcal{E}_\alpha^\frac12(\varepsilon),\quad \text{for any } t>0,
  \end{equation}
  where the constant $M_{11}= M_{11}(t,\|\psi_{\rm in}^0\|_{H^1(\Omega)},\|\partial_x\psi_{\rm in}^0\|_{H^1(\Omega)})$ is independent of $\varepsilon$, and 
 \begin{equation}
\mathcal{E}_\alpha(\varepsilon) :=\left\{ \begin{aligned}
     & h_0^\frac12 + K_0^\frac12+\left|\frac{h_0}{K_0} \right|^\frac12,&&\text{ if } \alpha =\infty,\\
   & h_0^\frac12 + K_0^\frac12+\left|\frac{K_0}{h_0} -\alpha \right|,&&\text{ if } \alpha\in [0,\infty).
  \end{aligned}\right.
 \end{equation}
\end{theorem}

\begin{proof} 
The proof is organized into seven steps. In each regime, we first estimate
the pressure error, use Darcy's law to control the velocity error, and then
close the scalar energy estimate. 
  
{\bf Step 1. Convergence of $p^\varepsilon$ when $\alpha=\infty$.} In the highly permeable interface regime
\begin{equation}\label{convergence-1}
h_0(\varepsilon)\to0,
\qquad
K_0(\varepsilon)\to0,
\qquad
\frac{K_0(\varepsilon)}{h_0(\varepsilon)}\to\infty
\quad\text{as }\varepsilon\to0,
\end{equation}
it follows that
\[
\mathcal E_\infty\to0
\qquad\text{as }\varepsilon\to0.
\]
Hence, without loss of generality,  we assume that
\[
0<\mathcal E_\infty\leq1.
\]
 As is discussed in Section \ref{sec.3.1}, in the regime \eqref{convergence-1}, the pressure of the limiting system satisfies the interfacial condition 
\begin{equation}\label{convergence-2}
\llbracket p^0\rrbracket_\Gamma=0.
\end{equation}
Hence, $p^0$ is identified with a function in $H^1(\Omega)$. 

Set
\begin{equation}\label{convergence-3}
\widetilde{\Bu}:=\Bu^\varepsilon-\Bu^0,\qquad \widetilde{p}:=p^\varepsilon-p^0,
\qquad
\widetilde{\psi}:=\psi^\varepsilon-\psi^0.
\end{equation}
Since $p^\varepsilon,p^0\in H^1(\Omega)$, we have $\widetilde{p}\in H^1(\Omega)$ and 
\[
\int_\Omega \widetilde{p}\dd \Bx=0.
\] 

The pressure $p^\varepsilon$ satisfies
\begin{equation}\label{convergence-4}
\int_\Omega K^\varepsilon \left( \nabla p^\varepsilon+\psi^\varepsilon e_z
\right) \cdot\nabla q\dd \Bx =0, \qquad \text{ for any } q\in H^1(\Omega),
\end{equation}
whereas $p^0$ satisfies
\begin{equation}\label{convergence-5}
\int_\Omega K^0 \left( \nabla p^0+\psi^0e_z \right) \cdot\nabla q\dd \Bx =0, \qquad\text{ for any }  q\in H^1(\Omega),
\end{equation}
% where
% \[
% K^0(x,z)
% =
% \begin{cases}
% K_+,&z>0,\\
% K_-,&z<0.
% \end{cases}
% \] 
Subtracting the two weak formulations and taking $q=\widetilde{p}$ give 
\begin{equation}\label{convergence-6}
\begin{aligned}
\int_\Omega K^\varepsilon|\nabla \widetilde{p}|^2\dd \Bx
= 
-\int_\Omega
K^\varepsilon\widetilde{\psi}\partial_z\widetilde{p}\dd\Bx -
\int_{\Omega_{0,\varepsilon}}
(K^\varepsilon-K^0)
(\nabla p^0+\psi^0e_z)\cdot\nabla \widetilde{p}\dd \Bx.
\end{aligned}
\end{equation}
Here we used the fact that
\[
K^\varepsilon=K^0
\qquad
\text{in }
\Omega_{+,\varepsilon}\cup\Omega_{-,\varepsilon}.
\] 
By Cauchy--Schwarz and Young inequalities, we have 
\begin{equation}
\left|
\int_\Omega
K^\varepsilon\widetilde{\psi}\partial_z\widetilde{p}\dd\Bx
\right|
\leq
\frac14
\int_\Omega
K^\varepsilon|\nabla \widetilde{p}|^2\dd\Bx
+
C\|\widetilde{\psi}\|_{L^2(\Omega)}^2.
\label{convergence-7}
\end{equation}
Using Lemma \ref{lem:thin-layer-L2} and Propositions \ref{H^1}, \ref{uniform-p}, we also have
\begin{equation}
\label{convergence-8}
\begin{aligned}
&~\left|
\int_{\Omega_{0,\varepsilon}}
(K^\varepsilon-K^0)
(\nabla p^0+\psi^0e_z)\cdot\nabla \widetilde{p}\dd \Bx
\right|\\
\leq&~
\frac14
\int_{\Omega_{0,\varepsilon}}
K_0|\nabla \widetilde{p}|^2\dd \Bx
+
\frac{C}{K_0}(
\|\nabla p^0\|_{L^2(\Omega_{0,\varepsilon})}^2+ 
\|\psi^0\|_{L^2(\Omega_{0,\varepsilon})}^2)\\
\leq&~\frac14
\int_{\Omega_{0,\varepsilon}}
K_0|\nabla \widetilde{p}|^2\dd \Bx
+
C\frac{h_0}{K_0}(
\| p^0\|_{H^2(\Omega_{+} \cup \Omega_{-})}^2+ 
\|\psi^0\|_{H^1(\Omega)}^2)\\
\leq&~\frac14
\int_{\Omega_{0,\varepsilon}}
K_0|\nabla \widetilde{p}|^2\dd \Bx
+
C\mathcal E_\infty^2
\|\psi^0\|_{H^1(\Omega)}^2\\
\leq&~\frac14
\int_{\Omega_{0,\varepsilon}}
K_0|\nabla \widetilde{p}|^2\dd \Bx
+
CM_7\mathcal E_\infty^2.
\end{aligned}
\end{equation}
 
Substituting
\eqref{convergence-7}-
\eqref{convergence-8} into
\eqref{convergence-6} gives
\begin{equation}
\label{convergence-9} 
\|\nabla \widetilde{p}\|_
{L^2(\Omega_{+,\varepsilon}\cup\Omega_{-,\varepsilon})}^2
+
K_0
\|\nabla \widetilde{p}\|_{L^2(\Omega_{0,\varepsilon})}^2 \leq
C\|\widetilde{\psi}\|_{L^2(\Omega)}^2
+
CM_7\mathcal E_\infty^2. 
\end{equation} 

On the other hand, note that $p^\varepsilon$ and $\Bu^\varepsilon$ are determined by $\psi^\varepsilon$ through the elliptic equations \eqref{pressure} (or \eqref{pressure-limit} when $\varepsilon=0$), which are linear equations  with coefficients independent of $x$. Then one can take the horizontal derivative $\partial_x$ on these equations and obtain the corresponding estimates for $\partial_x p^\varepsilon$ in a similar way. Consequently, we have
\begin{equation}
\label{convergence-10}
\begin{aligned} 
&~\|\nabla\partial_x\widetilde{p}\|_
{L^2(\Omega_{+,\varepsilon}\cup\Omega_{-,\varepsilon})}^2
+
K_0
\|\nabla\partial_x\widetilde{p}\|_
{L^2(\Omega_{0,\varepsilon})}^2 \\
\leq&~
C\|\partial_x\widetilde{\psi}\|_{L^2(\Omega)}^2
+
C\frac{h_0}{K_0}
\left(
\|\partial_x\psi^0\|_{H^1(\Omega)}^2
+
\|\partial_xp^0\|_
{H^2(\Omega_+\cup \Omega_-)}^2
\right) \\ 
% \leq&~C\|\partial_x\widetilde{\psi}\|_{L^2(\Omega)}^2 + C\frac{h_0}{K_0} 
% \|\partial_x\psi^0\|_{H^1(\Omega)}^2 \\ 
\leq&~C\|\nabla\widetilde{\psi}\|_{L^2(\Omega)}^2 + CM_{10} 
\mathcal E_\infty^2.
\end{aligned}
\end{equation} 
Since $\widetilde{p}$ satisfies the elliptic problem 
\begin{equation}\nonumber
-\Delta \widetilde{p} =\partial_z\widetilde{\psi},\quad \text{ in }\Omega_{\pm,\varepsilon},
\end{equation}
one has 
\begin{equation}\label{convergence-11}
\begin{aligned}
\|\nabla^2\widetilde{p}\|_
{L^2(\Omega_{\pm,\varepsilon})}^2
 \leq&~
C
\|\nabla\partial_x\widetilde{p}\|_
{L^2(\Omega_{\pm,\varepsilon})}^2
+
C\|\partial_z\widetilde{\psi}\|_{L^2(\Omega_{\pm,\varepsilon})}^2\\
\leq&~
C\|\nabla\widetilde{\psi}\|_{L^2(\Omega)}^2
+
CM_{10}\mathcal E_\infty^2. 
\end{aligned}
\end{equation}
 
{\bf Step 2. Convergence of the velocity when $\alpha=\infty$.} In the outer layer $\Omega_{\pm,\varepsilon}$, Darcy's law gives
\[
\widetilde{\Bu} = - K_\pm \left( \nabla \widetilde{p}+\widetilde{\psi} e_z
\right).
\]
Then it follows from
\eqref{convergence-9} that
\begin{equation}
\|\widetilde{\Bu}\|_
{L^2(\Omega_{\pm,\varepsilon})}
\leq C(\|\nabla \widetilde{p}\|_{L^2(\Omega_{\pm,\varepsilon})}+\|\widetilde{\psi}\|_{L^2(\Omega_{\pm,\varepsilon})}) \leq 
C\|\widetilde{\psi}\|_{L^2(\Omega)}
+
CM_7^{\frac12}
\mathcal E_\infty.
\label{convergence-12}
\end{equation}
Similarly, by \eqref{convergence-11}, one has
\begin{equation}
\|\nabla \widetilde{\Bu}\|_
{L^2(\Omega_{\pm,\varepsilon})}
\leq  C(\|\nabla^2 \widetilde{p}\|_{L^2(\Omega_{\pm,\varepsilon})}+\|\nabla\widetilde{\psi}\|_{L^2(\Omega_{\pm,\varepsilon})}) \leq 
C\|\nabla\widetilde{\psi}\|_{L^2(\Omega)}
+
CM_{10}^{\frac12}
\mathcal E_\infty.
\label{convergence-13}
\end{equation}
Then one uses Lemma \ref{lem:A2} to obtain
\begin{equation}
\label{convergence-14}
\begin{aligned} 
\|\widetilde{\Bu}\|_
{L^4(\Omega_{\pm,\varepsilon})} \leq&~C\|\nabla\widetilde{\Bu}\|_
{L^2(\Omega_{\pm,\varepsilon})}^\frac12\| \widetilde{\Bu}\|_
{L^2(\Omega_{\pm,\varepsilon})}^\frac12+C\| \widetilde{\Bu}\|_
{L^2(\Omega_{\pm,\varepsilon})} \\
\leq&~
C
\|\widetilde{\psi}\|_{L^2(\Omega)}^{\frac12}
\|\nabla\widetilde{\psi}\|_{L^2(\Omega)}^{\frac12}
+
C\|\widetilde{\psi}\|_{L^2(\Omega)} 
\\
& + C\left(
M_7^{\frac14}+M_{10}^{\frac14}
\right)
\mathcal E_\infty^{\frac12}
\|\nabla\widetilde{\psi}\|_{L^2(\Omega)}^{\frac12} 
+C\left(
M_7^{\frac12}+M_{10}^{\frac12}
\right)
\mathcal E_\infty.
\end{aligned}
\end{equation}  
In the intermediate layer $\Omega_{0,\varepsilon}$, one combines \eqref{H^1-7-1} and Proposition \ref{H^1} to obtain
\begin{equation}\label{convergence-15}
\|\Bu^\varepsilon\|_{L^2(\Omega_{0,\varepsilon})}
\leq CK_0^\frac12\|\psi^\varepsilon\|_{L^2(\Omega )}.
\end{equation}
% We divide the estimation of $\|\Bu^\varepsilon\|_{L^4(\Omega_{0,\varepsilon})}$ into two parts based on the magnitude of $\alpha$. If $\alpha\in [0,\infty)$, by \eqref{H^1-7}, one has further
% \begin{equation}
% \|\Bu^\varepsilon\|_{L^4(\Omega_{0,\varepsilon})}
% \leq CK_0^\frac14\|\psi^\varepsilon\|_{H^1(\Omega )}.
% \end{equation}
On the other hand, recall the notion used in Step 2. of the proof of Proposition \ref{H^1}. By Lemma \ref{lem:thin-layer-L2} and \eqref{H^1-7-1}-\eqref{H^1-7-2}, it holds that 
\begin{equation}\label{convergence-16}
\|\Bv^\varepsilon\|_{L^2(\Omega_{0,\varepsilon})}  \leq C\|\Bu^\varepsilon\|_{L^2(Q_{+,\varepsilon})} \leq Ch_0^\frac12\|\Bu^\varepsilon\|_{H^1(\Omega_{+,\varepsilon})}  \leq Ch_0^\frac12\|\psi^\varepsilon\|_{H^1(\Omega )}.
\end{equation}
\begin{equation}\label{convergence-17}
\|\nabla \Bv^\varepsilon\|_{L^2(\Omega_{0,\varepsilon})}  \leq C\|\nabla \Bu^\varepsilon\|_{L^2(Q_{+,\varepsilon})} \leq C \|\Bu^\varepsilon\|_{H^1(\Omega_{+,\varepsilon})}  \leq C \|\psi^\varepsilon\|_{H^1(\Omega )} 
\end{equation}
and 
\begin{equation}\label{convergence-18}
\|\Bv^\varepsilon\|_{L^4(\Omega_{0,\varepsilon})}  \leq C\|\Bu^\varepsilon\|_{L^4(Q_{+,\varepsilon})} \leq Ch_0^\frac14\|\Bu^\varepsilon\|_{H^1(\Omega_{+,\varepsilon})}  \leq Ch_0^\frac14\|\psi^\varepsilon\|_{H^1(\Omega )}.
\end{equation}
Combining \eqref{H^1-7-1}-\eqref{H^1-7-2} and \eqref{convergence-15}-\eqref{convergence-16}, we obtain
\begin{equation}\label{convergence-19}
\|\Bw^\varepsilon\|_{L^2(\Omega_{0,\varepsilon})} \leq \|\Bu^\varepsilon\|_{L^2(\Omega_{0,\varepsilon})}+\|\Bv^\varepsilon\|_{L^2(\Omega_{0,\varepsilon})}
\leq
C(h_0^\frac12+K_0^\frac12)\|\psi^\varepsilon\|_{H^1(\Omega)}
\end{equation}
and 
\begin{equation}\label{convergence-20}
\|\nabla\Bw^\varepsilon\|_{L^2(\Omega_{0,\varepsilon})} \leq \|\nabla\Bu^\varepsilon\|_{L^2(\Omega_{0,\varepsilon})}+\|\nabla\Bv^\varepsilon\|_{L^2(\Omega_{0,\varepsilon})}
\leq
C \|\psi^\varepsilon\|_{H^1(\Omega)}.
\end{equation}
Applying Lemma~\ref{lem:A2} to
$\Bw^\varepsilon$ gives
\begin{equation}\label{convergence-21}
\|\Bw^\varepsilon\|_{L^4(\Omega_{0,\varepsilon})}\leq
C
\|\Bw^\varepsilon\|_{L^2(\Omega_{0,\varepsilon})}^{\frac12}
\|\nabla \Bw^\varepsilon\|_{L^2(\Omega_{0,\varepsilon})}^{\frac12}
\leq
C(h_0^\frac14+K_0^\frac14)\|\psi^\varepsilon\|_{H^1(\Omega)}.
\end{equation}
Consequently, by \eqref{convergence-18} and \eqref{convergence-21}, one obtains
\begin{equation}\label{convergence-22}
\|\Bu^\varepsilon\|_{L^4(\Omega_{0,\varepsilon})} \leq \|\Bv^\varepsilon\|_{L^4(\Omega_{0,\varepsilon})}+\|\Bw^\varepsilon\|_{L^4(\Omega_{0,\varepsilon})}
\leq
C(h_0^\frac14+K_0^\frac14)\|\psi^\varepsilon\|_{H^1(\Omega)}.
\end{equation} 

Furthermore, using Lemma \ref{lem:thin-layer-L2} and Proposition \ref{uniform-p} gives 
\begin{equation}\label{convergence-23}
\|\Bu^0\|_{L^2(\Omega_{0,\varepsilon})} \leq Ch_0^\frac12(\|\Bu^0\|_{H^1(\Omega_+ )}+ \|\Bu^0\|_{H^1(  \Omega_-) }) \leq Ch_0^{\frac12} \|\psi^0\|_{H^1(\Omega)} 
\end{equation}
and 
\begin{equation}\label{convergence-24}
\|\Bu^0\|_{L^4(\Omega_{0,\varepsilon})} \leq Ch_0^\frac14(\|\Bu^0\|_{H^1(\Omega_+ )}+ \|\Bu^0\|_{H^1(  \Omega_-) }) \leq Ch_0^{\frac14} \|\psi^0\|_{H^1(\Omega)} .
\end{equation}

Combining \eqref{convergence-15}, \eqref{convergence-23} and using Proposition \ref{H^1}, we have
\begin{equation}\label{convergence-25}
\begin{aligned}
\|\widetilde{\Bu}\|_{L^2(\Omega_{0,\varepsilon})} \leq &~\|\Bu^\varepsilon\|_{L^2(\Omega_{0,\varepsilon})} +\|\Bu^0\|_{L^2(\Omega_{0,\varepsilon})}\\
\leq&~
C 
\left(
K_0^{\frac12}+h_0^{\frac12}
\right) (\|\psi^\varepsilon\|_{H^1(\Omega)} + \|\psi^0\|_{H^1(\Omega)} )
\leq 
CM_7^{\frac12}\mathcal E_\infty 
\end{aligned}
\end{equation}
and 
\begin{equation}\label{convergence-26}
\begin{aligned}
\|\widetilde{\Bu}\|_{L^4(\Omega_{0,\varepsilon})} \leq&~ \|\Bu^\varepsilon\|_{L^4(\Omega_{0,\varepsilon})} +\|\Bu^0\|_{L^4(\Omega_{0,\varepsilon})}\\
\leq&~
C 
\left(
K_0^{\frac14}+h_0^{\frac14}
\right) (\|\psi^\varepsilon\|_{H^1(\Omega)} + \|\psi^0\|_{H^1(\Omega)} )
\leq 
CM_7^{\frac12}\mathcal E_\infty^\frac12.
\end{aligned}
\end{equation}
These, together with \eqref{convergence-12} and \eqref{convergence-25}, yield
\begin{equation}\label{convergence-27}
\|\widetilde{\Bu}\|_
{L^2(\Omega)}
\leq  
C\|\widetilde{\psi}\|_{L^2(\Omega)}
+
CM_7^{\frac12}
\mathcal E_\infty 
\end{equation}
and 
\begin{equation}\label{convergence-28}
\begin{aligned} 
\|\widetilde{\Bu}\|_
{L^4(\Omega)}  
\leq&~
C
\|\widetilde{\psi}\|_{L^2(\Omega)}^{\frac12}
\|\nabla\widetilde{\psi}\|_{L^2(\Omega)}^{\frac12}
+
C\|\widetilde{\psi}\|_{L^2(\Omega)} \\
& 
+
C
\left(
M_7^{\frac14}+M_{10}^{\frac14}
\right)
\mathcal E_\infty^{\frac12}
\|\nabla\widetilde{\psi}\|_{L^2(\Omega)}^{\frac12}
+
C
\left(
M_7^{\frac12}+M_{10}^{\frac12}
\right)
\mathcal E_\infty^\frac12.
\end{aligned}
\end{equation}
 
{\bf Step 3. Convergence of $\psi^\varepsilon$ when $\alpha=\infty$.}  Now we establish the convergence of $\psi^\varepsilon$. The  difference $\widetilde{\psi} =\psi^\varepsilon - \psi^0$ satisfies the following equation 
\begin{equation}\label{convergence-29}
\partial_t \widetilde{\psi} + \Bu^\varepsilon\cdot\nabla \widetilde{\psi}+\widetilde{\Bu}\cdot\nabla \psi^0   +  \phi_b'\widetilde{\Bu} \cdot \Be_z   =  \nabla\cdot(D^\varepsilon\nabla \widetilde{\psi}) + \nabla\cdot(\widetilde{D} \nabla \psi^0 )+ \widetilde{D}\phi_b'',
\end{equation}
where $\widetilde{D} := D^\varepsilon - D^0$ is supported in $\Omega_{0,\varepsilon}$. 

Testing \eqref{convergence-29} by $ \widetilde{\psi}$ and using integration by parts, we obtain
\begin{equation}\label{convergence-30}
\begin{aligned}
&~\frac12\frac{\dd }{\dd t}\|\widetilde{\psi}\|_{L^2(\Omega)}^2  + \|\sqrt{D^\varepsilon} \nabla \widetilde{\psi}\|_{L^2(\Omega)}^2 \\
=&~(\widetilde{\Bu}\cdot\nabla   \widetilde{\psi}, \psi^0 ) -  (\phi_b'\widetilde{\Bu} \cdot \Be_z, \widetilde{\psi}) - (\widetilde{D} \nabla \psi^0 ,\nabla\widetilde{\psi}) +  (\widetilde{D}\phi_b'', \widetilde{\psi}).
\end{aligned}
\end{equation}

With the aid of \eqref{convergence-28} and Proposition \ref{H^1}, one applies Lemma \ref{lemmaA1} and Young's inequality to deduce
\begin{equation}\label{convergence-31}
\begin{aligned}
&~ |(  \widetilde{\Bu}\cdot\nabla  \widetilde{\psi}  ,\psi^0)| \leq \|\nabla  \widetilde{\psi}\|_{L^2(\Omega)}\|\psi^0\|_{L^4(\Omega)}  \|\widetilde{\Bu}\|_{L^4(\Omega)} \\ 
\leq&~\frac{\min_i D_i}{4} \|\nabla  \widetilde{\psi}\|_{L^2(\Omega)}^2 +  C \|\psi^0\|_{L^4(\Omega)}^2  \|\widetilde{\Bu}\|_{L^4(\Omega)}^2\\
\leq&~\frac{\min_i D_i}{4} \|\nabla  \widetilde{\psi}\|_{L^2(\Omega)}^2 +  C \|\psi^0\|_{H^1(\Omega)}^2  \|\widetilde{\Bu}\|_{L^4(\Omega)}^2\\
\leq&~\frac{\min_i D_i}{4} \|\nabla  \widetilde{\psi}\|_{L^2(\Omega)}^2  + C M_7( \|\nabla \widetilde{\psi}\|_{L^2(\Omega)}  \| \widetilde{\psi}\|_{L^2(\Omega)}   +   \|\widetilde{\psi}\|_{L^2(\Omega)} ^2  )  \\
    &~ + C M_7(   (M_7^\frac12 + M_{10}^\frac12 )  \|\nabla \widetilde{\psi}\|_{L^2(\Omega)} \mathcal{E}_\infty +  (M_7 +M_{10} )\mathcal{E}_\infty) \\
\leq&~\frac{\min_i D_i}{2} \|\nabla  \widetilde{\psi}\|_{L^2(\Omega)}^2 +  CM_7^2\| \widetilde{\psi}\|_{L^2(\Omega)}^2 +C(M_7^3+M_{10}^3)\mathcal{E}_\infty^2.
\end{aligned}
\end{equation}
Using H\"older inequality and \eqref{convergence-27}, one has
\begin{equation}\label{convergence-32}
\begin{aligned}
|( \phi_b'\widetilde{\Bu} \cdot \Be_z, \widetilde{\psi})| =&~ \left|\int_\Omega   \phi_b' \widetilde{\Bu} \cdot \Be_z  \widetilde{\psi}\dd \Bx \right| \leq~ C\|\widetilde{\Bu}\|_{L^2(\Omega)}\|\widetilde{\psi}\|_{L^2(\Omega)}\\
\leq&~C(\|\widetilde{\psi}\|_{L^2(\Omega)}  +  M_7^\frac12 \mathcal{E}_\infty)\| \widetilde{\psi}\|_{L^2(\Omega)}\\
\leq&~C \| \widetilde{\psi}\|_{L^2(\Omega)}^2+ CM_7\mathcal{E}_\infty^2 
\end{aligned}
\end{equation}
and
\begin{equation}\label{convergence-33}
\begin{aligned}
|(\widetilde{D}\phi_b'', \widetilde{\psi})| \leq& ~C\|\widetilde{D}\|_{L^2(\Omega)}\|\widetilde{\psi}\|_{L^2(\Omega)} \leq  Ch_0^\frac12\|\widetilde{\psi}\|_{L^2(\Omega)}\\
\leq&~C \| \widetilde{\psi}\|_{L^2(\Omega)}^2+ C \mathcal{E}_\infty^2 .
\end{aligned}
\end{equation}

Finally, by virtue of Lemmas \ref{lemmaA1}-\ref{lemma-equivalence}, one has
\begin{equation}\label{convergence-34}
\begin{aligned}
|( \widetilde{D} \nabla \psi^0, \nabla\widetilde{\psi} )| 
\leq&~\frac{\min_iD_i}{4} \|\nabla  \widetilde{\psi}\|_{L^2(\Omega)}^2 + C \|\widetilde{D}\|_{L^4(\Omega)}^2\|\nabla\psi^0\|_{L^4(\Omega)}^2\\
\leq&~\frac{\min_iD_i}{4} \|\nabla  \widetilde{\psi}\|_{L^2(\Omega)}^2 + C h_0^\frac{1}{2}  \|\psi^0\|_{W^0}^2\\
\leq&~\frac{\min_iD_i}{4} \|\nabla  \widetilde{\psi}\|_{L^2(\Omega)}^2 + C \mathcal{E}_\infty  \|\mathcal{L}^0\psi^0\|_{L^2(\Omega)}^2.
\end{aligned}
\end{equation}
Substituting \eqref{convergence-31}-\eqref{convergence-34} to \eqref{convergence-30} gives 
\begin{equation}\label{convergence-35}
\begin{aligned}
&~\frac{1}{2}\frac{\dd }{\dd t} \|\widetilde{\psi}\|_{L^2(\Omega)}^2 + \frac{\min_i D_i}{4}\|\nabla \widetilde{\psi}\|_{L^2(\Omega)}^2 \\
\leq&~  CM_7^2\| \widetilde{\psi}\|_{L^2(\Omega)}^2 +C(M_7^3+M_{10}^3)\mathcal{E}_\infty^2 +C \mathcal{E}_\infty  \|\mathcal{L}^0\psi^0\|_{L^2(\Omega)}^2.
\end{aligned}
\end{equation}
Since $\widetilde{\psi}(0)=0$, one utilizes  Gr\"onwall's inequality and Proposition \ref{H^1} to obtain
\begin{equation}\nonumber
\begin{aligned}
\|\widetilde{\psi}(t)\|_{L^2(\Omega)}^2 \leq&~ \int_0^t \left( C(M_7^3+M_{10}^3)\mathcal{E}_\infty^2 +C \mathcal{E}_\infty  \|\mathcal{L}^0\psi^0\|_{L^2(\Omega)}^2\right) \exp{\left(\int_\tau ^t CM_7^2\dd s\right)} \dd \tau\\ 
\leq&~  C \left((M_7^3+M_{10}^3)\mathcal{E}_\infty^2t+ \mathcal{E}_\infty\int_0^t  \|\mathcal{L}^0\psi^0(\tau )\|_{L^2(\Omega)}^2  \dd \tau\right) \exp{ (CM_7^2t)}\\
\leq&~  C \left((M_7^3+M_{10}^3)t+M_8(t)\right) \exp{ (CM_7^2t)}\mathcal{E}_\infty.
\end{aligned}
\end{equation}
This, together with \eqref{convergence-9} and \eqref{convergence-27}, gives 
\begin{equation}\label{convergence-37}
\begin{aligned}
 &~\|\widetilde{\Bu}\|_{L^2(\Omega)} + \| \widetilde{\psi}\|_{L^2(\Omega)} + \| \nabla\widetilde{p}\|_{L^2(\Omega_{\pm,\varepsilon})} \\
\leq &~C \|\widetilde{\psi}(t)\|_{L^2(\Omega)} + CM_7^\frac12 \mathcal{E}_\infty\\
\leq&~C \left[\left((M_7^3+M_{10}^3)t+M_8(t)\right)^\frac12  \exp{ (CM_7^2t)}  +M_7^\frac12 \right]\mathcal{E}_\infty^\frac12 \\ 
=&:M_{11}(t) \mathcal{E}_\infty^\frac12  .
\end{aligned}
\end{equation}
Then we prove Theorem \ref{thm:convergence} in the case $\alpha=\infty$.

{\bf Step 4. Convergence of $p^\varepsilon$ when $\alpha= 0$.}   We now consider low permeability regime, i.e., $\alpha=0$. In this case, we assume that
\begin{equation}
    h_0\to 0,\qquad K_0\to 0,\qquad
    \frac{K_0}{h_0}\to 0.
\end{equation}
Since $\mathcal E_0(\varepsilon)\to0$, we may also assume, without loss of generality, that $0<\mathcal E_0\le1$ for the remainder of this step.
In this regime, the limiting interface is impermeable, that is, $  u_z^0\big|_{\Gamma}=0$. We first estimate the pressure error $\widetilde{p}$ in $\Omega_{+,\varepsilon}$.
The argument in $\Omega_{-,\varepsilon}$ is analogous.

Recall that, in $\Omega_{+,\varepsilon}$, one has 
\begin{equation}
    \widetilde{\Bu}
    =
    -K_+ (\nabla \widetilde{p}+\widetilde{\psi}\Be_z ),
    \qquad
    \nabla\cdot \widetilde{\Bu}=0.
\end{equation}
Therefore, for any $q\in H^1(\Omega_{+,\varepsilon})$, integration by parts gives 
\begin{equation}\label{eq:alpha0-ptilde-weak}
    \int_{\Omega_{+,\varepsilon}}
    K_+ (\nabla \widetilde{p}+\widetilde{\psi} \Be_z )
    \cdot\nabla q\dd x
    =
    \int_{\Omega\cap \{z= \frac{h_0}{2}\}} 
   \widetilde{u}_z q\dd x.
\end{equation}

By incompressibility and periodicity of $\widetilde{\Bu}$ in the horizontal direction, one has
\begin{equation}
    \frac{\dd}{\dd z}
    \int_{\mathbb{T}_L}\widetilde{u}_z (x,z)\dd x
    =
    -\int_{\mathbb{T}_L}\partial_x\widetilde{u}_x (x,z)\dd x
    =
    0.
\end{equation}
Since $\widetilde{u}_z =0$ on $z=H$, it follows that
\begin{equation}\label{eq:alpha0-zero-mean-flux}
    \int_{\mathbb{T}_L}\widetilde{u}_z (x,z)\dd x=0,\quad \text{ for any } z\in  [-H,H].
\end{equation}
 
Set
\begin{equation}
    (\widetilde{p})_{+,\varepsilon}
    :=
    \frac{1}{|\Omega_{+,\varepsilon}|}
    \int_{\Omega_{+,\varepsilon}}\widetilde{p}\dd \Bx .
\end{equation}
Taking $q=\widetilde{p}-(\widetilde{p})_{+,\varepsilon}$ in \eqref{eq:alpha0-ptilde-weak} and using \eqref{eq:alpha0-zero-mean-flux}, we obtain
\begin{equation}\label{eq:alpha0-pressure-energy}
\begin{aligned}
    K_+\|\nabla \widetilde{p}\|_{L^2(\Omega_{+,\varepsilon})}^2
    = &~
    -K_+
    \int_{\Omega_{+,\varepsilon}}
    \widetilde{\psi}\,\partial_z\widetilde{p}\dd\Bx +
    \int_{\Omega \cap \{z=\frac{h_0}{2}\}}\widetilde{u}_z
    \left( \widetilde{p}  - (\widetilde{p})_{+,\varepsilon}
    \right)\dd x.
\end{aligned}
\end{equation}
For the first term on the right-hand side of
\eqref{eq:alpha0-pressure-energy}, Young's inequality gives
\begin{equation}\label{eq:alpha0-pressure-volume}
    \left|  K_+
    \int_{\Omega_{+,\varepsilon}}
    \widetilde{\psi}\,\partial_z\widetilde{p}\dd \Bx\right|
    \leq \frac{K_+}{4}
    \|\nabla \widetilde{p}\|_{L^2(\Omega_{+,\varepsilon})}^2+C\|\widetilde{\psi}\|_{L^2(\Omega)}^2.
\end{equation}

Since the height of $\Omega_{+,\varepsilon}$ is uniformly bounded from above and below for sufficiently small $\varepsilon$, it follows from the Poincar\'e inequality and the trace theorem that 
\begin{equation}\label{eq:alpha0-pressure-trace}
    \left\| \widetilde{p} -(\widetilde{p})_{+,\varepsilon}
    \right\|_{L^2(\Omega \cap \{z=\frac{h_0}{2}\})} \leq C \left\|
    \widetilde{p} 
    -
    (\widetilde{p})_{+,\varepsilon}
    \right\|_{H^1(\Omega_{+,\varepsilon})}
    \leq
    C\|\nabla \widetilde{p}\|_{L^2(\Omega_{+,\varepsilon})},
\end{equation}
where $C$ is independent of $\varepsilon$. We next estimate the normal velocity on $z=\frac{h_0}{2}$. For each fixed $x\in\mathbb{T}_L$ and any $s\in(-\frac{h_0}{2},\frac{h_0}{2})$, the fundamental
theorem of calculus gives
\begin{equation}\nonumber
    u_z^\varepsilon\left(x,\frac{h_0}{2}\right)
    =
    u_z^\varepsilon(x,s)
    +
    \int_s^{\frac{h_0}{2}}
    \partial_z u_z^\varepsilon(x,r)\dd r,
\end{equation}
and hence,
\begin{equation} \nonumber
    \left|
    u_z^\varepsilon\left(x,\frac{h_0}{2}\right)
    \right|^2
    \leq 
    2|u_z^\varepsilon(x,s)|^2
    +
    2\left|
    \int_s^{\frac{h_0}{2}}
    \partial_z u_z^\varepsilon(x,r)\dd r
    \right|^2. 
\end{equation}
Using Cauchy--Schwarz inequality, one has  
\begin{equation} \nonumber
\begin{aligned}
    \left|
    u_z^\varepsilon\left(x,\frac{h_0}{2}\right)
    \right|^2
    \leq &~2|u_z^\varepsilon(x,s)|^2
    + 2\left(\frac{h_0}{2}-s\right)
    \int_s^{\frac{h_0}{2}}
    |\partial_z u_z^\varepsilon(x,r)|^2\dd r \\
    \leq&~2|u_z^\varepsilon(x,s)|^2+ 2h_0
    \int_{-\frac{h_0}{2}}^{\frac{h_0}{2}}
    |\partial_z u_z^\varepsilon(x,r)|^2\dd r.
\end{aligned}
\end{equation}
Integrating the above inequality with respect to
$s\in(-\frac{h_0}{2},\frac{h_0}{2})$ and $x\in \mathbb{T}_L$, we obtain
\begin{equation}\nonumber
\begin{aligned}
   \left\|
    u_z^\varepsilon\left(\cdot,\frac{h_0}{2}\right)
    \right\|_{L^2(\mathbb{T}_L)}^2
    \leq&~
    \frac{2}{h_0}
    \|u_z^\varepsilon\|_{L^2(\Omega_{0,\varepsilon})}^2 +
    2h_0
    \|\partial_z u_z^\varepsilon\|_{L^2(\Omega_{0,\varepsilon})}^2.
\end{aligned}
\end{equation}

By Proposition \ref{uniform-p}, it follows that 
\begin{equation}\label{eq:alpha0-ueps-trace-final}
\begin{aligned}
   \left\|
    u_z^\varepsilon\left(\cdot,\frac{h_0}{2}\right)
    \right\|_{L^2(\mathbb{T}_L)}^2
    \leq &~
    C\frac{K_0}{h_0}
    \|\psi^\varepsilon\|_{L^2(\Omega)}^2+ 
    CK_0h_0
    \|\nabla\psi^\varepsilon\|_{L^2(\Omega)}^2.
\end{aligned}
\end{equation}
On the other hand, since $\alpha=0$, we have $ u_z^0(x,0)=0$, and hence,
\begin{equation}\nonumber
    \left|u_z^0\left(x,\frac{h_0}{2}\right)\right|^2
   = \left|
    \int_0^{\frac{h_0}{2}}
     \partial_zu_z^0(x,z) \dd z\right|^2.
\end{equation}
Integrating the equality over $x\in \mathbb{T}_L$ and using the Cauchy--Schwarz inequality, one has 
\begin{equation}\label{eq:alpha0-u0-trace}
\begin{aligned}
    \left\|
    u_z^0\left(\cdot,\frac{h_0}{2}\right)
    \right\|_{L^2(\mathbb{T}_L)}^2 \leq
    Ch_0
    \|\partial_zu_z^0\|_{L^2(\Omega_+)}^2 \leq
    Ch_0
    \|\nabla\psi^0\|_{L^2(\Omega)}^2,
\end{aligned}
\end{equation}
where Proposition \ref{uniform-p} has been used in the last inequality.

Combining \eqref{eq:alpha0-ueps-trace-final} and
\eqref{eq:alpha0-u0-trace}, we obtain
\begin{equation}\label{eq:alpha0-normal-flux-error}
\begin{aligned}
 \left\| \widetilde{u}_z \right\|_{L^2(\Omega \cap \{z=\frac{h_0}{2}\})}^2 
 \leq C\mathcal{E}_0
   ( \| \psi^\varepsilon\|_{H^1(\Omega)}^2+ 
    \| \psi^0\|_{H^1(\Omega)}^2).
\end{aligned}
\end{equation}
Furthermore, using Young's inequality and \eqref{eq:alpha0-pressure-trace}, \eqref{eq:alpha0-normal-flux-error}, we obtain
\begin{equation}\label{eq:alpha0-pressure-boundary}
\begin{aligned}
    &~\left|   \int_{\Omega \cap \{z=\frac{h_0}{2}\}}\widetilde{u}_z
    \left( \widetilde{p}  - (\widetilde{p})_{+,\varepsilon}
    \right)\dd x
    \right| \\
    \leq&~ \left\| \widetilde{u}_z \right\|_{L^2(\Omega \cap \{z=\frac{h_0}{2}\})} \left\| \widetilde{p} -(\widetilde{p})_{+,\varepsilon}
    \right\|_{L^2(\Omega \cap \{z=\frac{h_0}{2}\})} \\
  \leq&~
    \frac{K_+}{4}
    \|\nabla \widetilde{p}\|_{L^2(\Omega_{+,\varepsilon})}^2
    + C\mathcal{E}_0
   ( \| \psi^\varepsilon\|_{H^1(\Omega)}^2+ 
    \| \psi^0\|_{H^1(\Omega)}^2).
\end{aligned}
\end{equation}
Substituting \eqref{eq:alpha0-pressure-volume} and
\eqref{eq:alpha0-pressure-boundary} into
\eqref{eq:alpha0-pressure-energy} and using Proposition \ref{H^1}, we obtain
\begin{equation}\label{eq:alpha0-pressure-plus} 
    \|\nabla \widetilde{p}\|_{L^2(\Omega_{+,\varepsilon})}^2
   % \leq  C\|\widetilde{\psi}\|_{L^2(\Omega)}^2 +C\mathcal{E}_0( \| \psi^\varepsilon\|_{H^1(\Omega)}^2+ \| \psi^0\|_{H^1(\Omega)}^2)
   \leq  C\|\widetilde{\psi}\|_{L^2(\Omega)}^2
    +CM_7\mathcal{E}_0 . 
\end{equation}
Repeating the same argument in $\Omega_{-,\varepsilon}$ gives
\begin{equation}\label{eq:alpha0-pressure-minus} 
\|\nabla \widetilde{p}\|_{L^2(\Omega_{-,\varepsilon})}^2
    %\leq  C\|\widetilde{\psi}\|_{L^2(\Omega)}^2  +C\mathcal{E}_0( \| \psi^\varepsilon\|_{H^1(\Omega)}^2+ \| \psi^0\|_{H^1(\Omega)}^2) 
    \leq  C\|\widetilde{\psi}\|_{L^2(\Omega)}^2
    +CM_7\mathcal{E}_0.  
\end{equation}
A similar argument also holds for $\partial_x \Bu^\varepsilon$ and $\partial_x p^\varepsilon$ since $\Bu^\varepsilon$ and $p^\varepsilon$ are determined by $\psi^\varepsilon$ through linear equations with coefficients independent of $x$. Thus, we have also 
\begin{equation}\nonumber
\begin{aligned}
\| \nabla \partial_x \widetilde{p}\|_{L^2(\Omega_{\pm,\varepsilon} )}^2  \leq&~  C\|\partial_x\widetilde{\psi}\|_{L^2(\Omega)}^2+ C \mathcal{E}_0\left(\|\partial_x\psi^\varepsilon\|_{H^1(\Omega)}^2 + \|\partial_x\psi^0\|_{H^1(\Omega)}^2\right) . \\
 \leq&~  C\|\partial_x\widetilde{\psi}\|_{L^2(\Omega)}^2+ CM_{10} \mathcal{E}_0.
\end{aligned}
\end{equation}
Since $\widetilde{p}$ satisfies the elliptic problem 
\begin{equation}\nonumber 
-\Delta \widetilde{p} =\partial_z\widetilde{\psi},\quad \text{ in }\Omega_{\pm,\varepsilon},
\end{equation}
one obtains
\begin{equation}\label{convergence-69-0}
\|\nabla^2\widetilde{p}\|_{L^2(\Omega_{\pm,\varepsilon})}^2
\leq C\|\nabla \widetilde{\psi}\|_{L^2(\Omega)}^2+ CM_{10} \mathcal{E}_0.
\end{equation}

{\bf Step 5. Convergence of $\Bu^\varepsilon$ and $\psi^\varepsilon$ when $\alpha=0$.}
It follows from Darcy's law and \eqref{eq:alpha0-pressure-plus}-\eqref{eq:alpha0-pressure-minus}  that
  \begin{equation}\label{convergence-70}
\begin{aligned}
\| \widetilde{\Bu}\|_{L^2(\Omega_{\pm,\varepsilon})} = &~\|K_\pm(\nabla\widetilde{p} + \widetilde{\psi} \Be_z)\|_{L^2(\Omega_{\pm,\varepsilon})}\\
\leq &~C(\| \nabla\widetilde{p}\|_{L^2(\Omega_{\pm,\varepsilon})}+ \|\widetilde{\psi} \|_{L^2(\Omega_{\pm,\varepsilon})})\\
\leq&~C\|\widetilde{\psi}\|_{L^2(\Omega)}  + CM_7^\frac12 \mathcal{E}_0^\frac12.
\end{aligned}
\end{equation} 
Similarly, one also has from \eqref{convergence-69-0} that 
  \begin{equation}\label{convergence-71} 
\| \nabla \widetilde{\Bu}\|_{L^2(\Omega_{\pm,\varepsilon})}  \leq C(\| \nabla^2\widetilde{p}\|_{L^2(\Omega_{\pm,\varepsilon})}+ \|\nabla \widetilde{\psi} \|_{L^2(\Omega_{\pm,\varepsilon})})  \leq C\|\nabla \widetilde{\psi}\|_{L^2(\Omega)}  + CM_{10}^\frac12 \mathcal{E}_0^\frac12. 
\end{equation} 
By Poinc\'are inequality, Lemma \ref{lemmaA1} and \eqref{convergence-70}-\eqref{convergence-71}, we have 
\begin{equation}\label{convergence-72} 
\begin{aligned}
    \|\widetilde{\Bu}\|_{L^4(\Omega_{\pm,\varepsilon})}\leq&~  C ( \|\widetilde{\Bu}\|_{L^2(\Omega_{\pm,\varepsilon})}^\frac12\|\nabla\widetilde{\Bu}\|_{L^2(\Omega_{\pm,\varepsilon})}^\frac12 + \|\widetilde{\Bu}\|_{L^2(\Omega_{\pm,\varepsilon})}) \\
    \leq&~C  \|\nabla \widetilde{\psi}\|_{L^2(\Omega)} ^\frac12 \| \widetilde{\psi}\|_{L^2(\Omega)} ^\frac12
    + C\|\widetilde{\psi}\|_{L^2(\Omega)} \\
    &+ C\left(M_7^\frac14 + M_{10}^\frac14\right)\mathcal{E}_0^\frac14  \|\nabla \widetilde{\psi}\|_{L^2(\Omega)} ^\frac12
    + C\left(M_7^\frac12+ M_{10}^\frac12\right)\mathcal{E}_0^\frac12. 
\end{aligned}
\end{equation}

% This, together with Propositions \ref{H^1}-\ref{uniform-limit-p} and Lemma \ref{lemmaA1}, gives
% \begin{equation}\label{convergence-20-1}
% \begin{aligned}
%   \| \widetilde{\Bu}\|_{L^4(\Omega_{\pm,\varepsilon})} \leq & ~C  \| \nabla \widetilde{\Bu}\|_{L^2(\Omega_{\pm,\varepsilon})} ^\frac12 \| \widetilde{\Bu}\|_{L^2(\Omega_{\pm,\varepsilon})}^\frac12 + C\| \widetilde{\Bu}\|_{L^2(\Omega_{\pm,\varepsilon})} \\
%   \leq & ~C  (\| \nabla \Bu^\varepsilon\|_{L^2(\Omega_{\pm,\varepsilon})}+ \| \nabla \Bu^0\|_{L^2(\Omega_{\pm,\varepsilon})} )^\frac12 \| \widetilde{\Bu}\|_{L^2(\Omega_{\pm,\varepsilon})}^\frac12 + C\| \widetilde{\Bu}\|_{L^2(\Omega_{\pm,\varepsilon})} \\
%    \leq & ~C  (\| \psi^\varepsilon\|_{H^1(\Omega)}+ \| \psi^0\|_{H^1(\Omega)} )^\frac12 \| \widetilde{\Bu}\|_{L^2(\Omega_{\pm,\varepsilon})}^\frac12 + C\| \widetilde{\Bu}\|_{L^2(\Omega_{\pm,\varepsilon})} \\
%       \leq & ~C \|\widetilde{\psi}\|_{L^2(\Omega)}  + C M_6^\frac12 \|\widetilde{\psi}\|_{L^2(\Omega)} ^\frac12   +  C M_6 E^\frac12(\varepsilon) 
%    .
% \end{aligned}
% \end{equation}
On the other hand, in the thin layer $\Omega_{0,\varepsilon}$, using \eqref{H^1-7-1}, \eqref{H^1-7}, Proposition \ref{uniform-p} and Lemma \ref{lem:thin-layer-L2}, we have 
\begin{equation}\nonumber\label{convergence-73}
 \|\Bu^\varepsilon\|_{L^r(\Omega_{0,\varepsilon})} \leq  CK_0^\frac1r\| \psi^\varepsilon\|_{H^1(\Omega)} 
\end{equation}
and 
\begin{equation}\nonumber\label{convergence-74}
\begin{aligned}
\| \Bu^0\|_{L^r(\Omega_{0,\varepsilon})}  
\leq&~ \|\Bu^0\|_{L^r(\Omega_{+}\cap\Omega_{0,\varepsilon})}+ \|\Bu^0\|_{L^r(\Omega_{-}\cap \Omega_{0,\varepsilon})}  \\
\leq&~ Ch_0^\frac1r(\|\Bu^0\|_{H^1(\Omega_{+})}+ \|\Bu^0\|_{H^1(\Omega_{-} )} ) \\
\leq&~ Ch_0^\frac1r\|\psi^0\|_{H^1(\Omega)} ,
\end{aligned}
\end{equation}
for $r=2,4$. Hence, by  Propositions \ref{H^1}, one has 
\begin{equation}\label{convergence-75}
\begin{aligned}
  \|\widetilde{\Bu}\|_{L^r(\Omega_{0,\varepsilon})} \leq&~ \| \Bu^0\|_{L^r(\Omega_{0,\varepsilon})} + \|\Bu^\varepsilon\|_{L^r(\Omega_{0,\varepsilon})}    \\
  \leq&~ C(\|\psi^\varepsilon\|_{H^1(\Omega)}+\|\psi^0\|_{H^1(\Omega)})( h_0^\frac1r + K_0^\frac1r  )\\
  \leq&~ CM_7^\frac12( h_0^\frac1r + K_0^\frac1r  ),
\end{aligned}
\end{equation}
for $r =2,4$. Combining \eqref{convergence-70} and \eqref{convergence-72}-\eqref{convergence-75}, we obtain
\begin{equation}\label{convergence-76}
\|\widetilde{\Bu}\|_{L^2(\Omega)} \leq C\|\widetilde{\psi}\|_{L^2(\Omega)}  + CM_7^\frac12 \mathcal{E}_0^\frac12
\end{equation}  
and 
\begin{equation}\label{convergence-77}  
\begin{aligned}
    \|\widetilde{\Bu}\|_{L^4(\Omega)} 
    \leq&~C  \|\nabla \widetilde{\psi}\|_{L^2(\Omega)} ^\frac12 \| \widetilde{\psi}\|_{L^2(\Omega)} ^\frac12 +C\|\widetilde{\psi}\|_{L^2(\Omega)} \\
    &~+ C(M_7^\frac14 + M_{10}^\frac14 )\mathcal{E}_0^\frac14  \|\nabla \widetilde{\psi}\|_{L^2(\Omega)} ^\frac12   + C(M_7^\frac12 +M_{10}^\frac12 )\mathcal{E}_0^\frac12.
\end{aligned}
\end{equation}

With \eqref{convergence-76}--\eqref{convergence-77} at hand, we now repeat the scalar energy argument of Step~3, keeping track of the weaker powers of $\mathcal E_0$.  The difference $\widetilde\psi$ satisfies \eqref{convergence-29}.  Testing that equation by $\widetilde\psi$ and estimating the convection term with \eqref{convergence-77} gives, after Young's inequality,
\begin{equation}\nonumber\label{eq:alpha0-psi-energy} 
 \frac{\dd}{\dd t}\|\widetilde\psi\|_{L^2(\Omega)}^2
+ \|\nabla\widetilde\psi\|_{L^2(\Omega)}^2
\leq C M_7^2\|\widetilde\psi\|_{L^2(\Omega)}^2
+C(M_7^3+M_{10}^3)\mathcal E_0
+C\mathcal E_0\|\mathcal L^0\psi^0\|_{L^2(\Omega)}^2. 
\end{equation}
 Here the term $\mathcal E_0^{\frac12}\|\nabla\widetilde\psi\|_{L^2}$ arising from the square of the $\mathcal E_0^{\frac14}$ contribution in \eqref{convergence-77} has been absorbed by Young's inequality.  The terms involving $D^\varepsilon-D^0$ are treated exactly as in Step~3, using their support in the shrinking layer. Since $\widetilde\psi(0)=0$, Gr\"onwall's inequality and Proposition~\ref{H^1} yield
\[
\|\widetilde\psi(t)\|_{L^2(\Omega)}^2\leq C(t)\mathcal E_0.
\]
Combining this estimate with \eqref{eq:alpha0-pressure-plus}, \eqref{eq:alpha0-pressure-minus}, and \eqref{convergence-76} gives the asserted $O(\mathcal E_0^{\frac12})$ convergence for $\alpha=0$.

{\bf Step 6. Modified pressure for $\alpha\in (0,\infty)$.} Now we prove Theorem \ref{thm:convergence} in the case $\alpha\in (0,\infty)$. For finite $\alpha$, the limit pressure $p^0$ is only piecewise $H^1$ and may jump across $\Gamma$, whereas $p^\varepsilon\in H^1(\Omega)$. To compare the two pressure fields and their jumps across the interface $\Gamma$, a modified pressure is introduced as follows. We denote the trace of $p^\varepsilon$ on the interface $\{z=\frac{h_0}{2}\}$ by $g$ and expand $g$ in Fourier series as
\[ g(x) = \sum_{n\in \mathbb{Z}} g_n e^{\ii \frac{2\pi}{L} n x},\quad \text{ where }  \quad g_n = \frac{1}{L}\int_{0}^{L} g(x) e^{- \ii   \frac{2\pi}{L}nx} \dd x.\]
For each $n$, define 
\[ \phi_n(z) := \frac{\cosh\left( \beta (z - \frac{h_0}{4}) \right)}
                  {\cosh\left( \frac{\beta h_0}{4} \right)},
\qquad 0 \le z \le \frac{h_0}{2},\]
where $\beta := \frac{2\pi}{L} |n|$. It is easy to see that $\phi_n$ satisfies 
 \begin{equation}\label{convergence-38}
\phi_n(0) = \phi_n\left(\tfrac{h_0}{2}\right) = 1 \text{ and }\quad \phi_n'' =\beta^2\phi_n.
 \end{equation}
Now the modified pressure $\bar{p}^\varepsilon$ in $\Omega_+$ is defined by
\begin{equation}\label{convergence-39}
 \bar{p}^\varepsilon(x,z) := 
 \left\{\begin{aligned}
&p^\varepsilon  , &&\text{ in }\Omega_{+,\varepsilon}, \\
&\sum_{n\in\mathbb{Z}}   g_n \phi_n(z)  e^{ \ii \frac{2\pi}{L} n x}, &&\text{ in }\Omega_{+}\cap \Omega_{0,\varepsilon}.
\end{aligned}\right.
\end{equation}
Then it follows that  
\[\bar{p}^\varepsilon|_{z=\frac{h_0}{2}} =\bar{p}^\varepsilon|_{z=0} = g\quad \text{ and }\quad \Delta \bar{p}^\varepsilon = 0\quad \text{ in } \quad\Omega_+\cap \Omega_{0,\varepsilon}.\]
This implies $\bar{p}^\varepsilon\in H^1(\Omega_+)$. In particular, we have 
\[
\begin{aligned}
\|\bar{p}^\varepsilon\|_{H^1(\Omega_+\cap \Omega_{0,\varepsilon})}^2
\leq&~C \sum_{n\in\mathbb{Z}} |g_n|^2 
   \int_0^{\frac{h_0}{2}}  |\phi_n(z)|^2 + |\phi_n'(z)|^2 +\beta^2 |\phi_n(z)|^2   \dd z \\
   \leq&~C \sum_{n\in\mathbb{Z}} |g_n|^2 
  \left( \frac{h_0}{2}+\int_0^{\frac{h_0}{2}}  |\phi_n'(z)|^2 +\beta^2 |\phi_n(z)|^2   \dd z\right) \\
\leq&~ C\sum_{n\in\mathbb{Z}} |g_n|^2 \left( \frac{h_0}{2}+ 2\beta \tanh\!\left(\frac{\beta h_0}{4}\right)\right)\\
\leq&~Ch_0\sum_{n\in\mathbb{Z}} |g_n|^2  (1+\beta^2)\\
=&~Ch_0 \|g\|_{H^1(\mathbb{T}_L)}^2.
\end{aligned}
\]
For the zero Fourier mode, $\beta=0$ and $\phi_0\equiv1$, so the same bound is immediate; for $\beta>0$ we used $\tanh s<s$. Then it follows from the trace theorem that there exists a constant independent of $\varepsilon$ such that
\begin{equation}\label{convergence-40}
\|\bar{p}^\varepsilon\|_{H^1(\Omega_+\cap \Omega_{0,\varepsilon})} \leq Ch_0^\frac12\|g\|_{H^1(\mathbb{T}_L)} \leq C h_0^\frac12\|p^\varepsilon\|_{H^2(\Omega_{+,\varepsilon})} .
% \|\bar{p}^\varepsilon\|_{H^1(\Omega_+)} \leq \|\bar{p}^\varepsilon\|_{H^1(\Omega_{+,\varepsilon})}+ \|\bar{p}^\varepsilon\|_{H^1(\Omega_+\cap \Omega_{0,\varepsilon})} \le \|p^\varepsilon\|_{H^1(\Omega_+,\varepsilon)}+ C \|g\|_{H^\frac12(\mathbb{T}_L)}\leq C \|p^\varepsilon\|_{H^1(\Omega_+,\varepsilon)}.
\end{equation}
% Furthermore, for $\psi^\varepsilon\in H_0^1(\Omega)$, the pressure $p^\varepsilon$ is piece wise $H^2$ and the trace $g$ belongs to $H^{\frac{3}{2}}(\mathbb{T}_{L})$. Consequently, we have  
% \begin{equation}\label{convergence-41}
% \begin{aligned}
% \|\bar{p}^\varepsilon\|_{H^2(\Omega_+\cap\Omega_{0,\varepsilon})}^2 
%  =&~\sum_{n\in \mathbb{Z}} |g_n|^2 \int_0^{\frac{h_0}{2}}
%    \beta^4 |\phi_n|^2 + 2\beta^2 |\phi_n'|^2 +|\phi_n''|^2  \dd z \\
%    =&~\sum_{n\in \mathbb{Z}} |g_n|^2 \int_0^{\frac{h_0}{2}}
%    \beta^4 |\phi_n|^2 + 2\beta^2 |\phi_n'|^2  \dd z \\
% =&~ \sum_{n\in \mathbb{Z}} |g_n|^2  \left(3\beta^3 \tanh\left( \frac{\beta h_0}{4} \right) - \frac{\beta^4 h_0}{4} \operatorname{sech}^2\left( \frac{\beta h_0}{4} \right)\right)\\
% \leq &~ C \sum_{n\in \mathbb{Z}} |g_n|^2  \left( \beta^3  +  e^{-\frac{\beta h_0}{2}}  \beta^4 h_0  \right)\\
% \leq &~ C \sum_{n\in \mathbb{Z}} |g_n|^2   \beta^3\leq C\|g\|_{H^{\frac{3}{2}}(\mathbb{T}_{L})}^2 \leq C \|p^\varepsilon\|_{H^2(\Omega_+,\varepsilon)}^2 
% \end{aligned}
% \end{equation}
% and 
% \begin{equation}\label{convergence-42} 
% \|\bar{p}^\varepsilon\|_{H^2(\Omega_+)} \leq C \|p^\varepsilon\|_{H^2(\Omega_+,\varepsilon)}. 
% \end{equation}

Analogously, we can also define the modified pressure in $\Omega_-$, which satisfies
\begin{equation}\label{convergence-43}
\|\bar{p}^\varepsilon\|_{H^1(\Omega_-\cap \Omega_{0,\varepsilon})} \leq  C h_0^\frac12\|p^\varepsilon\|_{H^2(\Omega_{-,\varepsilon})} .
  % \|\bar{p}^\varepsilon\|_{H^1(\Omega_-)}\le C \|p^\varepsilon\|_{H^1(\Omega_-,\varepsilon)}\quad \text{ and }\quad\|\bar{p}^\varepsilon\|_{H^2(\Omega_-)}\le C \|p^\varepsilon\|_{H^2(\Omega_-,\varepsilon)}.
\end{equation} 
We henceforth write \(\bar p^\varepsilon=(\bar p_+^\varepsilon,\bar p_-^\varepsilon)\in H^1(\Omega_+)\times H^1(\Omega_-)\) for this broken extension; no continuity across \(\Gamma\) is imposed.

% Let us define
% \[\bar{p}^\varepsilon_{+}\left(x, z\right):= \begin{cases}
% p^\varepsilon(x,z),\quad& z\in \left(\tfrac{h_0}{2},H\right),\\
% p^\varepsilon(x,\tfrac{h_0}{2}),\quad& z\in \left(0,\tfrac{h_0}{2}\right)
% \end{cases}
% \]
% and 
% \[\bar{p}^\varepsilon_{-}\left(x, z\right):= \begin{cases}
% p^\varepsilon(x,z),\quad& z\in \left(-H, -\tfrac{h_0}{2}\right),\\
% p^\varepsilon(x,-\tfrac{h_0}{2}),\quad& z\in \left(-\tfrac{h_0}{2},0\right).
% \end{cases}
% \]
% It is easy to check that $\bar{p}^\varepsilon_\pm\in H^1(\Omega_\pm)$ with 
% \begin{equation}\nonumber
% \begin{aligned}
%  \|\bar{p}^\varepsilon_\pm\|_{H^1(\Omega_\pm)} \leq&~ \|p^\varepsilon \|_{H^1(\Omega_{\pm,\varepsilon})} + Ch_0^\frac12\|\bar{p}^\varepsilon_\pm\|_{H^1(\Gamma)} \leq C\|p^\varepsilon \|_{H^\frac32(\Omega_{\pm,\varepsilon})} \leq C\|p^\varepsilon \|_{H^2(\Omega_{\pm,\varepsilon})} .
% \end{aligned}
% \end{equation} 
  
{\bf Step 7. Convergence of $(\Bu^\varepsilon,p^\varepsilon,\psi^\varepsilon)$ when $\alpha\in (0,\infty)$.} In the rest of the proof, we denote 
\[\widetilde{p}:=\bar p^\varepsilon -p^0-c_\varepsilon,\]
where
\[
c_\varepsilon
:=
\frac{1}{|\Omega|}\int_{\Omega}\bar p^\varepsilon -p^0\dd \Bx.
\]
Then we have 
\[\int_{\Omega}\widetilde{p} \dd \Bx=0. \]

Now we are in a position to estimate $\widetilde{p}$. For the limit problem, we have 
\begin{equation}\label{convergence-44}
\int_{\Omega_\pm}K_\pm(\nabla p^0+\psi^0 \Be_z) \cdot\nabla w^\pm\dd \Bx \pm \alpha  \int_{\Gamma} \llbracket p^0 \rrbracket _\Gamma  w^\pm  \dd \sigma =0, 
\end{equation}
for any $w^\pm\in H^1(\Omega_\pm)$. Moreover, by extension theorem, there exists a function $W^\pm\in H^1(\Omega)$ such that $W^\pm=w^\pm$ in $\Omega_{\pm}$ and
\begin{equation}\label{convergence-45}
\|W^\pm\|_{H^1(\Omega)} \leq C\|w^\pm\|_{H^1(\Omega_{\pm})},
\end{equation}
where the constant $C$ is independent of $w^\pm$. By introducing the truncation function
\begin{equation}\label{convergence-46}
\zeta_+(x,z) =
\begin{cases}
1, & z \ge \frac{h_0}{2},\\ 
\dfrac{2z+h_0}{2h_0}, & -\frac{h_0}{2} < z < \frac{h_0}{2},\\
0,&z< -\frac{h_0}{2},
\end{cases}
\end{equation}
and taking the testing function $q= \zeta_+ W^+\in H^1(\Omega)$ in the weak formula \eqref{weakform-p} gives 
\begin{equation}\label{convergence-47}
\int_{\Omega_{+,\varepsilon}} K_+ (\nabla p^\varepsilon + \psi^\varepsilon \Be_z) \cdot \nabla w^+ \dd \Bx + \frac{K_0}{h_0} \int_{0}^L  \llbracket p^\varepsilon \rrbracket  w^+(x,0)  \dd x=I_+,
\end{equation}
where $\llbracket p^\varepsilon \rrbracket(x) = p^\varepsilon(x,\tfrac{h_0}{2}) - p^\varepsilon(x,-\tfrac{h_0}{2})$ and 
\begin{equation}\label{convergence-48}
\begin{aligned}
I_+ =&~ \int_{\Omega_{0,\varepsilon}} \Bu^\varepsilon \cdot \nabla W^+ \zeta_+ \dd \Bx  -\frac{K_0}{h_0}   \int_{\Omega_{0,\varepsilon}}    \psi^\varepsilon  W^+(x,0) \dd \Bx \\
&~+ \frac{1}{h_0} \int_{\Omega_{0,\varepsilon}} u_z^\varepsilon (W^+(x,z) -W^+ (x,0))\dd \Bx .
\end{aligned}
\end{equation}
We claim that there exists a constant $C$ independent of $\varepsilon$ such that
  \begin{equation}\label{convergence-49}
    |I_+ |
    \le C K_0^\frac12 \|\psi^\varepsilon\|_{H^1(\Omega)}  
       \|w^+\|_{H^1(\Omega_+)} .
  \end{equation}
To see this, it suffices to estimate all the terms in \eqref{convergence-48}. Using Lemma \ref{lem:thin-layer-L2} and trace theorem, one has 
  \begin{equation}\label{convergence-50}
    \begin{aligned}
      \frac{K_0}{h_0}\left|\int_{\Omega_{0,\varepsilon}}\psi^\varepsilon W^+(x,0)\dd \Bx\right|
      \leq&~C\frac{K_0}{h_0}h_0^\frac12 \|\psi^\varepsilon\|_{L^2(\Omega_{0,\varepsilon})}
        \|W^+(\cdot,0)\|_{L^2(\Gamma)}\\
       \leq&~C K_0 \|\psi^\varepsilon\|_{H^1(\Omega)}\|W^+\|_{H^1(\Omega)}\\
      \leq&~C K_0 \|\psi^\varepsilon\|_{H^1(\Omega)}\|w^+\|_{H^1(\Omega_+)} .
    \end{aligned}
    \end{equation} 
By H\"older's inequality, Proposition \ref{uniform-p} and Lemma \ref{lem:thin-layer-L2}, one has 
    \begin{equation}\label{convergence-51}
    \begin{aligned}
      \left|\int_{\Omega_{0,\varepsilon}} \Bu^\varepsilon\cdot\nabla W^+ \zeta_+\dd \Bx\right|
      \leq&~\|\Bu^\varepsilon\|_{L^2(\Omega_{0,\varepsilon})}\|\nabla W^+\|_{L^2(\Omega_{0,\varepsilon})} \\
      %  \leq&~C  h_0^\frac12 \|\Bu^\varepsilon\|_{H^1(\Omega_{0,\varepsilon})} \|\nabla W^+\|_{L^2(\Omega)}\\
       \leq&~C K_0^\frac12   \|\psi^\varepsilon\|_{L^2(\Omega)} \|W^+\|_{H^1(\Omega)}\\
       \leq&~C K_0^\frac12  \|\psi^\varepsilon\|_{L^2(\Omega)} \|w^+\|_{H^1(\Omega_{+})}.
    \end{aligned}
    \end{equation}
 Similarly, noting $\nabla\zeta_+ = (0,\tfrac{1}{h_0})$ on $\Omega_{0,\varepsilon}$, one uses Poinc\'are inequality, Proposition \ref{uniform-p} and Lemma \ref{lem:thin-layer-L2} to obtain
    \begin{equation}\label{convergence-52}
\begin{aligned}
 &~ \left|  \int_{\Omega_{0,\varepsilon}} \Bu^\varepsilon \cdot\nabla \zeta_+ (W^+(x,z) -W^+(x,0))\dd \Bx  \right|\\
 \leq&~\frac{1}{h_0}\|u_z^\varepsilon\|_{L^2(\Omega_{0,\varepsilon})} 
  \|W^+(x,z) -W^+ (x,0) \|_{L^2(\Omega_{0,\varepsilon})} \\
  \leq&~C\|u_z^\varepsilon\|_{L^2(\Omega_{0,\varepsilon})} 
  \|\nabla W^+ \|_{L^2(\Omega_{0,\varepsilon})} \\
% \leq&~C\|\Bu^\varepsilon\|_{L^2(\Omega_{0,\varepsilon})} \|\nabla W^+\|_{L^2(\Omega_{0,\varepsilon})} \\
% \leq&~C h_0^\frac12\|\Bu^\varepsilon\|_{H^1(\Omega_{0,\varepsilon})}  \|\nabla W^+\|_{L^2(\Omega)} \\
\leq&~C  K_0^\frac12   \|\psi^\varepsilon\|_{L^2(\Omega)} \|w^+\|_{H^1(\Omega_{+})}.
\end{aligned}
\end{equation}
Then we obtain \eqref{convergence-49}.  Analogously, using the truncation function 
\begin{equation}\label{convergence-53}
\zeta_-(x,z) =
\begin{cases}
1, & z \leq  -\frac{h_0}{2},\\ 
\dfrac{-2z+h_0}{2h_0}, & -\frac{h_0}{2} < z < \frac{h_0}{2},\\
0,&z> \frac{h_0}{2},
\end{cases}
\end{equation} 
one can also show that 
\begin{equation}\label{convergence-54}
\int_{\Omega_{-,\varepsilon}} K_- (\nabla p^\varepsilon + \psi^\varepsilon \Be_z) \cdot \nabla w^- \dd \Bx -\frac{K_0}{h_0} \int_{0}^L  \llbracket p^\varepsilon \rrbracket  w^-(x,0)  \dd x=I_-, 
\end{equation}
for any  $w^-\in H^1(\Omega_-)$, where 
\begin{equation}\label{convergence-55}
\begin{aligned}
I_- =&~ \int_{\Omega_{0,\varepsilon}} \Bu^\varepsilon \cdot \nabla W^- \zeta_- \dd \Bx + \frac{K_0}{h_0}   \int_{\Omega_{0,\varepsilon}}    \psi^\varepsilon  W^-(x,0) \dd \Bx \\
&~- \frac{1}{h_0} \int_{\Omega_{0,\varepsilon}} u_z^\varepsilon (W^-(x,z) -W^- (x,0))\dd \Bx 
\end{aligned}
\end{equation}
satisfies the similar estimate as \eqref{convergence-49}, i.e., there exists a constant $C$ independent of $\varepsilon$ such that  
\begin{equation}\label{convergence-56}
|I_-| \le C K_0^\frac12\|\psi^\varepsilon\|_{H^1(\Omega)}  \|w^-\|_{H^1(\Omega_-)} .
\end{equation}

With the modified pressure $\bar{p}^\varepsilon$, the weak formulations \eqref{convergence-47} and \eqref{convergence-54} can be rewritten as 
\begin{equation}\label{convergence-57}
\begin{aligned}
&~\int_{\Omega_{\pm}} K_\pm (\nabla \bar{p}^\varepsilon + \psi^\varepsilon \Be_z) \cdot \nabla w ^\pm \dd \Bx  \pm \frac{K_0}{h_0} \int_{\Gamma}  \llbracket \bar{p}^\varepsilon \rrbracket_{\Gamma} w ^\pm  \dd \sigma \\
=&~ I_\pm + \int_{\Omega_{\pm} \cap \Omega_{0,\varepsilon}} K_\pm (\nabla \bar{p}^\varepsilon+ \psi^\varepsilon \Be_z) \cdot \nabla w ^\pm \dd \Bx. 
\end{aligned}
\end{equation}
Subtracting \eqref{convergence-44} from \eqref{convergence-57} gives 
\begin{equation}\label{convergence-58}
\begin{aligned}
    &~\int_{\Omega_{\pm}} K_\pm (\nabla \widetilde{p}+ \widetilde{\psi} \Be_z) \cdot \nabla w ^\pm \dd \Bx \pm  \alpha \int_{\Gamma}  \llbracket \widetilde{p}\rrbracket_{\Gamma} w ^\pm  \dd \sigma  \\
  = &~I_\pm \mp \left(\frac{K_0}{h_0} -\alpha \right)\int_{\Gamma}  \llbracket \bar{p}^\varepsilon \rrbracket_{\Gamma} w ^\pm  \dd \sigma  + \int_{\Omega_{\pm} \cap \Omega_{0,\varepsilon}} K_\pm (\nabla \bar{p}^\varepsilon+ \psi^\varepsilon \Be_z) \cdot \nabla w ^\pm \dd \Bx.
\end{aligned}
\end{equation}
% where 
% \begin{equation}\label{convergence-59}
% \widetilde{p}:=\bar{p}^\varepsilon -p^0\quad \text{ and }\quad\widetilde{\psi}:=\psi^\varepsilon-\psi^0
% \end{equation}

Taking the test function $w^\pm=\widetilde{p}|_{\Omega_\pm}$, respectively, and summing the result equations, we obtain 
\begin{equation}\label{convergence-60}
\begin{aligned}
&~\int_{\Omega_+\cup \Omega_-}K^0|\nabla \widetilde{p}|^2\dd \Bx  + \alpha \int_{\Gamma}  \llbracket \widetilde{p}\rrbracket_{\Gamma}^2 \dd \sigma  \\
= &~-  \int_{\Omega_+\cup \Omega_-}K^0\widetilde{\psi}\partial_z \widetilde{p}  \dd \Bx 
- \left(\frac{K_0}{h_0} -\alpha \right)\int_{\Gamma}  \llbracket \bar{p}^\varepsilon \rrbracket_{\Gamma} \llbracket \widetilde{p} \rrbracket_{\Gamma} \dd \sigma  \\
&~
 + \int_{ (\Omega_+\cap \Omega_{0,\varepsilon}) \cup (\Omega_-\cap \Omega_{0,\varepsilon}) } K^0 (\nabla \bar{p}^\varepsilon+ \psi^\varepsilon \Be_z) \cdot \nabla \widetilde{p}\dd \Bx + I_+ +I_-.
\end{aligned}
\end{equation}
% By Lemma \ref{lem:generalized-Poincare}, we have
% \begin{equation}\nonumber
% \int_{\Omega_+\cup \Omega_-}K^0|\nabla \widetilde{p}|^2\dd \Bx  + \alpha \int_{\Gamma}  \llbracket \widetilde{p}\rrbracket_{\Gamma}^2 \dd \sigma \ge \eta \|\widetilde{p}\|_{H^1(\Omega_+\cup \Omega_-)}^2,
% \end{equation}
% for some $\eta>0$. 
Now we estimate the terms on the right-hand side of \eqref{convergence-60}.
Using Cauchy-Schwarz inequality gives
\begin{equation}\label{convergence-61}
  -  \int_{\Omega}K^0\widetilde{\psi}\partial_z \widetilde{p}  \dd \Bx  \leq  \int_{\Omega}K^0|\widetilde{\psi}|^2\dd \Bx+ \frac1 4 \int_{\Omega_+\cup \Omega_-} K^0|\nabla \widetilde{p}|^2\dd \Bx.
\end{equation}
Using the trace theorem and Proposition \ref{uniform-p}, we have
\begin{equation}\label{convergence-62}
\begin{aligned}
&~-\left(\frac{K_0}{h_0} -\alpha \right)\int_{\Gamma}  \llbracket \bar{p}^\varepsilon \rrbracket_{\Gamma} \llbracket \widetilde{p} \rrbracket_{\Gamma} \dd \sigma \\
% \leq &~\frac{\alpha}{2}\int_{\Gamma}  \llbracket \widetilde{p}\rrbracket_{\Gamma}^2 \dd \sigma   + C\left|\frac{K_0}{h_0} -\alpha \right|^2 \|\llbracket \bar{p}^\varepsilon \rrbracket_{\Gamma} \|_{L^2(\Gamma)}^2\\
% \leq &~\frac{\alpha}{4}\int_{\Gamma}  \llbracket \widetilde{p}\rrbracket_{\Gamma}^2 \dd \sigma    + C\left|\frac{K_0}{h_0} -\alpha \right|^2  (\|p^\varepsilon(x,\tfrac{h_0}{2}) \|^2_{L^2(\mathbb{T}_L)} +\| p^\varepsilon(x,-\tfrac{h_0}{2})\|_{L^2(\mathbb{T}_L)}^2 ) \\
\leq &~\frac{\alpha}{4}\int_{\Gamma}  \llbracket \widetilde{p}\rrbracket_{\Gamma}^2 \dd \sigma    + C\left|\frac{K_0}{h_0} -\alpha \right|^2  (\|p^\varepsilon \|^2_{H^1(\Omega_{+,\varepsilon})} +\| p^\varepsilon \|_{H^1(\Omega_{-,\varepsilon})}^2 )\\
\leq&~\frac{\alpha}{4}\int_{\Gamma}  \llbracket \widetilde{p}\rrbracket_{\Gamma}^2 \dd \sigma    + C\left|\frac{K_0}{h_0} -\alpha \right|^2 \|\psi^\varepsilon\|_{L^2(\Omega)}^2  .
\end{aligned}
\end{equation}
By Lemma \ref{lem:thin-layer-L2}, Propositions \ref{uniform-p} and \eqref{convergence-40}-\eqref{convergence-43}, one has 
\begin{equation}\label{convergence-63}
\begin{aligned}
&~\int_{\Omega_\pm\cap \Omega_{0,\varepsilon}} K^0 (\nabla \bar{p}^\varepsilon+ \psi^\varepsilon \Be_z) \cdot \nabla \widetilde{p}\dd \Bx\\
\leq&~
\frac{1}{4}\int_{\Omega_+\cup \Omega_-} K^0|\nabla \widetilde{p}|^2\dd \Bx +C\int_{\Omega_\pm\cap \Omega_{0,\varepsilon}}|\nabla \bar{p}^\varepsilon|^2+ |\psi^\varepsilon|^2 \dd \Bx \\
% \leq&~
% \frac{1}{4}\int_{\Omega_+\cup \Omega_-} K^0|\nabla \widetilde{p}|^2\dd \Bx +Ch_0\|p^\varepsilon\|_{\Omega_{\pm,\varepsilon}}\|\psi^\varepsilon\|_{H^1(\Omega)} \\
\leq&~
\frac{1}{4}\int_{\Omega_+\cup \Omega_-} K^0|\nabla \widetilde{p}|^2\dd \Bx +Ch_0(\|p^\varepsilon\|_{H^2(\Omega_{\pm,\varepsilon})}^2+\|\psi^\varepsilon\|_{H^1(\Omega)}^2)\\
\leq&~
\frac{1}{4}\int_{\Omega_+\cup \Omega_-} K^0|\nabla \widetilde{p}|^2\dd \Bx +Ch_0\|\psi^\varepsilon\|_{H^1(\Omega)}^2.
\end{aligned}
\end{equation} 
Furthermore,  using \eqref{convergence-49}, \eqref{convergence-56}, Cauchy-Schwarz inequality and Lemma \ref{lem:generalized-Poincare}, we have
\begin{equation}\label{convergence-64}
\begin{aligned}
|I_\pm|  \leq  &~C K_0^\frac12    \|\psi^\varepsilon\|_{H^1(\Omega)}  
\|\widetilde{p}\|_{H^1(\Omega_\pm)} \\
       \leq&~C K_0^\frac12\|\psi^\varepsilon\|_{H^1(\Omega)}\left(\int_{\Omega_+\cup \Omega_-}K^0|\nabla \widetilde{p}|^2\dd \Bx  + \alpha \int_{\Gamma}  \llbracket \widetilde{p}\rrbracket_{\Gamma}^2 \dd \sigma  \right)^\frac12\\
\leq&~\frac18\left(\int_{\Omega_+\cup \Omega_-}K^0|\nabla \widetilde{p}|^2\dd \Bx  + \alpha \int_{\Gamma}  \llbracket \widetilde{p}\rrbracket_{\Gamma}^2 \dd \sigma\right)+ CK_0\|\psi^\varepsilon\|_{H^1(\Omega)}^2.
\end{aligned}
\end{equation}
Substituting \eqref{convergence-61}-\eqref{convergence-64} into \eqref{convergence-60} and using Proposition \ref{H^1} give
\begin{equation}\label{convergence-65}
\begin{aligned}
  \| \nabla \widetilde{p}\|_{L^2(\Omega_\pm  )}^2  + \alpha \int_{\Gamma}  \llbracket \widetilde{p}\rrbracket_{\Gamma}^2 \dd \sigma    \leq &~  C\|\widetilde{\psi}\|_{L^2(\Omega)}^2+ C \mathcal{E}_\alpha^2\|\psi^\varepsilon\|_{H^1(\Omega)}^2  \\
  \leq&~  C\|\widetilde{\psi}\|_{L^2(\Omega)}^2+ CM_7 \mathcal{E}_\alpha^2.
\end{aligned}
\end{equation}
% where 
% \begin{equation}\label{convergence-66}
% \mathcal{E}_\alpha := h_0^\frac12 + K_0^\frac12+\left|\frac{K_0}{h_0} -\alpha \right|.
% \end{equation} 
A similar argument also holds for $\partial_x \Bu^\varepsilon$ and $\partial_x p^\varepsilon$ since $\Bu^\varepsilon$ and $p^\varepsilon$ are determined by $\psi^\varepsilon$ through linear equations with coefficients independent of $x$. This gives 
\begin{equation}\label{convergence-67}
\begin{aligned}
\| \nabla \partial_x \widetilde{p}\|_{L^2(\Omega_\pm  )}^2  + \alpha \int_{\Gamma}  \llbracket \partial_x \widetilde{p}\rrbracket_{\Gamma}^2 \dd \sigma \leq&~  C\|\partial_x\widetilde{\psi}\|_{L^2(\Omega)}^2+ C \mathcal{E}_\alpha^2\|\partial_x\psi^\varepsilon\|_{H^1(\Omega)}^2   \\
 \leq&~  C\|\partial_x\widetilde{\psi}\|_{L^2(\Omega)}^2+ CM_{10} \mathcal{E}_\alpha^2.
\end{aligned}
\end{equation}
% Furthermore, by Propositions \ref{uniform-p}-\ref{uniform-limit-p} , Lemma \ref{lemmaA1} and \eqref{convergence-19}, we have
% \begin{equation}
% \begin{aligned}
% \|\nabla \widetilde{p}\|_{L^4(\Omega_{\pm,\varepsilon})} \leq &~C\|\nabla \widetilde{p}\|_{L^2(\Omega_{\pm,\varepsilon})}^\frac12 \|\widetilde{p}\|_{H^2(\Omega_{\pm,\varepsilon})}^\frac12 \\
% \leq &~C\|\nabla \widetilde{p}\|_{L^2(\Omega_{\pm,\varepsilon})}^\frac12 (\|p^\varepsilon\|_{H^2(\Omega_{\pm,\varepsilon})}^\frac12 + \|p^0\|_{H^2(\Omega_{\pm,\varepsilon})}^\frac12 )\\
% \leq &~C\|\nabla \widetilde{p}\|_{L^2(\Omega_{\pm,\varepsilon})}^\frac12 (\|\psi^\varepsilon\|_{H^1(\Omega)}^\frac12 + \|\psi^0\|_{H^1(\Omega)}^\frac12 )\\
% \leq &~C \|\widetilde{\psi}\|_{L^2(\Omega)} (\|\psi^\varepsilon\|_{H^1(\Omega)}^\frac12 + \|\psi^0\|_{H^1(\Omega)}^\frac12 ) +\mathcal{E}_\alpha(\|\psi^\varepsilon\|_{H^1(\Omega)}  + \|\psi^0\|_{H^1(\Omega)} ) .
% \end{aligned}
% \end{equation}
Since $\widetilde{p}$ satisfies the elliptic problem 
\begin{equation}\label{convergence-68}
-\Delta \widetilde{p} =\partial_z\widetilde{\psi},\quad \text{ in }\Omega_{\pm,\varepsilon},
\end{equation}
one has further
\begin{equation}\label{convergence-69}
\begin{aligned} \|\nabla^2\widetilde{p}\|_{L^2(\Omega_{\pm,\varepsilon})}^2 
 \leq&~C\|\nabla \partial_x\widetilde{p}\|_{L^2(\Omega_{\pm,\varepsilon})}^2+\|\partial_z \widetilde{\psi}\|_{L^2(\Omega_{\pm,\varepsilon})}^2\\
 \leq&~C\|\nabla \widetilde{\psi}\|_{L^2(\Omega)}^2+ CM_{10} \mathcal{E}_\alpha^2.
\end{aligned}
\end{equation} 

Finally, using \eqref{convergence-65}-\eqref{convergence-69} instead of \eqref{eq:alpha0-pressure-plus}-\eqref{convergence-69-0}, we can follow the argument in Step 5 to obtain the convergence of $\Bu^\varepsilon$ and $\psi^\varepsilon$. Then we complete the proof of Theorem \ref{thm:convergence}.

\end{proof}

\subsection{Convergence of global attractors}

We now investigate the convergence of the solutions of the Darcy-Boussinesq system \eqref{Sharps}--\eqref{bc} in the long-time regime. For each $\varepsilon\in[0,1]$, the solution operator $S_\varepsilon(t):L^2(\Omega)\to L^2(\Omega)$ is defined as 
\[
S_\varepsilon(t)\psi_{\mathrm{in}}^\varepsilon
:=
\psi^\varepsilon(t),
\]
where $\psi^\varepsilon$ is the solution to
problem \eqref{Sharps}--\eqref{bc} when $\varepsilon>0$, and to the limit
problem in \eqref{Sharps-limit}, \eqref{initial-limit}-\eqref{interface1-limit} when $\varepsilon=0$. By Propositions \ref{prop1} and \ref{prop1-limit}, $\{S_\varepsilon(t)\}_{t\geq0}$ is a continuous semigroup on $L^2(\Omega)$ for every $\varepsilon\in[0,1]$.

To study the global attractors associated with the semigroup $\{S_\varepsilon\}_{t\ge 0}$, we introduce the higher-regularity space
\[
\mathcal{X}
:=
\left\{
\psi\in H_0^1(\Omega):\partial_x\psi\in H_0^1(\Omega)
\right\},
\]
endowed with the norm
\[
\|\psi\|_{\mathcal{X}}^2
:=
\|\psi\|_{H^1}^2
+
\|\partial_x\psi\|_{H^1}^2.
\]
The following theorem establishes global attractors in the phase space $L^2(\Omega)$ and a bound for them that is uniform in the higher-regularity space $\mathcal X$.
\begin{theorem}
\label{thm:global-attractor-convergence}
Assume \eqref{alpha}. For every $\varepsilon\in[0,1]$, the continuous semigroup
$\{S_\varepsilon(t)\}_{t\geq0}$ possesses a global attractor
$\mathcal{A}_\varepsilon$, which is uniformly bounded in $\mathcal{X}$ and compact  in $L^2(\Omega)$. Furthermore, as $\varepsilon$ goes to $0^+$, the attractor $\mathcal{A}_\varepsilon$ converges to  $\mathcal{A}_0$ in the following sense,
\begin{equation}\nonumber
d(\mathcal{A}_{\varepsilon}, \mathcal{A}_0) \to  0,
\end{equation}
where $d(\mathcal{A}_{\varepsilon}, \mathcal{A}_0)$ is the Hausdorff semi-distance between $\mathcal{A}_{\varepsilon}, \mathcal{A}_0$ defined by 
\begin{equation}\nonumber
d(\mathcal{A}_{\varepsilon}, \mathcal{A}_0) := \sup_{f\in \mathcal{A}_{\varepsilon}} \operatorname{dist} (f,\mathcal{A}_0)= \sup_{f\in \mathcal{A}_{\varepsilon}} \inf_{g\in \mathcal{A}_0}\|f-g\|_{L^2(\Omega)}.
\end{equation}

% Moreover, the family of attractors is uniformly bounded in $\mathcal{X}$
% \[
% \sup_{\varepsilon\in[0,1]}
% \sup_{\psi\in\mathcal{A}_\varepsilon}
% \left(
% \|\psi\|_{H^1}
% +
% \|\partial_x\psi\|_{H^1}
% \right)
% <\infty.
% \]
% In particular,
% \[
% \bigcup_{\varepsilon\in[0,1]}\mathcal{A}_\varepsilon
% \]
% is relatively compact in $H$.

% Furthermore, assume that
% \[
% h_0(\varepsilon)\to0,\qquad
% K_0(\varepsilon)\to0,\qquad
% \frac{K_0(\varepsilon)}{h_0(\varepsilon)}
% \to\alpha\in(0,\infty)
% \]
% as $\varepsilon\to0$. Then the family of global attractors is upper
% semicontinuous at $\varepsilon=0$, namely,
% \[
% \operatorname{dist}_H
% \left(
% \mathcal{A}_\varepsilon,\mathcal{A}_0
% \right)
% \longrightarrow0
% \qquad\text{as }\varepsilon\to0.
% \]
\end{theorem}

\begin{proof} 
  We first prove the existence of global attractors $\mathcal{A}_\varepsilon$ for the semigroup $\{S_\varepsilon(t)\}_{t\geq0}$.
It follows from Proposition \ref{L^2} that, for every bounded set
$B\subset L^2(\Omega)$, there exists a time $T_B >0$, depending only on the
$L^2(\Omega)$-bound of $B$, such that
\[
\|S_\varepsilon(t)\psi_{\mathrm{in}}\|_{L^2(\Omega)}^2
\leq 
\frac{M_1H^2}{\min_i D_i}+1
\]
for any $\psi_{\mathrm{in}}\in B$,   $t\geq T_B $, and $\varepsilon\in[0,1]$.

By Propositions \ref{H^1} and \ref{H^1-x}, when $t\geq T_B+2$, one also
has
\[
\|\nabla S_\varepsilon(t)\psi_{\mathrm{in}}\|_{L^2(\Omega)}^2\leq \frac{M_5}{\min_i D_i} \quad \text{ and } \quad \|\nabla\partial_x S_\varepsilon(t)\psi_{\mathrm{in}}\|_{L^2(\Omega)}^2
\leq \frac{M_9}{\min_i D_i} ,
\]
for any $\psi_{\mathrm{in}}\in B$ and $\varepsilon\in[0,1]$.
 
Define
\[
\mathcal{B}_*
:=
\left\{
\psi\in\mathcal{X}:
\begin{array}{l}
\displaystyle
\|\psi\|_{L^2(\Omega)}^2\leq \frac{M_1H^2}{\min_i D_i}+1,\\
\displaystyle
\| \nabla\psi\|_{L^2(\Omega)}^2
\leq \frac{M_5}{\min_i D_i},\\ 
\displaystyle
\|\partial_x\nabla \psi\|_{L^2(\Omega)}^2
\leq  \frac{M_9}{\min_i D_i}
\end{array}
\right\}.
\]
Then  $\mathcal{B}_*$ is a common absorbing set for
the semigroup $\{S_\varepsilon(t)\}_{t\ge 0}$ for any $\varepsilon\in [0,1]$. More precisely, for
every bounded set $B\subset L^2(\Omega)$, there exists $T_B>0$, independent of
$\varepsilon$, such that
\[
S_\varepsilon(t)B\subset\mathcal{B}_*
\]
for every $t\geq T_B+2$ and every $\varepsilon\in[0,1]$.

Since $\mathcal{B}_*$ is bounded in $\mathcal{X}$, it is in particular
bounded in $H_0^1(\Omega)$. The compact embedding $H_0^1(\Omega)\Subset L^2(\Omega)$ therefore implies that $\mathcal{B}_*$ is relatively compact in $L^2(\Omega)$. It follows from the standard theory of dissipative dynamical systems that
each semigroup $S_\varepsilon(t)$ possesses a global attractor
$\mathcal{A}_\varepsilon$ in $L^2(\Omega)$; see, for instance,
\cite{Temam1997}. The attractor $\mathcal{A}_\varepsilon$ is compact in $L^2(\Omega)$, strictly invariant,
\[
S_\varepsilon(t)\mathcal{A}_\varepsilon
=
\mathcal{A}_\varepsilon,
\qquad t\geq0,
\]
and attracts every bounded subset of $L^2(\Omega)$. Since $\mathcal{B}_*$ absorbs every bounded set and
$\mathcal{A}_\varepsilon$ is invariant, we also have $\mathcal{A}_\varepsilon\subset\mathcal{B}_*$ for every $\varepsilon\in[0,1]$. Hence, 
\[
\sup_{\varepsilon\in[0,1]}
\sup_{\psi\in\mathcal{A}_\varepsilon}
\|\psi\|_{\mathcal{X}}
<\infty.
\]

To show the convergence, we assume for contradiction that the assertion is false. Then there exists a  constant $c >0$ and a sequence $\{\varepsilon_k\}$ converging to $0^+$ such that
\[ d(\mathcal{A}_{\varepsilon_k},\mathcal{A}_0) \geq c , \qquad\text{for all } k \in \mathbb{Z}_+.
\]
Due to the compactness of $\mathcal{A}_{\varepsilon_k}$ in $L^2(\Omega)$, there exists an $a_k \in \mathcal{A}_{\varepsilon_k}$ such that
\[\operatorname{dist} (a_k,\mathcal{A}_0) =d(\mathcal{A}_{\varepsilon_k},\mathcal{A}_0) \geq c.\]
On the other hand, since $\mathcal{A}_0$ attracts all bounded sets in $L^2(\Omega)$, we choose $T > 0$ to be sufficiently large such that
\[d(S_0(T) \mathcal{B}_*  ,\mathcal{A}_0)<\frac{c}{4}.\]
Since the set $\mathcal{A}_\varepsilon$ is invariant with respect to $S_\varepsilon(t)$ for any $t\ge 0$, i.e., $S_\varepsilon (t) \mathcal{A}_\varepsilon= \mathcal{A}_\varepsilon$, then for each $a_k \in \mathcal{A}_{\varepsilon_k}$, there exists $b_k \in \mathcal{A}_{\varepsilon_k}$ such that $S_{\varepsilon_k}(T)b_k= a_k$ for all $k$. This implies that $b_k\in \mathcal{A}_{\varepsilon_k} \subset  \mathcal{B}_*$ for all $k$. By Theorem \ref{thm:convergence}, we have
\begin{equation}\nonumber
\begin{aligned}
  \|a_k -S_0(T)b_k\|_{L^2(\Omega)} \leq&~ \|S_{\varepsilon_k}(T)b_k-S_0(T)b_k\|_{L^2(\Omega)}\\
  \leq&~ \mathcal{E}_\alpha^\frac12  M_{11}\left(T,\mathcal{B}_*\right)\\
  \leq&~\frac{c}{4},
\end{aligned}
\end{equation}
provided $k$ is sufficiently large. Hence,
\begin{equation}\nonumber
\begin{aligned}
  \operatorname{dist}(a_k, \mathcal{A}_0) \leq &~\|a_k -S_0(T)b_k\|_{L^2(\Omega)} + \operatorname{dist}(S_0(T)b_k, \mathcal{A}_0)\\
  \leq&~  \|a_k -S_0(T)b_k\|_{L^2(\Omega)} + d(S_0(T) \mathcal{B}_*, \mathcal{A}_0) \\
  \leq&~ \frac{c}{2}.
\end{aligned}
\end{equation}
This contradiction completes the proof of this theorem.
\end{proof}

\section{Concluding remarks and outlook}\label{sec:conclusion}

We have investigated a simultaneous geometric and constitutive singular limit for the nonlinear Darcy--Boussinesq system in a layered porous medium.  When the thickness $h_0$ and permeability $K_0$ of the intermediate layer vanish together, the ratio $\frac{K_0}{h_0}$ selects one of three effective hydraulic interface laws: an impermeable interface, a finite-resistance transmission law, or perfect hydraulic contact.  The analysis provides a rigorous validation of these reduced continuum models on two complementary time scales.  On every finite time interval we obtain quantitative convergence of the scalar field, velocity, and outer pressure gradients.  At the same time, estimates uniform in the vanishing parameters yield a common dissipative structure and allow the passage to the infinite-time limit through upper semicontinuity of the global attractors.

Several ingredients are used repeatedly in both parts of the analysis.  The finite-resistance regime requires a pressure extension across the collapsing layer and a generalized Poincar\'e inequality adapted to piecewise $H^1$ pressures with jumps.  The impermeable regime requires direct control of the normal flux when the jump penalty degenerates, whereas the perfect-contact regime requires thin-strip estimates whose constants remain uniform as the layer collapses.  These estimates are also what make the long-time theory uniform with respect to the singular parameters, rather than merely providing convergence on bounded time intervals.

The flat-layer geometry used here is mainly a device that keeps the presentation transparent.  We expect that the arguments can be extended to smoothly curved interfaces under a uniform geometric nondegeneracy assumption.  In particular, the nonvanishing layers should retain thicknesses bounded uniformly above and below, while the two boundaries of the collapsing layer may be curved provided their local separation is comparable to a single small scale; schematically, if $h_\varepsilon(x')$ denotes the local layer thickness, one assumes
\[
        c h_0(\varepsilon)\le h_\varepsilon(x')\le C h_0(\varepsilon)
\]
with constants $c,C>0$ independent of $\varepsilon$, together with uniform smoothness of the coordinate charts used to flatten the interfaces.  Under such assumptions, the present energy estimates, pressure-extension arguments, and thin-strip inequalities should persist after a uniformly controlled change of variables.  By contrast, if different portions of a curved thin layer collapse at genuinely different rates, new local degeneracies can arise and uniform estimates may depend sensitively on the geometry of the narrow region.  Closely related phenomena are well known for high-contrast elliptic problems in narrow regions; see, for example, the geometry-dependent gradient estimates of Dong, Li, and Yang~\cite{DongLiYang2025}.  Treating such nonuniformly collapsing layers for the nonlinear Darcy--Boussinesq system would require additional ideas and is left for future work.

A second natural direction is to allow the diffusivity $D_0$ in the thin layer to vanish simultaneously with $h_0$ and $K_0$.  In the present paper $D_0$ remains nondegenerate, which enforces continuity of the scalar field in the limit.  If $D_0\to0$, the additional distinguished ratio $\frac{D_0}{h_0}$ is expected to measure an effective thermal or solutal conductance of the collapsing layer, in much the same way that $\frac{K_0}{h_0}$ measures its hydraulic conductance.  The combined limit may therefore generate a two-parameter family of interface models in which both pressure and temperature/concentration can exhibit finite jumps.  A rigorous treatment would require weighted estimates that remain uniform as both permeability and diffusivity degenerate, as well as a reformulation of the dissipative phase space suitable for the resulting broken scalar field.  Establishing finite-time convergence and, especially, uniform long-time validity for this more singular coupled limit is an interesting problem for future investigation.

\appendix
\section{Some basic inequalities} 
For any $\varepsilon\in [0,1]$, we define the weighted Sobolev space
\begin{equation}\label{def:W-epsilon}
  \mathcal{W}^\varepsilon = \{\psi \in H_0^1(\Omega):~\partial_x \psi \in H^1(\Omega),~D^\varepsilon\partial_z \psi \in H^1(\Omega)\}
\end{equation}
endowed with norm
\begin{equation}\label{def:Wnorm-epsilon}
  \|\psi\|_{\mathcal{W}^\varepsilon} = \|\psi\|_{H^1(\Omega)}+\|\partial_x \psi\|_{H^1(\Omega)}+\|D^\varepsilon\partial_z \psi\|_{H^1(\Omega)}.
\end{equation}
Since $D^\varepsilon>0$ is piecewise constant and uniformly bounded from above and below, one could establish the Gagliardo-Nirenberg type inequality for the weighted space $\mathcal{W}^\varepsilon$. The following lemma is given in \cite{CNW2025}.
\begin{lemma}\label{lemmaA1}
  It holds that
    \begin{equation}\label{A1-0}
\|\psi\|_{L^4(\Omega)} \leq C\|\nabla\psi\|_{L^2(\Omega)}^\frac12\|\psi\|_{L^2(\Omega)}^\frac12, \quad \text{ for any }\psi\in H_0^1(\Omega).
    \end{equation}
    Furthermore, we have
    \begin{equation}\label{A1-1}
\|\nabla \psi\|_{L^4(\Omega)} \leq C\|\nabla\psi\|_{L^2(\Omega)}^\frac12\|\psi\|_{\mathcal{W}^\varepsilon}^\frac12, \quad\text{ for any }\psi\in \mathcal{W}^\varepsilon.
    \end{equation}
    Here the constant  $C$ depends only on $L,H$ and $D_i$.
  \end{lemma}

  We also recall the equivalence between $\|\cdot\|_{\mathcal{W}^\varepsilon}$ and the norm associated with the operator 
\begin{equation}\label{def:L}
\mathcal{L}^\varepsilon\psi:= -\nabla\cdot (D^\varepsilon\nabla\psi).
\end{equation} 
  \begin{lemma}\label{lemma-equivalence}
    There exist two constants $C_l, C_u>0$ depending only on $H,L$ and $D_i$ such that
    \begin{equation} \label{A4-0-1}
C_l\|\mathcal{L}^\varepsilon\psi\|_{L^2(\Omega)} \leq \|\psi\|_{\mathcal{W}^\varepsilon} \leq C_u\|\mathcal{L}^\varepsilon  \psi\|_{L^2(\Omega)},\quad \text{ for any }\psi \in \mathcal{W}^\varepsilon.
    \end{equation}
  \end{lemma}
One may refer to \cite{CNW2025} for the detailed proof.

The following lemma is a uniform Ladyzhenskaya inequality on a thin strip.
\begin{lemma}
\label{lem:A2}
Let
\[
E_h:=\mathbb{T}_L\times\left(-\frac{h}{2},\frac{h}{2}\right),
\qquad 0<h\leq 1.
\]
Then the following assertions hold.
\begin{enumerate}
  \item For every $\psi\in H^1(E_h)$, one has 
  \begin{equation} \label{A2-1}
     \|\psi\|_{L^4(E_h)} \leq  Ch^{-\frac14}\left( \|\nabla\psi\|_{L^2(E_h)}^{\frac12} \|\psi\|_{L^2(E_h)}^{\frac12}+  \|\psi\|_{L^2(E_h)} \right).
  \end{equation}
  \item Suppose that $\psi\in H^1(E_h)$ and $\psi\big|_{z=\frac{h}{2}}=0 
$. Then one has 
\begin{equation}\label{eq:uniform-thin-ladyzhenskaya}
\|\psi\|_{L^4(E_h)} \leq C
\|\psi\|_{L^2(E_h)}^\frac12
\|\nabla \psi\|_{L^2(E_h)}^\frac12,
\end{equation}

\end{enumerate}
Here the constant $C>0$ depends only on $L$ and is independent
of $h$.
\end{lemma}

\begin{proof}
  The estimate \eqref{A2-1} follows from the standard Gagliardo-Nirenberg inequality on a bounded domain and a scaling argument. To prove \eqref{eq:uniform-thin-ladyzhenskaya}, we define
\[
\overline \psi(x,z)
:=\left\{
\begin{aligned}
&\psi(x,z), && -\frac{h}{2}<z<\frac{h}{2},\\
&\psi(x,-h-z), && -\frac{3h}{2}<z<-\frac{h}{2}.
\end{aligned}\right.
\]
Since $\psi|_{z=\frac{h}{2}}=0$, the reflected function $\overline{\psi}$ satisfies
\[
\overline \psi\big|_{z=-\frac{3h}{2}}
=
\overline \psi\big|_{z=\frac{h}{2}}
=0.
\]
We may therefore extend $\overline \psi$ by zero outside $\mathbb{T}_L\times(-\frac{3h}{2},\frac{h}{2})$ and regard it as an element of $H^1(\mathbb{T}_L\times\mathbb{R})$. The standard Gagliardo--Nirenberg inequality on
$\mathbb{T}_L\times\mathbb{R}$ gives
\[
\|\overline \psi\|_{L^4(\mathbb{T}_L\times\mathbb{R})}^2
\leq
C
\|\overline \psi\|_{L^2(\mathbb{T}_L\times\mathbb{R})}
\left(
\|\nabla\overline \psi\|_{L^2(\mathbb{T}_L\times\mathbb{R})}
+
\|\overline \psi\|_{L^2(\mathbb{T}_L\times\mathbb{R})}
\right).
\]
On the other hand, Poincar\'e inequality yields
\[
\|\overline \psi\|_{L^2(\mathbb{T}_L\times\mathbb{R})}
\leq
Ch
\|\partial_z\overline \psi\|_{L^2(\mathbb{T}_L\times\mathbb{R})}.
\]
Since $h\leq1$, we obtain
\[
\|\overline \psi\|_{L^4(\mathbb{T}_L\times\mathbb{R})}^2
\leq
C
\|\overline \psi\|_{L^2(\mathbb{T}_L\times\mathbb{R})}
\|\nabla\overline \psi \|_{L^2(\mathbb{T}_L\times\mathbb{R})}.
\]
Finally, since $\overline{\psi}$ is symmetric with respect to $\{z=-\frac{h}{2}\}$, one has
\[
\|\overline \psi\|_{L^2(\mathbb{T}_L\times\mathbb{R})}^2= 2\|\psi\|_{L^2(E_h)}^2,\quad \|\overline \psi\|_{L^4(\mathbb{T}_L\times\mathbb{R})}^4= 2\|\psi\|_{L^4(E_h)}^4 
\]
and
\[ 
\|\nabla\overline \psi\|_{L^2(\mathbb{T}_L\times\mathbb{R})}^2= 2\|\nabla \psi\|_{L^2(E_h)}^2.
\]
This proves \eqref{eq:uniform-thin-ladyzhenskaya}.
\end{proof}

The following lemma is frequently used to estimate the decay rate of $L^r$ norms on the thin layer $\Omega_{0,\varepsilon}\cap \Omega_{\pm}$.
% \begin{lemma}\label{lem:thin-layer-L2}
% Suppose that $f\in H^1(\Omega)$. Then it holds that 
% \begin{equation}\label{eq:thin-layer-estimate-H1}
% \|f\|_{L^2(\Omega_{0,\varepsilon})}
% \le Ch_0^\frac{1}{2}\|f\|_{H^1(\Omega)},
% \end{equation}
% where $C>0$ is independent of $h_0$.   
% \end{lemma}

% \begin{proof}
% For any fixed $x\in(0,L)$ and any $z,t\in(-H,H)$, we write
% \[
% |f(x,z)|^2=\left|f(x,t)+\int_t^z \partial_z f(x,s)\dd s\right|^2 \leq C |f(x,t)|^2 + C\int_{-H}^H |\partial_z f(x,s)|^2\dd s .
% \]
% Integrating the equality with respect to $t\in(-H,H)$ and using Cauchy-Schwarz inequality, we obtain
% \[
% |f(x,z)|^2
% \le C\int_{-H}^H |f(x,t)|^2\dd t
% + C\int_{-H}^H |\partial_z f(x,s)|^2\dd s.
% \]
% Integrating the inequality over $z\in(-\frac{h_0}{2},\frac{h_0}{2})$ and then $x\in(0,L)$ yields
% \[
% \|f\|_{L^2(\Omega_{0,\varepsilon})}^2
% \le
% Ch_0(\|f\|_{L^2(\Omega)}^2 + \|\partial_z f\|_{L^2(\Omega)}^2),
% \]
% which implies  \eqref{eq:thin-layer-estimate-H1}. 
% \end{proof}

\begin{lemma}\label{lem:thin-layer-L2}
For any $r\in [1,\infty)$, there exists a constant $C>0$ depending only on $r,L$ and $H$ such that
\begin{equation}\label{eq:thin-layer-estimate-L4}
\|f\|_{L^r(\Omega_\pm \cap \Omega_{0,\varepsilon})} \le Ch_0^\frac1r\|f\|_{H^1(\Omega_\pm)},
\qquad \text{ for any } f\in H^1(\Omega_\pm).
\end{equation}
\end{lemma}
\begin{proof}
Suppose that $f\in H^1(\Omega_+)$. 
%Extend $f$ by a bounded extension operator to a fixed strip containing $\Omega_+$; applying the trace theorem on translated horizontal slices gives a constant independent of the slice location.  Together with 
By virtue of the one-dimensional Sobolev embedding $H^{\frac12}(\mathbb{T}_L)\hookrightarrow L^r(\mathbb{T}_L)$ and the trace theorem, it holds that
\[
\|f(\cdot,t)\|_{L^r(\mathbb{T}_L)}  \le C \|f(\cdot,t)\|_{H^\frac12(\mathbb{T}_L)}\leq C \|f\|_{H^1(\Omega_+)},\quad \text{ for any }t\in (0,H).
\] 
Finally, integrate over the inequality over $t\in(0,\tfrac{h_0}{2})$:
\begin{align*}
\|f\|_{L^r(\Omega_+ \cap \Omega_{0,\varepsilon})}^r
= \int_0^{\tfrac{h_0}{2}} \|f(\cdot,t)\|_{L^r(\mathbb{T}_L)}^r\dd t \leq  Ch_0 \|f\|_{H^1(\Omega_+)}^r.
\end{align*}
Taking the $r$-th root completes the proof.
\end{proof}

Finally, we present a generalized Poincar\'e inequality for functions defined on two disjoint domains with a common interface. 
\begin{lemma} 
\label{lem:generalized-Poincare}
%Let $\Omega_\pm$ be two disjoint open subsets of $\Omega$ with a common interface $\Gamma:=\overline{\Omega_+}\cap \overline{\Omega_-}$. 
For any $p=(p_+,p_-)\in  H^1(\Omega_+)\times H^1(\Omega_-)$ satisfying
\[\int_{\Omega_+} p_+ \dd \Bx + \int_{\Omega_-} p_- \dd \Bx = 0,\]
there exists a constant $C>0$, depending only on
$\Omega_{+}$ and $\Omega_{-}$, such that
\begin{equation}\nonumber
 \|p_{+}\|_{L^{2}(\Omega_{+})} 
+
\|p_{-}\|_{L^{2}(\Omega_{-})} 
\leq
C
\left(
\|\nabla p_{+}\|_{L^{2}(\Omega_{+})}
+
\|\nabla p_{-}\|_{L^{2}(\Omega_{-})}
+
\left|\int_{\Gamma}\llbracket p\rrbracket_{\Gamma}\dd \sigma\right|
\right) 
\end{equation}
 where
\[
\llbracket p\rrbracket_{\Gamma}
:=
p_{+}|_{\Gamma}-p_{-}|_{\Gamma}.
\] 
\end{lemma}

\begin{proof}
For $i=+,-$, set 
\[
\bar p_i
:=
\frac{1}{|\Omega_{i}|}
\int_{\Omega_{i}}p_{i}\dd \Bx
\quad  \text{ and }\quad 
\tilde p_{i}:=p_{i}-\bar p_i.
\]
Then 
\[
\int_{\Omega_{i}}\tilde p_{i}\dd \Bx=0.
\]
By the usual Poincar\'e inequality, one has 
\begin{equation}
\|\tilde p_{i}\|_{H^{1}(\Omega_{i})}
\leq
C\|\nabla p_{i}\|_{L^{2}(\Omega_{i})}.
\label{eq:Poincare-zero-mean-parts}
\end{equation}

On the interface $\Gamma$, we have
\[
\llbracket p\rrbracket_{\Gamma}
=
\bar p_+-\bar p_-
+
\tilde p_{+}|_{\Gamma}
-
\tilde p_{-}|_{\Gamma}.
\]
Integrating this identity over $\Gamma$ yields
\[
|\Gamma|(\bar p_+-\bar p_-)
=
\int_{\Gamma}\llbracket p\rrbracket_{\Gamma}\dd \sigma
-
\int_{\Gamma}
\left(
\tilde p_{+}-\tilde p_{-}
\right)\dd \sigma.
\]
Therefore, by the Cauchy--Schwarz inequality, the trace theorem, and \eqref{eq:Poincare-zero-mean-parts}, one has 
\begin{equation}\label{eq:difference-of-means}
\begin{aligned}
|\bar p_+-\bar p_-| \leq& C\left|\int_{\Gamma}\llbracket p\rrbracket_{\Gamma}\dd \sigma\right|+ 
C \left( \|\tilde p_{+}\|_{L^2(\Gamma)} + \|\tilde p_{-}\|_{L^2(\Gamma)}
\right) \\
\leq& C\left|\int_{\Gamma}\llbracket p\rrbracket_{\Gamma}\dd \sigma\right|+ 
C \left( \|\nabla p_{+}\|_{L^{2}(\Omega_{+})} + \|\nabla p_{-}\|_{L^{2}(\Omega_{-})}
\right).
\end{aligned}
\end{equation}

On the other hand, the assumption
\[
\int_{\Omega_+} p_+ \dd \Bx + \int_{\Omega_-} p_- \dd \Bx = 0,
\]
implies that
\[
|\Omega_{+}| \bar p_+
+
|\Omega_{-}| \bar p_-
=0.
\]
Then we have
\[
\bar p_+ = \frac{|\Omega_{-}|}
{|\Omega_{+}|+|\Omega_{-}|}
(\bar p_+-\bar p_-)
\]
and
\[
\bar p_- = -\frac{|\Omega_{+}|}{|\Omega_{+}|+|\Omega_{-}|}(\bar p_+-\bar p_-  ).
\]
It follows from \eqref{eq:difference-of-means} that
\begin{align}
|\bar p_+|+|\bar p_-|
\leq
C
\left( \left|\int_{\Gamma}\llbracket p\rrbracket_{\Gamma}\dd \sigma\right|
+
\|\nabla p_{+}\|_{L^{2}(\Omega_{+})} + 
\|\nabla p_{-}\|_{L^{2}(\Omega_{-})}
\right).
\label{eq:estimates-of-means}
\end{align}
Finally, combining with \eqref{eq:Poincare-zero-mean-parts} and
\eqref{eq:estimates-of-means}, we obtain
\begin{equation}\nonumber
\begin{aligned}
 \|p_{+}\|_{L^{2}(\Omega_{+})}
+
\|p_{-}\|_{L^{2}(\Omega_{-})} \leq&  \|\tilde{p}_{+}\|_{L^{2}(\Omega_{+})}
+
\|\tilde{p}_{-}\|_{L^{2}(\Omega_{-})}+ \|\bar{p}_{+}\|_{L^{2}(\Omega_{+})}
+
\|\bar{p}_{-}\|_{L^{2}(\Omega_{-})}\\
\leq&~
\|\tilde p_{+}\|_{L^{2}(\Omega_{+})}
+
\|\tilde p_{-}\|_{L^{2}(\Omega_{-})}
+
|\Omega_{+}|^{\frac12}|\bar p_{+}|
+
|\Omega_{-}|^{\frac12}|\bar p_{-}|
\\
 \leq&~
C
\left(
\|\nabla p_{+}\|_{L^{2}(\Omega_{+})}
+
\|\nabla p_{-}\|_{L^{2}(\Omega_{-})}
+\left|\int_{\Gamma}\llbracket p\rrbracket_{\Gamma}\dd \sigma\right|
\right).
\end{aligned}
\end{equation}
Then we finish the proof.
\end{proof}

\input{references_m3as.tex}

\end{document}

%% file: references_m3as.tex
% M3AS-style references: numbered, alphabetized, full article titles, abbreviated journal names.

%% file: Effective_IBC_DB_261029.bbl
\begin{thebibliography}{99}
\bibitem{Allaire1989Stokes}
G. Allaire, Homogenization of the {Stokes} flow in a connected porous medium, \textit{Asymptot. Anal.} 2(3) (1989) 203--222.

\bibitem{AllairePankratovaPiatnitski2012}
G. Allaire, I. Pankratova, and A. Piatnitski, Homogenization of a nonstationary convection--diffusion equation in a thin rod and in a layer, \textit{SeMA} 58 (2012) 53--95.

\bibitem{AmazianePankratovPiatnitski2007}
B. Amaziane, L. Pankratov, and A. Piatnitski, Homogenization of a single phase flow through a porous medium in a thin layer, \textit{Math. Models Methods Appl. Sci.} 17(9) (2007) 1317--1349.

\bibitem{LiZhang2011}
J. Li  and K. Zhang, Reinforcement of the Poisson equation by a thin layer,  \textit{Math. Models Methods Appl. Sci.} 21(05) (2011), 1153--1192.
 

\bibitem{AngotGoyeauOchoa2017}
P. Angot, B. Goyeau, and J. A. Ochoa-Tapia, Asymptotic modeling of transport phenomena at the interface between a fluid and a porous layer: Jump conditions, \textit{Phys. Rev. E} 95(6) (2017) 063302.

\bibitem{AuvrayVial2018}
A. Auvray and G. Vial, Asymptotic expansions and effective boundary conditions: a short review for smooth and nonsmooth geometries with thin layers,  \textit{ESAIM Proc. Surv. 61 (2018) 38--54}.

\bibitem{bickle2007modelling}
M. J. Bickle, A. Chadwick, H. E. Huppert, M. A. Hallworth, and S. Lyle, Modelling carbon dioxide accumulation at Sleipner: Implications for underground carbon storage, \textit{Earth Planet. Sci. Lett.} 255(1-2) (2007) 164--176.

\bibitem{CNW2025}
Y. Cao, W. Niu, and X. Wang, Well-posedness and regularity of the Darcy--Boussinesq system in layered porous media, \textit{Discrete Contin. Dyn. Syst. Ser. B} 30(12) (2025) 4764--4783.

\bibitem{CardoneJagerWoukeng2024}
G. Cardone, W. J\"ager, and J. L. Woukeng, Derivation and analysis of a nonlocal {Hele--Shaw--Cahn--Hilliard} system for flow in thin heterogeneous layers, \textit{Math. Models Methods Appl. Sci.} 34(7) (2024) 1343--1400.

\bibitem{ChenPondWang2012}
X. Chen, C. Pond, and X. Wang, Effective boundary conditions resulting from anisotropic and optimally aligned coatings: the two dimensional case, \textit{Arch. Ration. Mech. Anal.} 206(3) (2012) 911--951.

\bibitem{doering1998jfm}
C.~R. Doering and P. Constantin, Bounds for heat transport in a porous layer, \textit{J. Fluid Mech.} 376 (1998) 263--296.

\bibitem{DongLiYang2025}
H. Dong, Y. Y. Li, and Z. Yang, Optimal gradient estimates of solutions to the insulated conductivity problem in dimension greater than two, \textit{J. Eur. Math. Soc.} 27(8) (2025) 3275--3296.

\bibitem{DW2025}
H. Dong and X. Wang, Sharp material interface limit of the Darcy--Boussinesq system, \textit{SIAM J. Appl. Math.} 85(4) (2025) 1621--1642.

\bibitem{fabrie1986aam}
P. Fabrie, Solutions fortes et comportement asymptotique pour un mod\`ele de convection naturelle en milieu poreux, \textit{Acta Appl. Math.} 7(1) (1986) 49--77.

\bibitem{fabrie1996aa}
P. Fabrie and B. Nicolaenko, Exponential attractors for non-dissipative systems modeling convection in porous media, \textit{Asymptot. Anal.} 12(4) (1996) 295--327.

\bibitem{GahnNeussRadu2025}
M. Gahn and M. Neuss-Radu, Effective interface laws for fluid flow and solute transport through thin reactive porous layers, \textit{J. Evol. Equ.} 25(2) (2025) 33.

\bibitem{hewitt2014jfm}
D.~R. Hewitt, J.~A. Neufeld, and J.~R. Lister, High Rayleigh number convection in a porous medium containing a thin low-permeability layer, \textit{J. Fluid Mech.} 756 (2014) 844--869.

\bibitem{hewitt2020jfm}
D.~R. Hewitt, G.~G. Peng, and J.~R. Lister, Buoyancy-driven plumes in a layered porous medium, \textit{J. Fluid Mech.} 883 (2020) A37.

\bibitem{hewitt2022jfm}
D.~R. Hewitt, Evolution of convection in a layered porous medium, \textit{J. Fluid Mech.} 941 (2022) A56.

\bibitem{huppert2014ar}
H.~E. Huppert and J.~A. Neufeld, The fluid mechanics of carbon dioxide sequestration, \textit{Annu. Rev. Fluid Mech.} 46 (2014) 255--272.

\bibitem{ly1999jns}
H. V. Ly and E. S. Titi, Global Gevrey regularity for the B\'enard convection in a porous medium with zero Darcy--Prandtl number, \textit{J. Nonlinear Sci.} 9 (1999) 333--362.

\bibitem{MarusicPazin2026}
E. Maru\v{s}i\'{c}-Paloka and I. Pa\v{z}anin, Fluid flow through a porous membrane and the {Darcy} interface condition, \textit{Multiscale Model. Simul.} 24(2) (2026) 798--819.

\bibitem{MikelicWheeler2012}
A. Mikeli\'{c} and M.~F. Wheeler, On the interface law between a deformable porous medium containing a viscous fluid and an elastic body, \textit{Math. Models Methods Appl. Sci.} 22(11) (2012) 1250031.

\bibitem{nield2017convection}
D.~A. Nield and A. Bejan, \textit{Convection in Porous Media}, 5th ed. (Springer, Cham, 2017).

\bibitem{oliver2000gevrey}
M. Oliver and E.~S. Titi, Gevrey regularity for the attractor of a partially dissipative model of B\'enard convection in a porous medium, \textit{J. Differ. Equ.} 163(2) (2000) 292--311.

\bibitem{Roubicek2017Geophysical}
T. Roub\'{i}\v{c}ek, Geophysical models of heat and fluid flow in damageable poro-elastic continua, \textit{Continuum Mech. Thermodyn.} 29(2) (2017) 625--646.

\bibitem{ShaWangPermeability2026}
K. Sha and X. Wang, Vanishing permeability limit of convection in multilayer porous media, \textit{Discrete Contin. Dyn. Syst.} 55 (2026) 421--441.

\bibitem{ShaWangThickness2026}
K. Sha and X. Wang, Vanishing layer thickness limit of convection in multilayer porous media, \textit{Discrete Contin. Dyn. Syst. Ser. S} 26 (2026) 461--476.

\bibitem{ShaWangWu2026}
K. Sha, X. Wang, and H. Wu, Asymptotic long-time behavior of {Darcy--Boussinesq} convection in layered porous media with narrow transition zones, preprint (2026), arXiv:2603.01362.

\bibitem{Temam1997}
R. Temam, \textit{Infinite-Dimensional Dynamical Systems in Mechanics and Physics}, 2nd ed. (Springer-Verlag, New York, 1997).

\end{thebibliography}
